\documentclass[a4paper,11pt,reqno]{amsart}
\usepackage{amsmath}
\usepackage{amssymb}
\usepackage{amsthm}
\usepackage[all]{xy}

\newtheorem{lemm}{Lemma}[section]
\newtheorem{thm}[lemm]{Theorem}
\newtheorem{cor}[lemm]{Corollary}
\newtheorem{defn}[lemm]{Definition}
\newtheorem{prop}[lemm]{Proposition}
\newtheorem{rem}[lemm]{Remark}

\newcommand{\ac}{\alpha_s^\vee}
\newcommand{\act}[2][s]{{}^{#1}\!#2}
\newcommand{\ad}{{\rm ad}\,}
\newcommand{\Ann}{{\rm Ann}\,}

\newcommand{\bx}{\mathbf x}
\newcommand{\bq}{\mathbf q}
\newcommand{\bp}{\mathbf p}
\newcommand{\C}{\mathbb C}
\newcommand{\ch}[1]{#1^\vee}
\newcommand{\cH}{\mathcal H}
\newcommand{\cI}{\mathcal I}
\newcommand{\cM}{\mathcal M}
\newcommand{\cN}{\mathcal N}
\newcommand{\cO}{\mathcal O}
\newcommand{\cco}{{c,\psi}}
\newcommand{\co}{{c,\omega}}
\newcommand{\cT}{\mathcal T}
\newcommand{\coh}{{\rm-mod}_{coh}}
\newcommand{\D}{\mathcal D}
\newcommand{\Der}{{\rm Der}}
\newcommand{\End}{{\rm End}}
\newcommand{\eu}{\mathbf{eu}}

\newcommand{\h}{\mathfrak h}
\newcommand{\hbr}{\bar\h_{\rm reg}}
\newcommand{\hO}{\hat{\mathcal O}}
\newcommand{\Hom}{{\rm Hom}\,}

\newcommand{\hr}{\h_{\rm reg}}
\newcommand{\hwp}[1][W']{\h^{#1}\!}

\newcommand{\im}{{\rm Im}\,}
\newcommand{\isom}{\stackrel{{}_\sim}{\rightarrow}}
\newcommand{\len}{{\rm len}}
\newcommand{\m}{\mathfrak m}
\newcommand{\Mat}{{\rm Mat}}
\newcommand{\ozt}{\overline{\Z[[t]]}}
\newcommand{\p}{\mathfrak p}

\newcommand{\R}{\mathbb R}
\newcommand{\re}{{\rm Re}}
\newcommand{\rs}{{\rm-mod}_{RS}}
\newcommand{\Sqp}{S_q\times S_p}
\newcommand{\Spec}{{\rm Spec}\,}
\newcommand{\Spf}{{\rm Spf}\,}
\newcommand{\stab}{{\rm Stab}}
\newcommand{\supp}{{\rm Supp}\,}
\newcommand{\tr}{{\rm tr}}
\newcommand{\wfd}[1][]{W#1{\rm-mod}_{fd}}
\newcommand{\Z}{\mathbb Z}

\begin{document}

\begin{titlepage}
\title[Representations of the rational {C}herednik algebra]{Supports of representations of the rational {C}herednik algebra of type {A}}
\author{Stewart Wilcox}
\address{Stewart Wilcox \\ Department of Mathematics \\ Harvard University \\ 1 Oxford Street \\ Cambridge MA 02318 \\ USA}
\email{swilcox@fas.harvard.edu}
\footnotetext{Date: May, 2010.}
\end{titlepage}

\begin{abstract}
We first consider the rational Cherednik algebra corresponding to the action of a finite group on a complex variety, as defined by Etingof. We define a category of representations of this algebra which is analogous 
to ``category $\cO$" for the rational Cherednik algebra of a vector space. We generalise to this setting Bezrukavnikov and Etingof's results about the possible support sets of such representations. Then we focus on the case 
of $S_n$ acting on $\C^n$, determining which irreducible modules in this category have which support sets. We also show that the category of representations with a given support, modulo those with smaller support, is equivalent to 
the category of finite dimensional representations of a certain Hecke algebra.
\end{abstract}
\maketitle

\section{Introduction}
\subsection{Linear actions}
Let $W$ be a finite group acting faithfully on a finite dimensional $\C$-vector space $\h$. The Weyl algebra $D(\h)$ of $\h$ admits an action of $W$, so $\C[W]\otimes_\C D(\h)$ becomes an algebra in a 
natural way. We denote this algebra by $\C[W]\ltimes D(\h)$. The \emph{rational Cherednik algebra}, defined by Etingof and Ginzburg \cite{EG}, is a universal flat deformation of this algebra. It is named thus because 
it is a degeneration of the \emph{double affine Hecke algebra} defined by Cherednik \cite{Cherednik}. We recall the definition of the rational Cherednik algebra below:
\begin{defn}\label{defnHc}
We define the set of \emph{reflections} in $W$ to be
\[
	S=\{s\in W\mid{\rm rk}(s-1)=1\}.
\]
For $s\in S$, let $\ac\in\h$ and $\alpha_s\in\h^*$ be the nontrivial eigenvectors of $s$, with eigenvalues $\lambda_s^{-1}$ and 
$\lambda_s$, normalised so that $\langle\ac,\alpha_s\rangle=2$. Given a $W$-invariant function $c:S\rightarrow\C$, the \emph{rational Cherednik algebra} 
$H_c(W,\h)$ is the unital associative $\C$-algebra generated by $\h$, $\h^*$ and $W$, with relations
\begin{eqnarray*}
	wx&=&\act[w]xw,\\
	wy&=&\act[w]yw,\\{}
	[x,x']&=&0,\\{}
	[y,y']&=&0,\\{}
	[y,x]&=&\langle y,x\rangle-\sum_{s\in S}c(s)\langle y,\alpha_s\rangle\langle\ac,x\rangle s,
\end{eqnarray*}
for $x,\,x'\in\h^*$, $y,\,y'\in\h$ and $w\in W$.
\end{defn}
If there is no risk of confusion, we denote this algebra simply by $H_c$.

Much progress has been made in the representation theory of $H_c$ by restricting attention to finitely generated modules on which $\h$ acts locally nilpotently. The category of such modules, introduced by Opdam and 
Rouquier \cite{GGOR}, is denoted by $\cO(H_c)$ and displays many similarities with ``category $\cO$" for semisimple complex Lie algebras; this point of view is explained in \cite{RouquierSurvey}. The natural homomorphism 
$\C[\h]\rightarrow H_c$ allows us to think of such modules as coherent sheaves on the complex variety $\h$. By completing at various points of $\h$, Bezrukavnikov and Etingof \cite{BE} characterised the possible support sets of 
such a module, showing in particular that any irreducible component of this set is the set of fixed points of some subgroup of $W$. Moreover they constructed the following flat connections from these modules (see 
Proposition 3.20 of \cite{BE}).
\begin{prop}\label{flatconn}
Suppose $M\in\cO(H_c)$ and $W'$ is a subgroup of $W$. Let $Y$ be the set of points in $\h$ whose stabiliser is $W'$, and let $i_Y:Y\hookrightarrow\h$ be the inclusion. Denoting by $Sh(M)$ the coherent sheaf on $\h$ corresponding 
to the $\C[\h]$-module $M$, there is a flat connection on the coherent sheaf pullback $i_Y^*Sh(M)$ determined by
\[
	\nabla_ym=ym-\sum_{s\in S\setminus W'}c(s)\langle y,\alpha_s\rangle\frac{2}{1-\lambda_s}\frac{1}{\alpha_s}(s-1)m
\]
for $m\in M$ and $y\in\h^{W'}$.
\end{prop}
This flat connection is a special case of Theorem \ref{coh}(2) below. In fact this statement holds for any module $M$ in the category $H_c\coh$ of modules finitely generated over $\C[\h]\subseteq H_c$. This allows us to give the 
following alternative characterisation of the category $\cO(H_c)$.
\begin{prop}\label{RS}
The category $\cO(H_c)$ is a Serre subcategory of $H_c\coh$. Moreover given an irreducible $M\in H_c\coh$, let $W'\subseteq W$ be a subgroup whose fixed point set is a component of $\supp M$. Then $M$ lies in $\cO(H_c)$ if and only 
if the flat connection of Proposition \ref{flatconn} has regular singularities.
\end{prop}
\subsection{Actions on Varieties}
Now suppose $W$ acts on a smooth complex algebraic variety $X$, and $\omega$ is a $W$-invariant closed $2$-form on $X$. We recall briefly the notion of twisted differential operators \cite{BB}. Let $\D_\omega(X)$ denote the 
sheaf of algebras generated over $\cO_X$ by the tangent bundle $\cT X$, with relations
\[
	xy-yx=[x,y]+\omega(x,y),\hspace{20mm}xf-fx=x(f)
\]
for vector fields $x$ and $y$ and regular functions $f$, where $[\cdot,\cdot]$ denotes the usual Lie bracket of vector fields (note that throughout this paper, scripted letters will generally denote sheaves of modules or algebras). 
To give an action of $\D_\omega(X)$ on a quasi-coherent sheaf $\cM$ is equivalent to giving a connection on $\cM$ with curvature $\omega$. Given an immersion of a smooth curve $i:C\hookrightarrow X$, we obtain a connection on the 
pullback $i^*\cM$ which is trivially flat, so $i^*\cM$ may be thought of as an untwisted $\D$-module. We say $\cM$ has regular singularities if $i^*\cM$ has regular singularities in the usual sense for every such immersion. This 
definition was given by Finkelberg and Ginzburg \cite{FG_TDOs} for 2-forms which are \'etale-locally exact. In fact we will only be interested in coherent sheaves $\cM$ over $\D_\omega(X)$, and the existence of such a sheaf 
ensures that $\omega$ is Zariski-locally exact.

Any 1-form $\alpha$ gives rise to an isomorphism $\D_\omega(X)\cong\D_{\omega+d\alpha}(X)$. Therefore by patching sheaves of algebras of the form $\D_\omega(X)$, we obtain a sheaf of algebras $\D_\psi(X)$ corresponding 
to any class $\psi\in H^2(X,\Omega_X^{\geq1})$, where $\Omega_X^{\geq1}$ is the two step complex $\Omega_X^1\rightarrow\Omega_X^{2,cl}$ lying in degrees $1$ and $2$, $\Omega_X^1$ is the sheaf of 1-forms and $\Omega_X^{2,cl}$ 
the sheaf of closed 2-forms. When $X$ is affine, any such class is represented by a global 2-form. Note that our definition of regular singularities depends on a global 2-form chosen to represent the class.

Etingof \cite{EtingofVariety} has defined a sheaf of algebras $\cH_\cco(W,X)$ on $X/W$, generalizing Definition \ref{defnHc}. We will recall this definition below (see Definition \ref{HWX}) after developing some preliminaries. 
There is a natural copy of the structure sheaf $\cO_X$ in $\cH_\cco(W,X)$, and we will consider the full subcategory $\cH_\cco\coh$ of $\cH_\cco(W,X)$-mod, consisting of sheaves of modules which are coherent as $\cO_X$-modules. 
Our first goal is to classify possible support sets of such modules, in analogy with the results of \cite{BE}. Explicitly, given a subgroup $W'\subseteq W$, let
\begin{eqnarray*}
	X^{W'}&=&\{x\in X\mid\act[w]x=x\text{ for }w\in W'\},\\
	X^{W'}_{\rm reg}&=&\{x\in X\mid\stab_W(x)=W'\}.
\end{eqnarray*}
Also define
\[
	P=\{Y\mid Y\text{ is a component of }X^{W'}_{\rm reg}\text{ for some }W'\subseteq W\}.
\]
These subsets are locally closed, and may be viewed as (non-affine) varieties. Let $P'$ denote the set of all $Y\in P$ such that $H_c(W',T_xX/T_xX^{W'})$ admits a nonzero finite dimensional module, where $x$ is any point of 
$Y$ and $W'=\stab_W(x)$. We will prove:
\begin{thm}\label{coh}
Suppose $\cM\in\cH_\cco\coh$.
\begin{enumerate}
\item
	Suppose $Z\subseteq X$ is a closed $W$-invariant subset of $X$, and consider the subsheaf of ``$Z$-torsion" elements in $\cM$,
	\[
		\Gamma_Z(\cM)(U)=\{m\in\cM(U)\mid\supp m\subseteq Z\}.
	\]
	That is, $\Gamma_Z(\cM)$ is the sum of all coherent subsheaves of $\cM$ which are set-theoretically supported on $Z$. Then $\Gamma_Z(\cM)$ is an $\cH_\cco$-submodule of $\cM$.
\item
	Let $Y\in P$ and let $i_Y:Y\hookrightarrow X$ be the inclusion. The coherent sheaf pullback $i_Y^*(\cM)$ on $Y$ admits a natural 
	action of $\D_{i_Y^*\psi}(Y)$. In particular, if $\psi=0$, then $i_Y^*(\cM)$ admits a natural
	flat connection.
\item
	The set-theoretical support of $\cM$ has the form
	\[
		\supp\cM=\bigcup_{Y\in P_{\cM}}\overline Y
	\]
	for some $W$-invariant subset $P_{\cM}\subseteq P'$.
\item There is an integer $K>0$, depending only on $c$, $W$ and $X$, such that any such $\cM$ is scheme-theoretically supported on the $K^{\rm th}$ neighbourhood of its set-theoretical support.
\item Every object of $\cH_\cco\coh$ has finite length.
\item If $\cM$ is irreducible then we may take $P_{\cM}$ in part (3) to be a single $W$-orbit in $P'$.
\end{enumerate}
\end{thm}
We would like a sensible subcategory of $\cH_\cco\coh$ in which to study the representation theory of $\cH_\cco$, analogous to the category $\cO(H_c)$ in the linear case. Motivated by Proposition \ref{RS}, we make the following 
definition. Again we need to choose a global 2-form $\omega$ representing the class $\psi$ for this definition.
\begin{defn}\label{RSdefn}
Let $\cH_\co\rs$ denote the Serre subcategory of $\cH_\co\coh$, such that an irreducible $\cM\in\cH_\co\coh$ lies in $\cH_\co\rs$ exactly when the connection on $i_Y^*(M)$ given in Theorem \ref{coh}(2) has regular singularities, 
where $Y\in P_{\cM}$ is as in Theorem \ref{coh}(6).
\end{defn}
For a linear action, Proposition \ref{RS} shows that this category coincides with $\cO(H_c)$. Nevertheless we will use the notation $\cH_\co\rs$ even in the linear case to avoid confusion with the structure sheaf of a variety.
\subsection{The Type A Case}
Taking $X$ to be an open subset of a vector space, the above will be of use in the sequel, in which we study representations of $H_c=H_c(S_n,\C^n)$, where $S_n$ is the symmetric group acting on $\C^n$ by permuting coordinates. 
The category $H_c\rs$ is semisimple unless $c$ is rational with denominator between $2$ and $n$ (see \cite{BEG1}), so we take $c=\frac rm$ where $m\geq2$ is coprime with $r$. It is shown in \cite{BE} (and follows from 
Theorem \ref{coh}) that the support of any module in $H_c\rs$ is of the form
\[
	X_q=\{b\in\h\mid\stab_{S_n}(b)\cong S_m^q\}
\]
for some integer $q$ with $0\leq q\leq\frac nm$. It is known that the irreducible modules in $H_c\rs$ are parameterised by the irreducible representations of $\C[S_n]$, which are in turn parameterised by partitions of $n$. 
Given a partition $\lambda\vdash n$, let $\tau_\lambda$ and $L(\tau_\lambda)$ denote the corresponding representation of $S_n$ and $H_c$ respectively. The support of the latter is determined by the following.
\begin{thm}\label{suppSnpart}
If $c>0$, then the support of the $H_c$-module $L(\tau_\lambda)$ is $X_{q_m(\lambda)}$, where
\[
	q_m(\lambda)=\sum_{i\geq1}i\left\lfloor\frac{\lambda_i-\lambda_{i+1}}{m}\right\rfloor.
\]
If $c<0$, the support of $L(\tau_\lambda)$ is $X_{q_m(\lambda')}$, where $\lambda'$ is the transpose of $\lambda$.
\end{thm}
Note that any $\lambda\vdash n$ can be uniquely written as $m\mu+\nu$ where $\mu\vdash q_m(\lambda)$ and $\nu'$ is $m$-regular. In particular, the above proves the following conjecture of Bezrukavnikov and Okounkov. While this 
paper was in preparation, this result was generalised to the cyclotomic case by Shan and Vasserot \cite{ShanVasserot}.
\begin{cor}\label{BOconj}
Consider the universal enveloping algebra $A$ of the Heisenberg algebra, with generators $\{\alpha_i\mid i\in\Z,\,i\neq0\}$ and relation $[\alpha_i,\alpha_j]=i\delta_{i,-j}$. Consider the grading on $A$ defined by 
$\deg(\alpha_i)=i$. Let $F$ denote \emph{Fock space}, that is, the left $A$-module
\[
	F=A/{\rm span}\{A\alpha_i\mid i>0\}.
\]
The number of irreducibles in $H_c\rs$ whose support is $X_q$ is the dimension of the $qm$-eigenspace of the operator
\[
	\sum_{i>0}\alpha_{-im}\alpha_{im}
\]
acting on the degree $n$ part of $F$.
\end{cor}
Denoting by $H_c\rs^q$ the Serre subcategory of $H_c\rs$ consisting of all modules supported on $X_q$, we will determine the structure of the quotient category
\[
	H_c\rs^q/H_c\rs^{q+1}
\]
(where $H_c\rs^{\lfloor n/m\rfloor+1}$ is the subcategory containing only the zero module). Explicitly, let $p=n-qm$ and $\bq=e^{2\pi ic}$, and consider the \emph{Hecke algebra} $H_\bq(S_p)$ with generators 
$T_1,\ldots,T_{p-1}$ and relations
\begin{eqnarray*}
	T_iT_j&=&T_jT_i\text{ if }|i-j|>1,\\
	T_iT_{i+1}T_i&=&T_{i+1}T_iT_{i+1},\\
	(T_i-1)(T_i+\bq)&=&0.
\end{eqnarray*}
The irreducible modules over $H_\bq(S_p)$ are indexed by $m$-regular partitions of $p$ \cite{DJ}. Given $\nu\vdash p$ with $q_m(\nu')=0$, let $D_\nu$ denote the corresponding irreducible. We will show:
\begin{thm}\label{rsqrsq+1}
With $c=\frac rm$ and $\bq=e^{2\pi ic}$, the category $H_c\rs^q/H_c\rs^{q+1}$ is equivalent to the category of finite dimensional modules over $\C[S_q]\otimes_\C H_\bq(S_p)$. If $\nu\vdash p$ is $m$-regular, and $\mu\vdash q$, 
the irreducible in $H_c\rs^q$ corresponding to $\tau_\mu\otimes D_\nu$ under this equivalence is $L(\tau_{m\mu+\nu'})$ if $c>0$, and $L(\tau_{(m\mu+\nu')'})$ if $c<0$.
\end{thm}
\subsection{Outline of the Paper}
The paper is organised as follows. In Section \ref{Sec:coh}, after some algebraic geometry preliminaries we recall the definition of the rational Cherednik algebra of a variety, and prove Theorem \ref{coh}. In Section 
\ref{Sec:linear} we state some known results about the representation theory for linear actions, and in particular for $H_{\frac rm}(S_n,\C^n)$. From this we prove Proposition \ref{RS} and deduce one direction of Theorem 
\ref{suppSnpart}. We restrict to the case $c=\frac1m$ in Section \ref{Sec:minsupp}, and construct an explicit equivalence from the category of minimally supported representations. This enables us, in Section \ref{Sec:mon}, 
to prove Theorem \ref{rsqrsq+1} for $c=\frac1m$. From this we deduce Theorems \ref{rsqrsq+1} and \ref{suppSnpart} in general.
\subsection{Acknowledgements}
The author thanks Pavel Etingof for many helpful suggestions and insights, Ivan Losev for useful discussions concerning Theorem \ref{twosided}, and Dennis Gaitsgory for explanations about the theory of $D$-modules.
\section{Coherent Representations}\label{Sec:coh}
Suppose $X$ is a smooth algebraic variety over $\C$ and $W$ a finite group acting on $X$, such that the set of points with trivial stabiliser is dense in $X$. Let $\omega$ be a $W$-invariant closed 2-form on $X$. Suppose for the 
moment that $X$ is affine. In order to define the rational Cherednik algebra $\cH_\co(W,X)$, we require the following lemma, which is shown in Section 2.4 of \cite{EtingofVariety}.
\begin{lemm}
Suppose $Z\subseteq X$ is a smooth closed subscheme of codimension $1$. Let $\cO(X)$ denote the ring of regular functions on $X$, and $\cO(X)\langle Z\rangle$ the space of rational functions on $X$ whose only pole is along $Z$, 
with order at most $1$. There is a natural $\cO(X)$-module homomorphism $\xi_Z:TX\rightarrow\cO(X)\langle Z\rangle/\cO(X)$ whose kernel consists of all vector fields preserving the ideal sheaf of $Z$.
\end{lemm}
Since $TX$ is a projective $\cO(X)$-module, we may lift $\xi_Z$ along the surjection
\[
	\cO(X)\langle Z\rangle\twoheadrightarrow\cO(X)\langle Z\rangle/\cO(X)
\]
to an $\cO(X)$-module homomorphism $\zeta_Z:TX\rightarrow\cO(X)\langle Z\rangle$. It is known (and follows from Proposition \ref{linear} below) that $X^{W'}$ is a smooth closed subscheme for any subset $W'\subseteq W$.
\begin{defn}[Definitions 2.7 and 2.8 of \cite{EtingofVariety}]\label{HWX}
Let $S$ denote the set of pairs $(Z,s)$, where $s\in W$ and $Z$ is an irreducible component of $X^s$ of codimension $1$ in $X$. Let $c:S\rightarrow\C$ be a $W$-invariant function. Let $X_{\rm reg}$ denote the set of points 
in $X$ with trivial stabiliser in $W$, and
\[
	D_\omega(X_{\rm reg})=\Gamma(X_{\rm reg},\mathcal D_\omega(X))
\]
the algebra of global algebraic twisted differential operators on the smooth scheme $X_{\rm reg}$. For each vector field $v$ on $X$, we define the Dunkl-Opdam operator $D_v\in\C[W]\ltimes D_\omega(X_{\rm reg})$ by
\[
	D_v=v+\sum_{(Z,s)\in S}\frac{2c(Z,s)}{1-\lambda_{Z,s}}\zeta_Z(v)(s-1),
\]
where $\lambda_{Z,s}$ is the determinant of $s$ on $T_xX^*$ for any $x\in Z$. The rational Cherednik algebra $H_\co(W,X)$ is the unital $\C$-subalgebra of $\C[W]\ltimes D_\omega(X_{\rm reg})$ generated by $\C[W]\ltimes\cO(X)$ 
and the $D_v$.
\end{defn}
\textbf{Remarks:}
\begin{enumerate}
\item
	Although $D_v$ depends on the choice of lift $\zeta_Z$, the algebra $H_\co(W,X)$ does not.
\item
	Proposition \ref{local} below shows that this algebra behaves well with respect to \'etale morphisms. Moreover if $\alpha$ is any $W$-invariant 1-form on $X$, the isomorphism
	\[
		\C[W]\ltimes D_\omega(X_{\rm reg})\cong\C[W]\ltimes D_{\omega+d\alpha}(X_{\rm reg})
	\]
	identifies $H_\co(W,X)$ with $H_{c,\omega+d\alpha}(W,X)$. Therefore if we do not assume $X$ is affine, and we take a $W$-invariant class 
	$\psi\in H^2(X,\Omega_X^{\geq1})^W$ rather than a global 2-form, we may patch algebras of the above form to construct a sheaf of algebras $\cH_\cco(W,X)$ on $X/W$. Nevertheless, for the moment we will continue to 
	assume $X$ is affine and $\omega$ is a specified 2-form.
\end{enumerate}
\begin{prop}\label{local}
Suppose $p:U\rightarrow X$ is a $W$-equivariant \'etale morphism, with $U$ affine. For each component $Z'$ of $U^s$ of codimension $1$, the image of $Z'$ is a component $Z$ of $X^s$ of codimension $1$, and 
we set $c'(Z',s)=c(Z,s)$. There is a natural $\cO(U)$-module isomorphism $\cO(U)\otimes_{\cO(X)}H_\co(W,X)\isom H_{c',p^*\omega}(W,U)$, whose composition with 
$H_\co(W,X)\rightarrow\cO(U)\otimes_{\cO(X)}H_\co(W,X)$ is an algebra homomorphism. Moreover given any $H_\co(W,X)$-module $M$, there is a natural action of $H_{c',p^*\omega}(W,U)$ on $\cO(U)\otimes_{\cO(X)}M$.
\end{prop}
\begin{proof}
Since $\cO(U)$ is flat over $\cO(X)$, the inclusion $H_\co(W,X)\subseteq\C[W]\ltimes D_\omega(X_{\rm reg})$ induces an $\cO(U)$-module monomorphism
\begin{eqnarray*}
	i:\cO(U)\otimes_{\cO(X)}H_\co(W,X)
		&\hookrightarrow&\cO(U)\otimes_{\cO(X)}\C[W]\ltimes D_\omega(X_{\rm reg})\\
		&=&\C[W]\ltimes D_{p^*\omega}(U_{\rm reg}).
\end{eqnarray*}
Moreover, for an appropriate choice of the lifts $\zeta_Z$, for any $v\in TX$ the image of $D_v\in H_\co(W,X)$ under $i$ is $D_{p^*v}$. Now $H_{c',p^*\omega}(W,U)$ is the subalgebra of $\C[W]\ltimes D_{p^*\omega}(U_{\rm reg})$ 
generated by $\C[W]\ltimes\cO(U)$ and $\{D_v\mid v\in TU\}$. But
\[
	TU=\cO(U)\otimes_{\cO(X)}TX,
\]
and
\[
	[D_v,f]=v(f)+\sum_{(Z,s)\in S}\frac{2c(Z,s)}{1-\lambda_{Z,s}}\zeta_Z(v)(\act f-f)s\in\C[W]\ltimes\cO(U)
\]
for any $v\in TU$ and $f\in\cO(U)$. It follows that $H_{c',p^*\omega}(W,U)$ is spanned by elements of the form
\[
	fwD_{p^*v_1}D_{p^*v_2}\ldots D_{p^*v_k}
\]
for $f\in\cO(U)$, $w\in W$ and $v_i\in TX$. Applying the same argument with $p$ equal to the identity on $X$, we see that $H_\co(W,X)$ is spanned by
\[
	fwD_{v_1}D_{v_2}\ldots D_{v_k}
\]
for $f\in\cO(X)$, $w\in W$ and $v_i\in TX$. Thus the image of $i$ is exactly $H_{c',p^*\omega}(W,U)$, giving the required isomorphism $j:\cO(U)\otimes_{\cO(X)}H_\co(W,X)\isom H_{c',p^*\omega}(W,U)$. 
Note that the composition of $i$ with $H_\co(W,X)\rightarrow\cO(U)\otimes_{\cO(X)}H_\co(W,X)$ equals the composite
\[
	H_\co(W,X)\subseteq\C[W]\ltimes D_\omega(X_{\rm reg})\rightarrow\C[W]\ltimes D_{p^*\omega}(U_{\rm reg}),
\]
which is an algebra homomorphism. In particular, $j$ also preserves the right $\cO(X)$-module structure.

Now suppose $M$ is an $H_\co(W,X)$-module. The multiplication map
\[
	H_{c',p^*\omega}(W,U)\otimes_\C\cO(U)\rightarrow H_{c',p^*\omega}(W,U)
\]
and the action map $H_\co(W,X)\otimes_{\cO(X)}M\rightarrow M$ give rise to a map
\begin{eqnarray*}
	H_{c',p^*\omega}(W,U)\otimes_\C\cO(U)\otimes_{\cO(X)}M\hspace{-20mm}\\
		&\rightarrow&H_{c',p^*\omega}(W,U)\otimes_{\cO(X)}M\\
		&\cong&\cO(U)\otimes_{\cO(X)}H_\co(W,X)\otimes_{\cO(X)}M\\
		&\rightarrow&\cO(U)\otimes_{\cO(X)}M.
\end{eqnarray*}
It is straightforward to check that this defines an action.
\end{proof}
As in \cite{BE}, we will study representations of $\cH_\co(W,X)$ by restricting to the formal neighbourhood of a point, or more generally a closed subset.
\begin{prop}\label{formal}
Suppose $Z$ is a $W$-invariant closed subset of $X$, and let $I\subseteq\cO(X)$ denote the ideal vanishing on $Z$. Consider the coordinate ring of the ``formal neighbourhood" of $Z$,
\[
	\hO_{X,Z}=\lim_{\leftarrow\atop k}\cO(X)/I^k.
\]
There is a natural algebra structure on $\hO_{X,Z}\otimes_{\cO(X)}H_\co(W,X)$, and this algebra acts naturally on $\hO_{X,Z}\otimes_{\cO(X)}M$ for any $M$ in $H_\co(W,X)\coh$.
\end{prop}
\begin{proof}
We define an algebra filtration $H_\co^{\leq d}$ of $H_\co=H_\co(W,X)$ as follows. Let $H_\co^{\leq0}=\C[W]\ltimes\cO(X)$ and
\[
	H_\co^{\leq1}=H_\co^{\leq0}+\C[W]\{D_v\mid v\in TX\}.
\]
This is independent of the choice of lift $\zeta_Z$, since different choices of $D_v$ differ by elements of $\C[W]\ltimes\cO(X)$. Finally let
\[
	H_\co^{\leq d}=\left(H_\co^{\leq1}\right)^d.
\]
As in the previous proof, if $v_1,\ldots,v_m$ generate $TX$ over $\cO(X)$, then $H_\co^{\leq d}$ is generated over $\cO(X)$ by
\[
	wD_{v_{i_1}}\ldots D_{v_{i_k}}
\]
for $w\in W$ and $k\leq d$. In particular, $H_\co^{\leq d}$ is a finitely generated left $\cO(X)$-module.

Now $I$ is $W$-invariant, so inside $H_\co$ we have
\[
	(\C[W]\ltimes\cO(X))I=I(\C[W]\ltimes\cO(X))
\]
and
\[
	[D_v,I]\subseteq[D_v,\cO(X)]\subseteq\C[W]\ltimes\cO(X).
\]
It follows by induction on $k$ that $H_\co^{\leq1}I^{k+1}\subseteq I^kH_\co^{\leq1}$, so that $H_\co^{\leq d}I^{k+d}\subseteq I^kH_\co^{\leq d}$ for all $d,\,k\geq0$. The multiplication map
\[
	H_\co^{\leq d}\otimes H_\co^{\leq e}\rightarrow H_\co^{\leq d+e}
\]
therefore naturally induces a map
\[
	H_\co^{\leq d}\otimes\left(\cO(X)/I^{d+k}\otimes_{\cO(X)}H_\co^{\leq e}\right)\rightarrow\cO(X)/I^k\otimes_{\cO(X)}H_\co^{\leq d+e}.
\]
Taking inverse limits we obtain $\hat H_\co^{\leq d}\otimes\hat H_\co^{\leq e}\rightarrow\hat H_\co^{\leq d+e}$, where
\[
	\hat H_\co^{\leq d}=\lim_{\leftarrow\atop k}\cO(X)/I^k\otimes_{\cO(X)}H_\co^{\leq d}=\hO_{X,Z}\otimes_{\cO(X)}H_\co^{\leq d}.
\]
In this way the space
\[
	\hat H_\co=\bigcup_{d\geq0}\hat H_\co^{\leq d}=\hO_{X,Z}\otimes_{\cO(X)}H_\co
\]
becomes an associative algebra. Moreover for any $M\in H_\co\coh$, the action map induces
\[
	H_\co^{\leq d}\otimes\left(\cO(X)/I^{d+k}\otimes_{\cO(X)}M\right)\rightarrow\cO(X)/I^k\otimes_{\cO(X)}M,
\]
and taking inverse limit gives an action of $\hat H_\co$ on $\hO_{X,Z}\otimes_{\cO(X)}M$.
\end{proof}
Next we show that the action of $W$ on $X$ looks, on the formal neighbourhood of the fixed point set, like a linear action.
\begin{prop}\label{linear}
Suppose $Z\subseteq X$ is a smooth closed subset which is fixed pointwise by the action of $W$. Then every $x\in Z$ admits an affine open 
$W$-invariant neighbourhood $U\subseteq X$ such that there is a $W$-equivariant ring isomorphism $\phi$ making the following diagram commute:
\[\xymatrix{
	\hO_{(U\cap Z)\times(T_xX/T_xZ),(U\cap Z)\times\{0\}}\ar[r]^-{\phi}_-{\sim}\ar@{->>}[dr]&\hO_{U,U\cap Z}\ar@{->>}[d]\\
	&\cO(U\cap Z).
}\]
Here $W$ acts on the first ring according to its linear action on $T_xX/T_xZ$.
\end{prop}
\begin{proof}
Both rings are inverse limits, so it suffices to construct compatible $W$-equivariant ring isomorphisms
\[
	\phi_k:\Gamma(X,\cO_X/\cI)\otimes\C[T_xX/T_xZ]/\m^k\isom\Gamma(X,\cO_X/\cI^k),
\]
such that $\phi_1$ is the identity, where $\m\subseteq\C[T_xX/T_xZ]$ is the ideal corresponding to the origin and $\cI\subseteq\cO_X$ is the ideal sheaf vanishing on $Z$.

Let $T_xZ^\perp$ denote the subspace of $T_xX^*$ vanishing on $T_xZ$. This is the image of $\Gamma(X,\cI)$ under the gradient map $\Gamma(X,\cI)\rightarrow T_xX^*$. Let $a_1,\ldots,a_n$ be a basis for $T_xX$, and 
$b_1,\ldots,b_r$ a basis for $T_xZ^\perp$, such that $\langle a_i,b_j\rangle=\delta_{ij}$ for $1\leq i\leq n$ and $1\leq j\leq r$. Since $W$ is finite and acts linearly on $\Gamma(X,\cI)$ and $T_xZ^\perp$, and we are working over 
characteristic zero, Maschke's theorem implies the existence of a $W$-equivariant $\C$-linear map $\beta:T_xZ^\perp\rightarrow \Gamma(X,\cI)$ which is right inverse to the surjection $\Gamma(X,\cI)\twoheadrightarrow T_xZ^\perp$. 
Let $f_j\in \Gamma(X,\cI)$ be the image of $b_j$. Also choose $v_i\in TX$ mapping to $a_i\in T_xX$. Let $U$ denote the open neighbourhood of $x$ on which the matrix $(v_i(f_j))_{1\leq i,j\leq r}$ is invertible. The functions 
$f_1,\ldots,f_r$ have linearly independent gradients on $U$, so their zero set $Z'\subseteq U$ has codimension $r$. However $Z'\supseteq Z\cap U$ since $f_j\in \Gamma(X,\cI)$, and the dimension $r$ of $T_xZ^\perp$ equals the 
codimension in $X$ of the component of $Z$ containing $x$. Therefore $Z$ coincides with $Z'$ on some neighbourhood of $x$. By shrinking $U$, we may suppose that $U$ is an affine, $W$-invariant open neighbourhood of $x$ such that 
$\Gamma(U,\cI)$ is generated by the $f_j$. We now inductively construct ring homomorphisms
\[
	\gamma_k:\Gamma(U,\cO_X/\cI)\rightarrow\Gamma(U,\cO_X/\cI^k)^W
\]
for $k\geq1$, compatible with the projections $\Gamma(U,\cO_X/\cI^{k+1})^W\twoheadrightarrow\Gamma(U,\cO_X/\cI^k)^W$, 
and such that $\gamma_1$ is the identity (note that $W$ acts trivially on $\Gamma(U,\cO_X/\cI)=\Gamma(U\cap Z,\cO_Z)$, since $Z$ is fixed pointwise by $W$). Suppose we have $\gamma_k$, where $k\geq1$. Let $A=\C[x_1,\ldots,x_m]$ be 
a polynomial ring mapping surjectively to $\Gamma(U,\cO_X)$, and let $\p\subseteq A$ be the inverse image of the ideal $\Gamma(U,\cI)$. Note that $\Gamma(U,\cO_X/\cI)=A/\p$. Now choose $y_i\in\Gamma(U,\cO_X/\cI^{k+1})^W$ mapping to 
$\gamma_k(x_i)\in\Gamma(U,\cO_X/\cI^k)^W$. We have a ring homomorphism $\gamma':A\rightarrow\Gamma(U,\cO_X/\cI^{k+1})^W$ sending $x_i$ to $y_i$, and the composite with $\Gamma(U,\cO_X/\cI^{k+1})^W\rightarrow\Gamma(U,\cO_X/\cI^k)^W$ factors 
through $\gamma_k$. In particular, the composite kills $\p$, so $\gamma'(\p)\subseteq\Gamma(U,\cI^k/\cI^{k+1})^W$. Also the composite of $\gamma'$ with $\Gamma(U,\cO_X/\cI^{k+1})^W\rightarrow\Gamma(U,\cO_X/\cI)$ is the natural projection 
$A\rightarrow\Gamma(U,\cO_X/\cI)$. It follows that the restriction
\[
	\delta=\gamma'|_{\p}:\p\rightarrow\Gamma(U,\cI^k/\cI^{k+1})^W
\]
is an $A$-module homomorphism. Certainly then $\delta(\p^2)=0$. We have an exact sequence
\[
	0\rightarrow\p^2\rightarrow\p\rightarrow A/\p\otimes_AT(\Spec A)^*\rightarrow T(\Spec A/\p)^*,
\]
where the map $\p\rightarrow A/\p\otimes_AT(\Spec A)^*$ is the gradient map. Since
\[
	\Spec A/\p=U\cap Z
\]
is smooth, $T(\Spec A/\p)^*$ is a projective $A/\p$-module. Therefore $\delta$ factors through the gradient map. Moreover $T(\Spec A)^*$ is freely generated over $A$ by $dx_1,\ldots,dx_m$. Therefore we may find 
$z_1,\ldots,z_m\in\Gamma(U,\cI^k/\cI^{k+1})^W$ such that
\[
	\delta(f)=\sum_{i=1}^m\frac{\partial f}{\partial x_i}z_i.
\]
Let $\gamma'':A\rightarrow\Gamma(U,\cO_X/\cI^{k+1})^W$ be the ring homomorphism sending $x_i$ to $y_i-z_i$. Since $k\geq1$, we have
\[
	\gamma''(f)=\gamma'(f)-\sum_{i=1}^m\frac{\partial f}{\partial x_i}z_i.
\]
In particular, $\gamma''$ kills $\p$, so it induces the required map $\gamma_{k+1}:A/\p\rightarrow\Gamma(U,\cO_X/\cI^{k+1})^W$.

Now $\C[T_xX/T_xZ]$ is freely generated by $b_1,\ldots,b_r$, so we may extend $\gamma_k$ to a homomorphism
\[
	\phi_k':\Gamma(U,\cO_X/\cI)\otimes\C[T_xX/T_xZ]\rightarrow\Gamma(U,\cO_X/\cI^k)
\]
by sending $b_j$ to $f_j$. Note that the $\phi_k'$ are compatible as $k$ varies, and are $W$-equivariant since $\beta$ is. Since $f_j\in\Gamma(X,\cI)$ for each $j$, we have 
$\phi_k'(\Gamma(U,\cO_X/\cI)\otimes\m^l)\subseteq\Gamma(U,\cI^l/\cI^k)$ for $0\leq l\leq k$. Thus $\phi_k'$ induces a map
\[
	\phi_k:\Gamma(U,\cO_X/\cI)\otimes\C[T_xX/T_xZ]/\m^k\rightarrow\Gamma(U,\cO_X/\cI^k),
\]
and to prove $\phi_k$ is an isomorphism, it suffices to show that the induced maps
\[
	\Gamma(U,\cO_X/\cI)\otimes\m^l/\m^{l+1}\rightarrow\Gamma(U,\cI^l/\cI^{l+1})
\]
are isomorphisms, for $0\leq l<k$. We took $\gamma_1$ to be the identity, so in fact this is a map of $\Gamma(U,\cO_X/\cI)$-modules. That is, we are required to prove that $\Gamma(U,\cI^l/\cI^{l+1})$ is freely generated over 
$\Gamma(U,\cO_X/\cI)$ by degree $l$ monomials in the $f_j$. This is clear when $l=0$. Moreover the monomials generate $\Gamma(U,\cI^l)$ over $\Gamma(U,\cO_X)$, since the $f_j$ generate $\Gamma(U,\cI)$ over $\Gamma(U,\cO_X)$. 
Finally suppose
\[
	\sum_\alpha g_\alpha f_1^{\alpha_1}\ldots f_r^{\alpha_r}\in\Gamma(U,\cI^{l+1})
\]
for some $g_\alpha\in\Gamma(U,\cO_X)$, where each monomial has degree $l$. The matrix $(v_i(f_j))_{1\leq i,j\leq r}$ is invertible on $U$, so we may find vector fields $v_1',\ldots,v_r'$ on $U$ satisfying 
$v_i'(f_j)=\delta_{ij}$. Since $v_i'\Gamma(U,\cI^{l+1})\subseteq\Gamma(U,\cI^l)$, we conclude that
\[
	\sum_\alpha\alpha_ig_\alpha f_1^{\alpha_1}\ldots f_i^{\alpha_i-1}\ldots f_r^{\alpha_r}\in\Gamma(U,\cI^l)
\]
for each $1\leq i\leq r$. If $l>0$ then for each $\alpha$ we have $\alpha_i\neq0$ for some $i$, so by induction we conclude that $g_\alpha\in\Gamma(U,\cI)$, as required.
\end{proof}
The next proposition generalises Theorem 3.2 of \cite{BE}, which applies to linear actions.
\begin{prop}\label{linearH}
Suppose $Z$ is a smooth connected closed subset of $X$, every point of which has the same stabiliser $W'$ in $W$. Suppose the $W$-translates of $Z$ are all equal to or disjoint with $Z$, and let $WZ$ denote their 
union. Finally let $W''$ be the subgroup of $W$ fixing $Z$ setwise. Then
\[
	\hO_{X,WZ}\otimes_{\cO(X)}H_\co(W,X)\cong\Mat_{[W:W'']}(\C[W'']\otimes_{\C[W']}H_\co(W',\Spf\hO_{X,Z})),
\]
where $H_\co(W',\Spf\hO_{X,Z})$ is an algebra depending only on the following data:
\begin{itemize}
\item the ring $\hO_{X,Z}$,
\item the action of $W'$ on $\hO_{W,Z}$,
\item the extension of $\omega$ to a map $\Der_\C(\hO_{X,Z})\wedge\Der_\C(\hO_{X,Z})\rightarrow\hO_{X,Z}$, and
\item the parameters $c(Z',s)$ for $Z'\supseteq Z$.
\end{itemize}
The isomorphism is natural up to a choice of coset representatives for $W''$ in $W$. There is a natural action of $\C[W'']\otimes_{\C[W']}H_\co(W',\Spf\hO_{X,Z})$ on $\hO_{X,Z}\otimes_{\cO(X)}M$ for any $H_\co(W,X)$-module $M$.
\end{prop}
\textbf{Remarks:}
\begin{enumerate}
\item
	The construction of $H_\co(W,X)$ can be extended to allow $X$ to be a formal scheme, and the algebra $H_\co(W',\Spf\hO_{X,Z})$ is an example of this construction. Nevertheless we give a self-contained definition 
	of this algebra below without making reference to formal schemes.
\item
	When $\omega=0$, the following proof can be simplified by embedding $H_\co(W,X)$ in $\End_\C(\cO(X))$ rather than $\C[W]\ltimes D_\omega(X_{\rm reg})$.
\end{enumerate}
\begin{proof}
We have a natural isomorphism
\[
	\Der_\C(\hO_{X,WZ})\cong\hO_{X,WZ}\otimes_{\cO(X)}TX,
\]
so the closed 2-form
\[
	\omega:TX\otimes_{\cO(X)}TX\rightarrow\cO(X)
\]
extends naturally to a closed 2-form $\Der_\C(\hO_{X,WZ})\otimes_{\hO_{X,WZ}}\Der_\C(\hO_{X,WZ})\rightarrow\hO_{X,WZ}$. We can now define an algebra $D_\omega(\hO_{X,WZ})$ in the same way that $D_\omega(X)$ was defined, and
\[
	D_\omega(\hO_{X,WZ})\cong\hO_{X,WZ}\otimes_{\cO(X)}D_\omega(X)
\]
as an $\hO_{X,WZ}$-module. Let $K(\hO_{X,WZ})$ be the localisation of $\hO_{X,WZ}$ at all elements which are not zero divisors. There is a natural algebra structure on
\[
	K(\hO_{X,WZ})\otimes_{\hO_{X,WZ}}D_\omega(\hO_{X,WZ})=K(\hO_{X,WZ})\otimes_{\cO(X)}D_\omega(X).
\]
Since $X$ is smooth and $X_{\rm reg}$ is dense in $X$, the inclusion $\cO(X)\hookrightarrow\hO_{X,Z}$ extends to a monomorphism $\Gamma(X_{\rm reg},\cO_X)\hookrightarrow K(\hO_{X,WZ})$. We therefore obtain an algebra monomorphism
\[
	D_\omega(X_{\rm reg})=\Gamma(X_{\rm reg},\cO_X)\otimes_{\cO(X)}D_\omega(X)\hookrightarrow K(\hO_{X,WZ})\otimes_{\cO(X)}D_\omega(X).
\]
From this we construct a monomorphism
\[
	H_\co(W,X)\hookrightarrow\C[W]\ltimes D_\omega(X_{\rm reg})\hookrightarrow\C[W]\ltimes(K(\hO_{X,WZ})\otimes_{\cO(X)}D_\omega(X)).
\]
Therefore $\hO_{X,WZ}\otimes_{\cO(X)}H_\co(W,X)$ may be naturally identified with a subalgebra of $\C[W]\ltimes(K(\hO_{X,WZ})\otimes_{\cO(X)}D_\omega(X))$.

Let $C$ be a set of left coset representatives for $W''$ in $W$. We choose $1\in C$ to be the representative for $W''$ itself. We have
\[
	WZ=\coprod_{w\in C}wZ.
\]
We have assumed that the closed sets on the right are pairwise disjoint, so
\[
	K(\hO_{X,WZ})=\bigoplus_{w\in C}K(\hO_{X,wZ}),
\]
where $K(\hO_{X,wZ})$ is the field of fractions of $\hO_{X,wZ}$. Let $e$ denote the identity of $\hO_{X,Z}$ in the above direct sum. Note that $W''$ fixes $e$. Moreover for any $w\in C\setminus\{1\}$ we have $(\act[w]e)e=0$, 
since $\act[w]e$ is the identity of $\hO_{X,wZ}$ in this direct sum. It follows that there is an isomorphism
\[
	\phi:\C[W]\ltimes(K(\hO_{X,WZ})\otimes_{\cO(X)}D_\omega(X))\rightarrow\Mat_C(\C[W'']\ltimes(K(\hO_{X,Z})\otimes_{\cO(X)}D_\omega(X))),
\]
where $\Mat_C$ denotes the algbera of matrices with rows and columns indexed by $C$. Explicitly, for $w_1,\,w_2\in C$ and $a\in\C[W'']\ltimes(K(\hO_{X,Z})\otimes_{\cO(X)}D_\omega(X))$, $\phi$ sends $w_1eaw_2^{-1}$ to the matrix with 
$a$ in entry $(w_1,w_2)$ and zeros elsewhere. Therefore to describe $\hO_{X,WZ}\otimes_{\cO(X)}H_\co(W,X)$, it suffices to determine its image under $\phi$.

Now $\hO_{X,WZ}\otimes_{\cO(X)}H_\co(W,X)$ is generated as an algebra by $\C[W]\ltimes\hO_{X,WZ}$ and the Dunkl-Opdam operators. Under $\phi$, the first subalgebra generates
\[
	\Mat_C(\C[W'']\ltimes\hO_{X,Z})\subseteq\Mat_C(\C[W'']\ltimes(K(\hO_{X,Z})\otimes_{\cO(X)}D_\omega(X))).
\]
In particular, this contains $\Mat_C(\C)$, so the image of $\hO_{X,WZ}\otimes_{\cO(X)}H_\co(W,X)$ is $\Mat_C(A)$, where $A$ is the subalgebra of $\C[W'']\ltimes(K(\hO_{X,Z})\otimes_{\cO(X)}D_\omega(X))$ generated by 
$\C[W'']\ltimes\hO_{X,Z}$ and the entries of the images of the Dunkl-Opdam operators under $\phi$. Recall that these operators are given by
\[
	D_v=v+\sum_{(Z',s)\in S}\frac{2c(Z',s)}{1-\lambda_{Z',s}}\zeta_{Z'}(v)(s-1)\in\C[W]\ltimes D_\omega(X_{\rm reg}),
\]
for $v\in TX$. Given $(Z',s)\in S$, if $Z$ intersects $Z'$, then there is a point in $Z$ fixed by $s$, so $s\in W'$ and $Z\subseteq Z'$. On the other hand, if $Z'$ is disjoint with $Z$, then $\zeta_{Z'}(v)$ 
defines a regular function in $\hO_{X,Z}$. Therefore $A$ is generated by $\C[W'']\ltimes\hO_{X,Z}$ and the image of the map
\begin{eqnarray*}
	&&D':TX\rightarrow\C[W'']\ltimes(K(\hO_{X,Z})\otimes_{\cO(X)}D_\omega(X)),\\
	&&D'(v)=v+\sum_{(Z',s)\in S\atop Z'\supseteq Z}\frac{2c(Z',s)}{1-\lambda_{Z',s}}\zeta_{Z'}(v)(s-1).
\end{eqnarray*}
Since $\Der_\C(\hO_{X,Z})\cong\hO_{X,Z}\otimes_{\cO(X)}TX$, we may extend the $\cO(X)$-linear maps $\zeta_{Z'}:TX\rightarrow\cO(X)\langle Z'\rangle$ to $\hO_{X,Z}$-linear maps
\[
	\hat\zeta_{Z'}:\Der_\C(\hO_{X,Z})\rightarrow\hO_{X,Z}\otimes_{\cO(X)}\cO(X)\langle Z'\rangle\subseteq K(\hO_{X,Z}),
\]
and the above formula then extends $D'$ to an $\hO_{X,Z}$-linear map
\[
	\hat D':\Der_\C(\hO_{X,Z})\rightarrow\C[W']\ltimes(K(\hO_{X,Z})\otimes_{\cO(X)}D_\omega(X)).
\]
Since $A$ contains $\hO_{X,Z}$ and the image of $D'$, it also contains the image of $\hat D'$. Given $s\in W'$, let $I_s\subseteq\hO_{X,Z}$ denote the ideal generated by $\act f-f$ for $f\in\hO_{X,Z}$. Let $Z'$ denote the 
component of $X^s$ containing $Z$, and let $I_{Z'}\subseteq\cO(X)$ be the ideal vanishing on $Z'$. Then $I_s=\hO_{X,Z}I_{Z'}$, so $(Z',s)\in S$ exactly when $I_s$ is locally principal. Moreover if $f\in\Gamma(U,\cO_X)$ generates 
$I_{Z'}$ on some open subset $U\subseteq X$, then it generates $I_s$ on the corresponding open subset of $\Spec\hO_{X,Z}$, and
\[
	\hat\zeta_{Z'}(v)\in\frac{v(f)}{f}+\hO_{X,Z}\otimes_{\cO(X)}\Gamma(U,\cO_X)
\]
for any $v\in\Der_\C(\hO_{X,Z},\hO_{X,Z})$. This formula determines $\hat\zeta_{Z'}(v)$ up to an element of $\hO_{X,Z}$. Thus $\hat D'$ is determined, up to a map $\Der_\C(\hO_{X,Z})\rightarrow\C[W']\ltimes\hO_{X,Z}$, 
by the action of $W'$ on $\hO_{X,Z}$ and the parameters $c(Z',s)$ for $Z'\supseteq Z$. These data therefore determine the subalgebra
\[
	H_\co(W',\Spf\hO_{X,Z})\subseteq\C[W']\ltimes(K(\hO_{X,Z})\otimes_{\cO(X)}D_\omega(X))
\]
generated by $\C[W']\ltimes\hO_{X,Z}$ and the image of $\hat D'$. In particular, $H_\co(W',\Spf\hO_{X,Z})$ is preserved by conjugation by $W''$, so $A=\C[W'']\otimes_{\C[W']}H_\co(W',\Spf\hO_{X,Z})$, as required.

Finally consider any $H_\co(W,X)$-module $M$. The previous proposition gives an action of $\hO_{X,WZ}\otimes_{\cO(X)}H_\co(W,X)$ on $\hO_{X,WZ}\otimes_{\cO(X)}M$, so
\[
	\hO_{X,Z}\otimes_{\cO(X)}M=e\hO_{X,WZ}\otimes_{\cO(X)}M
\]
admits an action of $e\hO_{X,WZ}\otimes_{\cO(X)}H_\co(W,X)e$. But $\phi(e)$ is the monomial matrix with a $1$ in the $(1,1)$ entry, so restricting $\phi$ gives an isomorphism
\[
	e\hO_{X,WZ}\otimes_{\cO(X)}H_\co(W,X)e\cong\C[W'']\ltimes H_\co(W',\Spf\hO_{X,Z}),
\]
thus giving the required action. Note that this does not depend on the choice of coset representatives $C$, since we always take $C$ to contain $1$.
\end{proof}
We may now prove our first main result. In this proof we allow $X$ to not be affine, and our class $\psi\in H^2(X,\Omega_X^{\geq1})$ may not be represented by a global 2-form.
\begin{proof}[Proof of Theorem \ref{coh}]
\begin{enumerate}
\item
	Since being a submodule is a local property, we may suppose $X$ is affine and that $\psi$ is represented by $\omega\in(\Omega_X^{2,cl})^W$. Consider the module of global sections $M=\Gamma(X,\cM)\in H_\co(W,X){\rm-mod}$. Since 
	$Z$ is $W$-invariant, it is clear that $\Gamma_Z(M)$ is preserved by $\C[W]\ltimes\cO(X)$. It suffices to show that $D_v$ preserves $\Gamma_Z(M)$ for each $v\in TX$. Let $I\subseteq\cO(X)$ denote the ideal vanishing on $Z$. Recall 
	that $[D_v,\cO(X)]\subseteq\C[W]\ltimes\cO(X)$. Since $WI=IW$, it follows inductively that $I^{k+1}D_v\subseteq H_\co I^k$ for each $k\geq0$. Therefore if $m\in\Gamma_Z(M)$, then $I^km=0$ for some $k$, whence $I^{k+1}D_vm=0$, so 
	$D_vm\in\Gamma_Z(M)$.
\item
	Consider $Y\in P$ and $x\in Y$, with stabiliser $W'\subseteq W$. Let $\h=T_xX/T_xX^W$. Applying Proposition \ref{linear} to $Y$, there is a $W'$-invariant affine open neighbourhood $U$ of $x$, with 
	$U\cap Y$ closed in $U$, and a $W'$-equivariant isomorphism
	\[
		\phi:\hO_{(U\cap Y)\times\h,U\cap Y}\isom\hO_{U,U\cap Y},
	\]
	where we identify $U\cap Y$ with $(U\cap Y)\times\{0\}\subseteq(U\cap Y)\times\h$. Moreover $\phi$ induces the identity on $\cO(U\cap Y)$. Let $I$ be the kernel of the map 
	$\hO_{(U\cap Y)\times\h,U\cap Y}\twoheadrightarrow\cO(U\cap Y)$, let $C$ be a set of left coset representatives for $W'$ in $W$, and let
	\[
		\bar U=\coprod_{w\in C}wU.
	\]
	We have a natural \'etale morphism $\bar U\rightarrow X$, so by Propositions \ref{local} and \ref{linearH}, for any $\cM\in\cH_\cco\coh$ we have a natural action of $H_\co(W',\Spf\hO_{U,U\cap Y})$ on
	\[
		M=\hO_{U,U\cap Y}\otimes_{\cO(U)}\Gamma(U,\cM),
	\]
	where $\omega\in(\Omega_U^{2,cl})^W$ represents $\psi$ on $U$. The isomorphism $\phi$ gives rise to an isomorphism
	\[
		\phi_\co:H_{c,\phi^*\omega}(W',\Spf\hO_{(U\cap Y)\times\h,U\cap Y})\isom H_\co(W',\Spf\hO_{U,U\cap Y}).
	\]
	Let $\nu$ denote the pullback of $i_{U\cap Y}^*\omega$ to $(U\cap Y)\times\h$ under the projection map $(U\cap Y)\times\h\rightarrow U\cap Y$. By abuse of notation, we will also use $\nu$ to denote the completed map
	\[
		\Der_\C(\hO_{(U\cap Y)\times\h,U\cap Y})\otimes_\C\Der_\C(\hO_{(U\cap Y)\times\h,U\cap Y})\rightarrow\hO_{(U\cap Y)\times\h,U\cap Y}.
	\]
	Then
	\[
		\nu(v,v')-(\phi^*\omega)(v,v')\in I
	\]
	for vector fields $v,\,v'\in\Der_\C(\hO_{(U\cap Y)\times\h,U\cap Y})$ which preserve $I$. It follows that $\nu-\phi^*\omega=d\alpha$ for some 1-form
	\[
		\alpha:\Der_\C(\hO_{(U\cap Y)\times\h,U\cap Y})\rightarrow\hO_{(U\cap Y)\times\h,U\cap Y}
	\]
	which satisfies $\alpha(v)\in I$ whenever $v\in\Der_\C(\hO_{(U\cap Y)\times\h,U\cap Y})$ preserves $I$. Since we are working over characteristic 0, we may suppose $\alpha$ is $W'$-invariant. This gives an 
	isomorphism
	\[
		\alpha_\co:H_{c,\nu}(W',\Spf\hO_{(U\cap Y)\times\h,U\cap Y})\isom H_{c,\phi^*\omega}(W',\Spf\hO_{(U\cap Y)\times\h,U\cap Y}).
	\]
	Proposition \ref{linearH} also gives a natural isomorphism
	\[
			H_{c,\nu}(W',\Spf\hO_{(U\cap Y)\times\h,U\cap Y})\cong\hO_{(U\cap Y)\times\h,U\cap Y}\otimes_{\cO(U\cap Y)\otimes\C[\h]}H_{c,\nu}(W',(U\cap Y)\times\h),
	\]
	and $H_{c,\nu}(W',(U\cap Y)\times\h)=D_{i_{U\cap Y}^*\omega}(U\cap Y)\otimes H_c(W',\h)$. Here for any $s\in W'$ acting by reflection on $\h$, we take $c(s)=c(Z,s)$, where $Z$ is the component of $X^s$ containing $Y$. 
	Composing these isomorphisms, we obtain an action of $D_{i_{U\cap Y}^*\omega}(U\cap Y)$ on $M$ commuting with the action of $\phi(\C[\h])\subseteq\hO_{U,U\cap Y}$. Therefore $D_{i_{U\cap Y}^*\omega}(U\cap Y)$ acts on
	\[
		M/\phi(\h^*)M=M/\phi(I)M=\Gamma(U\cap Y,i_Y^*\cM).
	\]
	We will show that the latter action is independent of the choices of $\phi$ and $\alpha$, thus proving that the action is natural and patches to give an action on all of $Y$.
	
	Recall from the proof of Proposition \ref{linearH} that $H_\co(W',\Spf\hO_{U,U\cap Y})$ is a subalgebra of $\C[W']\ltimes(K(\hO_{U,U\cap Y})\otimes_{\cO(U\cap Y)}D_\omega(U))$. The latter also contains a natural 
	copy of $\Der_\C(\hO_{U,U\cap Y})$. We constructed an $\hO_{U,U\cap Y}$-linear map
	\[
		\hat D':\Der_\C(\hO_{U,U\cap Y})\rightarrow H_\co(W',\Spf\hO_{U,U\cap Y}).
	\]
	If $v\in\Der_\C(\hO_{U,U\cap Y})$ is $W'$-invariant then
	\[
		\hat D'(v)\in v+\C[W']\ltimes\hO_{U,U\cap Y}\subseteq\C[W']\ltimes(K(\hO_{U,U\cap Y})\otimes_{\cO(U\cap Y)}D_\omega(U)).
	\]
	In particular $v\in H_\co(W',\Spf\hO_{U,U\cap Y})$. Moreover by choosing $\hat\zeta_Z$ appropriately, we may ensure $\hat D'(v)\in v+\C[W']\ltimes\phi(I)$ for any such $v$.
	
	Consider a vector field $v$ on $U\cap Y$. We may extend $v$ naturally to a vector field on $(U\cap Y)\times\h$, and therefore a derivation of $\hO_{(U\cap Y)\times\h,U\cap Y}$. Let $\bar v$ denote the pushforward of this 
	derivation under $\phi$. This has the following properties:
	\begin{enumerate}
	\item $\bar v\in\Der_\C(\hO_{U,U\cap Y})$ is $W'$ invariant.
	\item
		The following diagram commutes:
		\[\xymatrix{
		\hO_{U,U\cap Y}\ar[r]^-{\bar v}\ar@{->>}[d]&\hO_{U,U\cap Y}\ar@{->>}[d]\\
		\cO(U\cap Y)\ar[r]^-{v}&\cO(U\cap Y).
		}\]
	\end{enumerate}
	As noted above, the first property ensures $\bar v\in H_\co(W',\Spf\hO_{U,U\cap Y})$. The action of $v$ on $M$ constructed above is exactly the action of
	\[
		\bar v+\phi\alpha(v)\in H_\co(W',\Spf\hO_{U,U\cap Y})\subseteq\C[W']\ltimes(K(\hO_{U,U\cap Y})\otimes_{\cO(U\cap Y)}D_\omega(U)).
	\]
	However, $\phi\alpha(v)\in\phi(I)$, so $v$ acts on $M/\phi(I)M$ as simply $\bar v$. Therefore it suffices to prove that the action of $\bar v$ on $M/\phi(I)M$ is determined by the above two properties. If $\bar v'$ also 
	satisfies these properties, then
	\[
		(\bar v-\bar v')(\hO_{U,U\cap Y})\subseteq\ker(\hO_{U,U\cap Y}\rightarrow\cO(U\cap Y))=\phi(I).
	\]
	That is, $\bar v-\bar v'\in\phi(I)\Der_\C(\hO_{U,U\cap Y})$. Since $\bar v$ and $\bar v'$ are fixed by $W'$, this gives
	\begin{eqnarray*}
		\bar v-\bar v'
			&\in&\hat D'(\bar v-\bar v')+\C[W']\ltimes\phi(I)\\
			&\subseteq&\phi(I)\hat D'(\Der_\C(\hO_{U,U\cap Y}))+\phi(I)\C[W']\\
			&\subseteq&\phi(I)H_\co(W',\Spf\hO_{U,U\cap Y}).
	\end{eqnarray*}
	Thus $\bar v-\bar v'$ acts as zero on $M/\phi(I)M$, as required.
\item
	It is well known that a coherent sheaf with a connection is locally free, so the previous part shows that each $Y\in P$ is either contained in or disjoint with $\supp\cM$. Since
	\[
		X=\coprod_{Y\in P}Y,
	\]
	we conclude that $\supp\cM$ is a disjoint union of sets in $P$. Moreover $\supp\cM$ is closed and the closure $\overline Y$ of any $Y\in P$ is irreducible, so the irreducible components of $\supp\cM$ 
	have the form $\overline Y$ for some $Y\in P$. Let $P_{\cM}$ denote the set of $Y\in P$ such that $\overline Y$ is an irreducible component of $\supp\cM$. The action of $W$ on $M$ ensures that $P_{\cM}$ is $W$-invariant, 
	so it suffices to prove that $P_{\cM}\subseteq P'$.
	
	Pick $Y\in P_{\cM}$ and let $x$, $W'$, $U$, $M$, $\h$, $\phi$ and $I$ be as above. Suppose $x\in\overline Y'$ for some $Y'\in P_{\cM}$. Then $\overline Y'$ is a connected component of $X^{W''}$ for some $W''\subseteq W$, and we 
	must have $W''\subseteq\stab_W(x)=W'$. But then $\overline Y\subseteq X^{W''}$, so $\overline Y\subseteq\overline Y'$ since $\overline Y$ is connected. Since $\overline Y$ is an irreducible component of $\supp\cM$, we conclude 
	that $Y'=Y$. That is, $\supp\cM$ coincides with $Y$ on some neighbourhood of $x$. By shrinking $U$, we suppose that this holds on $U$. Then some power of $\phi(I)$ kills $M$. 
	It follows that
	\[
		N=k_x\otimes_{\cO(U\cap Y)}M
	\]
	is finite dimensional, where $k_x$ is the residue field of the point $x\in U\cap Y$, and the map $\cO(U\cap Y)\rightarrow\hO_{U,U\cap Y}$ is given by the map $\phi$. Moreover the action 
	of $D_{i_{U\cap Y}^*\omega}(U\cap Y)\otimes H_c(W',\h)$ on $M$ gives rise to an action of $H_\co(W',\h)$ on $N$. Finally $N$ is nonzero, since
	\[
		\C[\h]/\h^*\C[\h]\otimes_{\C[\h]}N
	\]
	is the fibre of $\cM$ at $x\in X$, which is nonzero by assumption.
\item
	We keep the above notation. Since some power of $I$ kills $M$, and the ring $\hO_{(U\cap Y)\times\h,U\cap Y}/I^l$ is finitely generated over $\cO(U\cap Y)$ for any $l$, we conclude that $M$ is a finitely 
	generated module over $\cO(U\cap Y)$. Since it admits a connection, it is locally free. We will show in Lemma \ref{nilp} that there is an integer $K_Y$, depending only on $\h$, $W'$ and $c$, such that 
	$N$ is killed by $(\h^*)^{K_Y}$. That is,
	\[
		I_Y^{K_Y}M\subseteq\m_xM,
	\]
	where $I_Y\subseteq\cO(U)$ is the ideal vanishing on $U\cap Y$, and $\m_x\subseteq\cO(U\cap Y)$ is the maximal ideal corresponding to the point $x$. Up to (non-canonical) isomorphism, the algebra $H_\co(W',\h)$ is 
	independent of the point $x\in Y$. Therefore this equation holds for any $x\in Y\cap U$. Together with local freeness, this ensures that $I_Y^{K_Y}M=0$. Since $\supp\cM$ coincides with $Y$ on $U$, we conclude that 
	$\cI^{K_Y}\cM$ vanishes on $U$, where $\cI\subseteq\cO_X$ is the ideal sheaf vanishing on $\supp\cM$. Now $U$ was chosen to contain an arbitrary point on $Y$, and $\cI^{K_Y}$ is $W$-invariant, so $\cI^{K_Y}\cM$ vanishes on 
	the union $WY$ of all $W$-translates of $Y$. That is, $\cI^{K_Y}\cM\subseteq\Gamma_Z(M)$, where $Z$ is the complement of $WY$ in $\supp\cM$. Note that $Z$ is closed and $W$-invariant. It now follows by induction on $\supp\cM$ 
	that $M$ is killed by $\cI$ to the power of
	\[
		\sum_{Y\in P'\atop Y\subseteq\supp\cM}K_Y.
	\]
	In particular, this proves the statement with $K=\sum_{Y\in P'}K_Y$.
\item
	Let $K$ be the integer constructed in the previous part. Again we prove the statement by induction on $\supp\cM$, and the case $\supp\cM=\emptyset$ is trivial. Suppose every module with smaller support has finite 
	length. Choose $Y\in P_{\cM}$, and let $\cI_Y\subseteq\cO_X$ denote the ideal sheaf vanishing on $\overline Y$. Let $U\subseteq X$ be some open affine subset intersecting $Y$, and 
	let $I_Y=\Gamma(U,\cI_Y)$. This is prime since $Y$ is irreducible. Consider the ring
	\[
		R=\cO(U)_{(I_Y)}/I_Y^K\cO(U)_{(I_Y)},
	\]
	that is, the $K^{\rm th}$ formal neighbourhood of the (non-closed) generic point of $Y$. Then $I_Y^kR/I_Y^{k+1}R$ is finite dimensional over the field $R/I_YR$ for each $k$, so $R$ is Artinian. Since $\cM$ is 
	coherent, it follows that
	\[
		R\otimes_{\cO(U)}\Gamma(U,\cM)
	\]
	has finite length over $R$. We may therefore assume by induction that the statement also holds for modules with the same support, for which the above $R$-module has smaller length.
	
	Again let $Z$ be the complement of $WY$ in $\supp\cM$. Let $\cI_Z$ and $\cI$ be the ideal sheaves vanishing on $Z$ and $\supp\cM$ respectively, Let $\cM'=\cM/\Gamma_Z(\cM)$ and $\cM''=\cH_\cco\cI_Z^K\cM'$. Since $\Gamma_Z(\cM)$ 
	and $\cM'/\cM''$ are supported on $Z$, they have finite length by induction. If $\cM''$ is zero or irreducible, we are done. Suppose otherwise, and let $\cN'$ be a proper nonzero submodule of $\cM''$, and let $\cN$ be the 
	inverse image of $\cN'$ in $\cM$. We have $\cI_Y\cI_Z\subseteq\cI$, so
	\[
		\cI_Z^K\cI_Y^K\cM\subseteq\cI^K\cM=0.
	\]
	Therefore $\cI_Y^K\cM\subseteq\Gamma_Z(\cM)$, so $\cI_Y^K\cM'=0$ and the map
	\[
		\cO(U)_{(I_Y)}\otimes_{\cO(U)}\Gamma(U,\cM')\rightarrow R\otimes_{\cO(U)}\Gamma(U,\cM')
	\]
	is an isomorphism. Since localisation is exact, we conclude that
	\[
		R\otimes_{\cO(U)}\Gamma(U,\cN')\rightarrow R\otimes_{\cO(U)}\Gamma(U,\cM')
	\]
	is injective. Therefore we have a short exact sequence
	\[
		R\otimes_{\cO(U)}\Gamma(U,\cN')\hookrightarrow R\otimes_{\cO(U)}\Gamma(U,\cM')\twoheadrightarrow R\otimes_{\cO(U)}\Gamma(U,\cM'/\cN').
	\]
	We claim that the first and last modules are nonzero. This is equivalent to $\supp\cN'$ and $\supp(\cM'/\cN')$ containing $Y$. Suppose the first fails. Then $\cN'$ is supported on $Z$, so $\cI_Z^K\cN'=0$. Thus
	\[
		\cI_Z^{2K}\cN\subseteq\cI_Z^K\Gamma_Z(\cM)=0.
	\]
	Hence $\cN\subseteq\Gamma_Z(\cM)$, so that $\cN'=0$, a contradiction. Now suppose $\cM'/\cN'$ is supported on $Z$. Then $\cI_Z^K(\cM'/\cN')=0$, so $\cN'\supseteq\cI_Z^K\cM'$. Thus $\cN'\supseteq\cH_\cco\cI_Z^K\cM'=\cM''$, 
	contradicting the assumption that $\cN'$ is a proper submodule of $\cM''$. This proves the claim, and we conclude that
	\[
		\len_R(R\otimes_{\cO(U)}\Gamma(U,\cN')),\,\len_R(R\otimes_{\cO(U)}\Gamma(U,\cM'/\cN'))<\len_R(R\otimes_{\cO(U)}\Gamma(U,\cM)),
	\]
	so $\cN'$ and $\cM'/\cN'$ have finite length by induction. Again, since $\Gamma_Z(\cM)$ has finite length, we are done.
\item
	Again let $Y$ be any element of $P_{\cM}$ and let $WY$ and $Z$ be as above. Let $\cI_{WY}$, $\cI_Z$ and $\cI$ be the ideal sheaves in $\cO_X$ vanishing on $\overline{WY}$, $Z$ and $\supp\cM$ respectively. Then
	\[
		\cI_Z^K\cI_{WY}^K\cM\subseteq\cI^K\cM=0.
	\]
	Thus $\cI_{WY}^K\cM\subseteq\Gamma_Z(\cM)$. Now $\Gamma_Z(\cM)$ is a submodule of $\cM$ by (1), and it is proper since $Z\subsetneq\supp\cM$. Since $\cM$ is irreducible, we conclude that $\Gamma_Z(\cM)=0$. 
	Hence $\cI_{WY}^K\cM=0$, so $\supp\cM=\overline{WY}$ as required.
\end{enumerate}
\end{proof}
\section{Linear Actions}\label{Sec:linear}
Now suppose $X=\h$ is a finite dimensional vector space with a linear action. We briefly review some known results concerning representations of $H_c(W,\h)$; see \cite{DunklOpdam}. We also prove Proposition \ref{RS} 
and one direction of Theorem \ref{suppSnpart}.
\subsection{Verma modules}
Any $W$-module $\tau$ becomes a $\C[W]\ltimes\C[\h^*]$-module by declaring that $\h$ acts as $0$. We may therefore construct an $H_c$-module
\[
	M(\tau)=H_c\otimes_{\C[W]\ltimes\C[\h^*]}\tau.
\]
This is the \emph{Verma module} corresponding to $\tau$. The multiplication map
\[
	\C[\h]\otimes_\C\C[W]\otimes_\C\C[\h^*]\rightarrow H_c
\]
is a vector space isomorphism, so $M(\tau)\cong\C[\h]\otimes_\C\tau$ as a $\C[W]\ltimes\C[\h]$-module. As a special case, when $\tau$ is the trivial representation, we obtain an action of $H_c$ on 
$\C[\h]$; this is called the \emph{polynomial representation}. If $\tau$ is irreducible, then there is a unique maximal proper submodule $J(\tau)$ of $M(\tau)$, and the quotient $L(\tau)$ is irreducible. Let $y_i$ be a basis of $\h$, 
and $x_i$ the dual basis of $\h^*$. Define the \emph{Euler element} by
\[
	\eu=\sum_ix_iy_i+\sum_{s\in S}c(s)\frac{2}{\lambda_s-1}s\in H_c.
\]
This element has the useful property that $\ad\eu$ acts as $0$ on $W$, as $1$ on $\h^*\subseteq H_c$, and as $-1$ on $\h\subseteq H_c$. From this fact we deduce the following lemma, which was used in the proof of Theorem \ref{coh}.
\begin{lemm}\label{nilp}
There exists a positive integer $K$, depending only on $c$ and $W$, such that $\h^K\subseteq H_c$ and $(\h^*)^K\subseteq H_c$ annihilate any finite dimensional $H_c$-module.
\end{lemm}
\begin{proof}
Let $M$ be a finite dimensional $H_c$-module. Then $M$ decomposes as
\[
	M=\bigoplus_{\lambda\in\Lambda}M_\lambda,
\]
where $M_\lambda$ is the generalised eigenspace of $\eu$ with eigenvalue $\lambda$, and $\Lambda\subseteq\C$ is the finite set of eigenvalues. Since $\ad\eu$ acts as $-1$ on $\h$, it is clear that $\h$ sends 
$M_\lambda$ to $M_{\lambda-1}$. Thus $\h^K$ kills $M$, where $K$ is any integer larger than $d=\max(\re(\Lambda))-\min(\re(\Lambda))$. The same is true of $\h^*$, and it remains to bound $d$ independently of $M$. 
Pick $\lambda\in\Lambda$ with $\re(\lambda)$ minimal. Then $\h$ acts as $0$ on $M_\lambda$. Since $W$ commutes with $\eu$, it preserves the eigenspace $M_\lambda$. We may therefore find a subspace $\tau\subseteq M_\lambda$ which 
is irreducible under the action of $W$. Then
\[
	\eu|_\tau=\left.\left(\sum_ix_iy_i+\sum_{s\in S}c(s)\frac{2}{\lambda_s-1}s\right)\right|_\tau
		=\sum_{s\in S}c(s)\frac{2}{\lambda_s-1}s|_\tau.
\]
This depends only on the action of $W$ on $\tau$. Since $W$ has only finitely many irreducible modules, there are only finitely many possible values for $\lambda$, once $c$ and $W$ have been chosen. Similarly 
by writing
\[
	\eu=\sum_iy_ix_i-\dim\h+\sum_{s\in S}c(s)\left(\frac{2}{\lambda_s-1}+2\right)s,
\]
we see that there are only finitely many possibilities for $\max(\re(\Lambda))$. Thus $d$ has only finitely many possible values, depending on $W$ and $c$, and may therefore be bounded independently of $M$.
\end{proof}
For a module $M\in H_c\coh$, the following conditions are equivalent:
\begin{enumerate}
\item The action of $\eu$ on $M$ is locally finite.
\item The action of $\h$ on $M$ is locally nilpotent.
\item Every composition factor of $M$ is isomorphic to some $L(\tau)$.
\end{enumerate}
Proposition \ref{RS} states that the category of modules satisfying these conditions is exactly the category $H_c\rs$ of Definition \ref{RSdefn}. We will require the following lemma to prove this.
\begin{lemm}\label{monodromic}
Let $\h$ be a finite dimensional vector space, and suppose $Z\subseteq\h$ is the zero set of a homogeneous ideal in $\C[\h]$ (that is, $Z$ is a cone). Let $\xi\in D(\h\setminus Z)$ denote the Euler vector field. Then $\xi$ 
acts locally finitely on the global sections of any $\cO$-coherent $\D$-module on $\h\setminus Z$ with regular singularities.
\end{lemm}
\begin{proof}
Let $\cM$ be an $\cO$-coherent $\D$-module on $\h\setminus Z$ with regular singularities. Then $\cM$ is locally free, so if $U$ is an open subset of $\h\setminus Z$, the restriction map
\[
	\Gamma(\h\setminus Z,M)\rightarrow\Gamma(U,M)
\]
is injective. We may therefore replace $\h\setminus Z$ by any smaller $\C^*$-invariant open subset $U$. Denoting by $x_0,\ldots,x_n$ the coordinates on $\h$, we suppose that $U$ is affine and disjoint 
with the zero set of $x_0$. We have an isomorphism
\[
	\C[\h][x_0^{-1}]\isom\C[t,t^{-1}]\otimes\C[y_1,\ldots,y_n]
\]
given by $x_0\mapsto t$ and $x_i\mapsto ty_i$ for $i>0$. Note the $\C^*$ action on the left, which scales each $x_i$, corresponds to the $\C^*$ action on the right scaling only $t$. Thus we have a 
$\C^*$-equivariant isomorphism
\[
	U\cong\C^*\times Y
\]
where $Y$ is some affine open subset of $\Spec\C[y_1,\ldots,y_n]$. In particular the vector field $\xi$ on the left corresponds to $t\partial_t$ on $\C^*$.

It therefore suffices to consider a module $M$ over $D(\C^*\times Y)=D(\C^*)\otimes_\C D(Y)$ which is $\cO$-coherent with regular singularities. Moreover we may suppose that 
$M$ is irreducible. The Riemann-Hilbert correspondence \cite{Deligne} implies that $M$ is of the form $L\otimes_\C N$ for some irreducible modules $L\in D(\C^*)$-mod and $N\in D(Y)$-mod, and that 
$L=\C[t,t^{-1}]v$ with connection
\[
	\nabla_{\partial_t}f(t,t^{-1})v=(\partial_tf(t,t^{-1}))v+\lambda t^{-1}f(t,t^{-1})v
\]
for some $\lambda\in\C$. Since $t\partial_t$ acts as $n+\lambda$ on $t^nv$, and $\{t^nv\mid n\in\Z\}$ is a basis for $L$, we are done.
\end{proof}
\begin{proof}[Proof of Proposition \ref{RS}]
First we show each $L(\tau)$ lies in $H_c\rs$. Choose $Y\in P_{L(\tau)}$, and let $i_Y:Y\hookrightarrow X$ denote the inclusion. We are required to show that the connection on $i_Y^*Sh(L(\tau))$ has regular singularities. 
Certainly we have a surjection $i_Y^*Sh(M(\tau))\twoheadrightarrow i_Y^*Sh(L(\tau))$ intertwining the connections, so it suffices to prove that the connection on $i_Y^*Sh(M(\tau))$ has regular singularities. However,
\[
	\Gamma(Y,i_Y^*Sh(M(\tau)))\cong\cO(Y)\otimes_{\C[\h]}\C[\h]\otimes_\C\tau=\cO(Y)\otimes_\C\tau.
\]
Moreover $\tau$ is killed by $\h$, so by Proposition \ref{flatconn}, the connection on $i_Y^*Sh(M)$ is described by
\[
	\nabla_ym=-\sum_{s\in S\setminus W'}c(s)\langle y,\alpha_s\rangle\frac{2}{1-\lambda_s}\frac{1}{\alpha_s}(s-1)m
\]
for $m\in\tau$ and $y\in\hwp$, where $W'$ is the stabiliser of any point in $Y$. Since this expression only contains poles of first order, the connection has regular singularities. Since $H_c\rs$ is a Serre 
subcategory of $H_c\coh$ by definition, this proves any module satisfying condtion (3) above lies in $H_c\rs$.

Conversely, we will show that any $M\in H_c\rs$ satisfies condition (1) above. This condition is preserved by extensions, and Theorem \ref{coh}(5) shows that $M$ has finite length, so we may suppose $M$ is 
irreducible. As usual let $Y$ be in $P_M$ with stabiliser $W'$. We may find a homogeneous polynomial $f\in\C[\h]$ which vanishes on each $Y'\in P_M\setminus\{Y\}$, but not on $Y$ itself. Note that the 
kernel of the map
\[
	M\rightarrow\bigoplus_{w\in W}M[(\act[w]f)^{-1}]
\]
is $\Gamma_Z(M)$, where $Z$ is the common zero set of all the $\act[w]f$. This is zero by Theorem \ref{coh}(1) and the irreducibility of $M$. Since $\ad\eu$ acts as $\deg f$ on each $\act[w]f$, the action 
of $\eu$ on $M$ extends naturally to one on $M[(\act[w]f)^{-1}]$. Therefore since $\eu$ is $W$-invariant, it suffices to show that the action of $\eu$ on $M[f^{-1}]$ is locally finite.

Let $U\subseteq\h$ denote the affine open subset on which $f$ is nonzero, and let $I_Y\subseteq\C[\h][f^{-1}]$ denote the ideal vanishing on $U\cap Y$. Note that $I_Y$ is generated by some subspace 
$V\subseteq\h^*$, so $I_Y^kM[f^{-1}]$ is invariant under $\eu$ for each $k$. Moreover since $M[f^{-1}]$ is supported on $U\cap Y$, we have $I_Y^KM[f^{-1}]=0$ for some $K>0$. Therefore 
it suffices to show that $\eu$ acts locally finitely on each $I_Y^kM[f^{-1}]/I_Y^{k+1}M[f^{-1}]$. Finally for each $k$ we have a surjective map
\[
	V^{\otimes k}\otimes_\C M[f^{-1}]/I_YM[f^{-1}]\twoheadrightarrow I_Y^kM[f^{-1}]/I_Y^{k+1}M[f^{-1}]
\]
intertwining $1\otimes\eu$ with $\eu-k\deg f$. We may therefore consider just $k=0$. But
\[
	M[f^{-1}]/I_YM[f^{-1}]=\cO(U\cap Y)\otimes_{\C[\h]}M
\]
is an $\cO$-coherent $\D$-module on $U\cap Y$ with connection given by Proposition \ref{flatconn}. It has regular singularities by assumption. Note that $U\cap Y$ is the basic open subset of the vector space $\h^{W'}$ 
on which $f|_{\h^{W'}}$ is nonzero. Therefore the vector field $\xi$ of Lemma \ref{monodromic} acts locally finitely. To describe the action of $\xi$ explicitly, let $\{x_1,\ldots,x_n\}$ and $\{y_1,\ldots,y_n\}$ be dual bases for 
$\h^*$ and $\h$, such that $\{y_1,\ldots,y_r\}$ span $\h^{W'}$. Then $x_{r+1},\ldots,x_n$ are zero in $\cO(U\cap Y)$, and a straighforward calculation shows that
\[
	\xi m=\eu\,m+\sum_{s\in S\setminus W'}\frac{2c(s)}{1-\lambda_s}m+\sum_{s\in S\cap W'}\frac{2c(s)}{1-\lambda_s}sm
\]
for $m\in M$. Note that $\cO(U\cap Y)\otimes_{\C[\h]}M$ admits an action of $W'$ commuting with $\cO(U\cap Y)$. Since $\xi$ and $\eu$ have the same commutator with any element of $\cO(U\cap Y)$, 
this formula holds for any $m\in\cO(U\cap Y)\otimes_{\C[\h]}M$. Therefore since $\xi$ and $\C[W']$ act locally finitely, so does $\eu$, as required.
\end{proof}
\subsection{Characters of modules}
Let $\ozt$ denote the space of formal $\Z$-linear combinations of powers of $t$, such that the exponents appearing belong to $A+\Z_{\geq0}$ for some finite subset $A\subseteq\C$. Let $K(\wfd)$ denote the 
Grothendieck group of the category of finite dimensional representations of $W$. There is a homomorphism
\[
	Ch:K(H_c(W,\h)\rs)\rightarrow K(\wfd)\otimes_\Z\ozt
\]
sending $[M]$ to
\[
	\sum_{\lambda\in\C}[M_\lambda]t^\lambda,
\]
where $M_\lambda\subseteq M$ is the generalised $\lambda$-eigenspace of $\eu$, considered as a $W$-module. If $\tau$ is an irreducible $W$-module, then
\[
	Ch([M(\tau)]),\,Ch([L(\tau)])\in[\tau]t^{h(\tau)}+K(\wfd)\otimes\Z[[t]]t^{h(\tau)+1},
\]
where $h(\tau)$ is the scalar by which
\[
	\sum_{s\in S}c(s)\frac{2}{\lambda_s-1}s
\]
acts on $\tau$. It follows that $Ch$ is injective, and both $\{[L(\tau)]\}$ and $\{[M(\tau)]\}$ form bases for $K(H_c(W,\h)\rs)$. Moreover the matrix relating the $[M(\tau)]$ to $[L(\tau)]$ is upper triangular, when the 
irreducibles are ordered by $\re(h(\tau))$. It follows that the functor
\begin{eqnarray*}
	Verma_W:\wfd&\rightarrow&H_c(W,\h)\rs,\\
	\tau&\mapsto&M(\tau)
\end{eqnarray*}
induces an isomorphism on Grothendieck groups.
\subsection{Induction and restriction}
Bezrukavnikov and Etingof \cite{BE} also construct ``parabolic induction and restriction" functors for rational Cherednik algebras. Statements (1-4) of the following theorem summarise Propositions 3.9, 3.10 and 3.14 of 
\cite{BE}. Statement (5) follows from the construction of $Res_b$ and Proposition 2.21 of \cite{GGOR}, and was communicated to the author by Etingof.
\begin{thm}\label{resind}
Consider a point $b\in\h$, with stabiliser $W'\subseteq W$. There exist exact functors $Res_b:H_c(W,\h)\rs\rightarrow H_c(W',\h)\rs$ and $Ind_b:H_c(W',\h)\rs\rightarrow H_c(W,\h)\rs$ with the 
following properties:
\begin{enumerate}
\item The functor $Ind_b$ is right adjoint to $Res_b$.
\item The support of $Res_b(M)$ is the union of the components of $\supp M$ passing through $b$.
\item The support of $Ind_b(N)$ is the union of $W$-translates of $\supp N$.
\item
	The induced maps $[Res_b]$ and $[Ind_b]$ on Grothendieck groups satisfy
	\begin{eqnarray*}
		[Res_b][Verma_W]&=&[Verma_{W'}][Res],\\{}
		[Ind_b][Verma_{W'}]&=&[Verma_W][Ind],
	\end{eqnarray*}
	where $Res:\wfd\rightarrow\wfd[']$ and $Ind:\wfd[']\rightarrow\wfd$ are the usual restriction and induction functors.
\item
	If $M$ is a Verma module, then $Res_bM$ has a filtration whose successive quotients are Verma modules.
\end{enumerate}
\end{thm}
\begin{rem}\end{rem}
We will use a definition of $Res_b$ that differs slightly from that in \cite{BE} in two ways. Firstly, in the notation of \cite{BE}, we do not include $\zeta$, so $Res_b$ produces modules in $H_c(W',\h)\rs$ 
rather than $H_c(W',\h/\h^{W'})\rs$. Secondly, Theorem 3.2 of \cite{BE} includes a shift from $b$ to the origin, which we omit. In both cases the functor we have omitted is an equivalence, so the 
properties of $Res_b$ are unchanged.

By its construction, $Res_b$ behaves well with respect to monodromy. We give two results to this effect. The first, concerning monodromy around a single hyperplane, follows easily from \cite{BE}.
\begin{prop}\label{resmon}
Suppose $b,\,v\in\h$ have stabilisers $W',\,W''\subseteq W$ respectively. Suppose that $W''\subseteq W'$, and that $C\subseteq W'$ is a subgroup acting faithfully on $\C v$. Consider the map 
$\phi:\C\rightarrow\h$ given by $\phi(z)=b+zv$. There is a Zariski open subset $U\subseteq\C^*$ such that $\phi$ maps $U$ into $Y=\h^{W''}_{\rm reg}$. Let $i_Y:Y\hookrightarrow\h$ denote the inclusion. 
For any $M\in H_c(W,\h)\rs$, the $\D$-modules $i_Y^*Sh(M)$ and $i_Y^*Sh(Res_bM)$ satisfy
\[
	\C((z))\otimes_{\cO_U}\phi|_U^*i_Y^*Sh(M)\cong\C((z))\otimes_{\cO_U}\phi|_U^*i_Y^*Sh(Res_bM)
\]
as $\C[C]\ltimes\C((z))[\partial_z]$-modules. In particular, the monodromies about the origin of the $\D$-modules $\phi|_U^*i_Y^*Sh(M)$ and $\phi|_U^*i_Y^*Sh(Res_bM)$ are conjugate, 
and the same is true when these equivariant $\D$-modules are pushed down to $U/C$.
\end{prop}
To prove the next result, we required the following simple algebraic geometry lemma.
\begin{lemm}\label{cohnonaff}
Suppose $U$ is an open subset (not necessarily affine) of an affine reduced Noetherian scheme $\Spec A$. Suppose $\cM$ is a coherent sheaf on $U$ such that the support of any nonzero $m\in\Gamma(U,\cM)$ contains some irreducible 
component of $U$. Finally suppose $\Gamma(U\cap Z,\cO_Z)$ is finitely generated over $A$ for each irreducible component $Z$ of $\Spec A$. Then $\Gamma(U,\cM)$ is finitely generated over $A$.
\end{lemm}
\begin{proof}
By Exercise II.5.15 of \cite{Hartshorne}, there is a finitely generated $A$-module $N$ such that $\cM\cong Sh(N)|_U$. Let $P$ denote the set of irreducible components of $U$. For $Y\in P$, let $\overline Y$ denote the closure of $Y$ 
in $\Spec A$. This is an irreducible component of $\Spec A$, and is therefore an integral affine Noetherian scheme with the reduced induced scheme structure. Let $K(Y)$ be the field of fractions of $\cO(\overline Y)$, and let
\[
	K=\bigoplus_{Y\in P}K(Y).
\]
Let $N'$ be the kernel of $N\rightarrow K\otimes_AN$. If $n\in N'$, then the support of $n$ does not contain any $Y\in P$. Therefore the image of $n$ in $\Gamma(U,Sh(N))=\Gamma(U,\cM)$ is zero, so $n$ is supported on 
$\Spec A\setminus U$. Thus
\[
	Sh(N/N')|_U\cong Sh(N)|_U\cong\cM.
\]
Now $K\otimes_AN$ is finitely generated over $K$, so it is a direct sum of finite dimensional vector spaces over the $K(Y)$. We may write
\[
	K\otimes_AN=\bigoplus_{i=1}^mK(Y_i)x_i,
\]
where $x_i\in K\otimes_AN$ and $Y_i\in P$ (where $Y_i$ may equal $Y_j$ for $i\neq j$). Since $N$ is finitely generated over $A$, we may choose the $x_i$ so that
\[
	N/N'\subseteq\bigoplus_{i=1}^m\cO(\overline{Y_i})x_i.
\]
Therefore
\[
	\Gamma(U,\cM)=\Gamma(U,Sh(N/N'))\subseteq\bigoplus_{i=1}^m\Gamma(Y_i,\cO_{\overline{Y_i}}).
\]
By assumption, the right hand side is finitely generated over $A$. Since $A$ was assumed to be Noetherian, the result follows.
\end{proof}
We can now prove our second result relating $Res_b$ and monodromy.
\begin{prop}\label{resmonglobal}
Suppose $b\in\h$ with $\stab_W(b)=W'$. Suppose $M\in H_c(W,\h)\rs$ is scheme theoretically supported on $\supp M$. Also suppose that for each $s\in S\setminus W'$, and for each subgroup $W''\subseteq W'$ for 
which $\h^{W''}$ is an irreducible component of $\supp M$, the $\D$-module $M|_{\hr^{W''}}$ has trivial monodromy around $Z(\alpha_s)\cap\hr^{W''}$. Then for each such $W''$, the $N_{W'}(W'')$-equivariant $\D$-module 
$M|_{\hr^{W''}}$ is exactly the restriction of the $N_{W'}(W'')$-equivariant $\D$-module $Res_bM|_{\hbr^{W''}}$, where
\[
	\hbr^{W''}=\{y\in\h\mid\stab_{W'}(y)=W''\}.
\]
\end{prop}
\begin{proof}
Let
\[
	U=\{y\in\h\mid\stab_W(y)\subseteq W'\}.
\]
In particular, if $s\in S\setminus W'$ then $\alpha_s$ is nonzero on $U$. The action $\rho:H_c(W,\h)\rightarrow\End_\C(M)$ naturally induces an action
\[
	\rho':H_c(W',U)\rightarrow\End_\C(\cO(U)\otimes_{\cO(\h)}M)
\]
such that
\[
	\rho'(y)(1\otimes m)=1\otimes ym-\sum_{s\in S\setminus W'}\frac{2c(s)}{1-\lambda_s}\langle y,\alpha_s\rangle\frac1{\alpha_s}\otimes(s-1)m
\]
for $y\in\h$ and $m\in M$. According to Proposition \ref{flatconn} and the above construction, $M$ and $\cO(U)\otimes_{\cO(\h)}M$ induce the same $N_{W'}(W'')$-equivariant $\D$-modules on $\hr^{W''}$. We will 
construct an $N\in H_c(W',\h)\rs$ such that
\[
	\cO(U)\otimes_{\cO(\h)}N\cong\cO(U)\otimes_{\cO(\h)}M
\]
as $H_c(W',U)$-modules. This implies that
\[
	\hO_{\h,b}\otimes_{\cO(\h)}N\cong\hO_{\h,b}\otimes_{\cO(U)}(\cO(U)\otimes_{\cO(\h)}M)
\]
as $H_c(W',\Spf\hO_{\h,b})$-modules. By definition of $Res_b$, this ensures that $Res_bM\cong N$. Since $\cO(U)\otimes_{\cO(\h)}N\cong\cO(U)\otimes_{\cO(\h)}M$, the $N_{W'}(W'')$-equivariant $\D$-modules on $\hr^{W''}$ 
induced by $\cO(U)\otimes_{\cO(\h)}M$ is the restriction of that on $\hbr$ induced by $N$. The result follows.

It remains to construct $N$. Let $P$ be the set of subgroups $W''$ of $W'$ for which $\h^{W''}$ is an irreducible component of $\supp M$. Let
\[
	Z=\bigcup_{W''\in P}\h^{W''}.
\]
Any component of $\supp M$ not contained in $Z$ is contained in $\h\setminus U$, so the support of $\cO(U)\otimes_{\cO(\h)}M$ is $Z\cap U$. By (1) of Theorem \ref{coh}, we may replace $M$ by a quotient of $M$ and suppose that the 
support of any nonzero $m\in\Gamma(U,Sh(M))$ intersects $\hr^{W''}$ for some $W''\in P$. Since $M|_{\hr^{W''}}$ is locally free, any such support would then contain $\hr^{W''}$.

By assumption, if $W''\in P$ then the coherent $\D$-module $M|_{\hr^{W''}}$ has trivial monodromy around each hyperplane in $\hbr^{W''}\setminus\hr^{W''}$. Therefore there is a coherent $D$-module $N(W'')$ on 
$\hbr^{W''}$ whose restriction to $\hr^{W''}$ is $M|_{\hr^{W''}}$. Patching these with $\cO(U)\otimes_{\cO(\h)}M$, we obtain a coherent sheaf $\cN$ on the open subset
\[
	U'=(U\cap Z)\cup\bigcup_{W''\in P}\hbr^{W''}
\]
of $Z$. The support of any nonzero $n\in N=\Gamma(U',\cN)$ contains $\hr^{W''}$ for some $W''\in P$, and therefore contains some component of $U'$. Moreover for $W''\in P$,
\[
	\h^{W''}\setminus U'=\bigcup_{w_1\in W\setminus W'\atop w_2\in W'\setminus W''}\{y\in\h^{W''}\mid w_1y=y=w_2y\}
\]
has codimension at least 2 in $\h^{W''}$. We may therefore apply Lemma \ref{cohnonaff} to conclude that $N$ is finitely generated over $\cO(Z)$. We may identify $N$ with the submodule of $\cO(U)\otimes_{\cO(\h)}M$ consisting 
of elements whose restriction to $\hr^{W''}$ has no pole around $Z(\alpha_s)$, for each $W''\in P$ and $s\in S\setminus W''$. This condition is preserved by the action of $H_c(W',\h)$, so $N\in H_c(W',\h)\rs$. Finally
\[
	\cO(U)\otimes_{\cO(\h)}N=\Gamma(U\cap Z,\cN)=\cO(U)\otimes_{\cO(\h)}M,
\]
as required.
\end{proof}
\subsection{Type {A}}
Now let $W=S_n$, the symmetric group $S_n$ on $n$ letters, acting on $\h=\C^n$ by permuting coordinates. The reflections in this case are transpositions. As they are all conjugate, the function 
$c:S\rightarrow\C$ must be constant, and we identify it with its value in $\C$. Let $\h/\C$ denote the quotient of $\h$ by the line fixed by $W$. In this case we have the following simple criterion for when 
$H_c(W,\h/\C)$ admits a finite dimensional representation.
\begin{thm}[\cite{BEG2} Theorem 1.2]\label{finite}
Suppose $n>1$. The algebra $H_c(W,\h/\C)$ admits a nonzero finite dimensional representation if and only if $c=\frac rn$ for some integer $r$ coprime with $n$. In this case the category of finite dimensional modules is semisimple 
with one irreducible. Moreover if $c=\frac1n$, this irreducible is one dimensional.
\end{thm}
Using this with Theorem \ref{coh}(3) gives the following (see Example 3.25 of \cite{BE}).
\begin{thm}\label{suppSn}
Suppose $c=\frac{r}{m}$, where $r$ and $m$ are integers with $m$ positive and coprime with $r$. For each nonnegative integer $q\leq n/m$, let
\[
	X_q'=\left\{b\in\h\,\left|\,b_i=b_j\text{ whenever }\left\lceil\frac{i}{m}\right\rceil=\left\lceil\frac{j}{m}\right\rceil\leq q\right.\right\}
	\hspace{5mm}\text{and}\hspace{5mm}
	X_q=\bigcup_{w\in W}wX_q'.
\]
Then any module in $H_c(W,\h)\coh$ is supported on one of the $X_q$. If $c$ is irrational then every such module has full support.
\end{thm}
We would like to determine more explicitly which irreducibles in $H_c\rs$ have which support sets. The irreducible representations of $W=S_n$ are well known to be parameterised by partitions of $n$. 
Given a partition $\lambda\vdash n$, the corresponding irreducible is denoted $\tau_\lambda$ and called the \emph{Specht module} indexed by $\lambda$. We will represent a partition $\lambda\vdash n$ as a 
nonincreasing sequence of nonnegative integers, $\lambda=(\lambda_1,\lambda_2,\ldots,\lambda_k)$, whose sum is $n$, where two sequences are identified if their nonzero entries agree. If $\lambda\vdash n$ 
and $\mu\vdash m$, we may obtain a partition $\lambda+\mu=(\lambda_1+\mu_1,\lambda_2+\mu_2,\ldots)$ of $n+m$. We first prove a lemma about the induction functor introduced in Theorem \ref{resind}.
\begin{lemm}\label{indSn}
Suppose $c$ is a positive real number, and suppose we have partitions $\lambda\vdash n$ and $\mu\vdash m$. Let $b$ be a point in $\C^{n+m}$ whose stabiliser in $S_{n+m}$ is $S_n\times S_m$. Then 
$Ind_b(L(\tau_\lambda)\otimes L(\tau_\mu))$ admits a nonzero map from $M(\tau_{\lambda+\mu})$.
\end{lemm}
\begin{proof}
We first prove

\textbf{Claim 1}: the lowest order term in $Ch[Ind_b(M(\tau_\lambda)\otimes M(\tau_\mu))]$ is $[\tau_{\lambda+\mu}]t^{h(\tau_{\lambda+\mu})}$.

Indeed, the construction of Verma modules shows that $M(\tau_\lambda)\otimes M(\tau_\mu)\cong M(\tau_\lambda\otimes\tau_\mu)$, so Theorem \ref{resind}(4) implies that
\[
	Ch[Ind_b(M(\tau_\lambda)\otimes M(\tau_\mu))]=Ch[M(Ind(\tau_\lambda\otimes\tau_\mu))].
\]
The Littlewood-Richardson rule describes how $Ind(\tau_\lambda\otimes\tau_\mu)$ splits into Specht modules for $S_{n+m}$. In particular, given partitions $\alpha$ and $\beta$ of $n+m$, say $\alpha$ \emph{dominates} 
$\beta$, denoted $\alpha\geq\beta$, if
\[
	\sum_{i=1}^p\alpha_i\geq\sum_{i=1}^p\beta_i
\]
for all $p\geq1$. Then
\[
	[Ind(\tau_\lambda\otimes\tau_\mu)]=[\tau_{\lambda+\mu}]+\sum_{\nu\leq\lambda+\mu}c^\nu_{\lambda\mu}[\tau_\nu]
\]
for some coefficients $c^\nu_{\lambda\mu}\in\Z$. Certainly $Ch[M(\tau_{\lambda+\mu})]$ has the required lowest order term, so it suffices to prove that $h(\tau_\alpha)<h(\tau_\beta)$ whenever $\alpha>\beta$. Recall 
that $h(\tau_\alpha)$ is the action of
\[
	-c\sum_{i\neq j}s_{ij}
\]
on $S(\alpha)$. By the Frobenious character formula (see Exercise 4.17(c) of \cite{FultonHarris}), this is
\[
	h(\tau_\alpha)=-c\sum_{i\geq1}\frac{1}{2}\alpha_i^2-\left(i-\frac{1}{2}\right)\alpha_i.
\]
Using the Abel summation formula,
\begin{eqnarray*}
	h(\tau_\alpha)-h(\tau_\beta)
		&=&-c\sum_{i\geq1}(\alpha_i-\beta_i)\left(\frac{\alpha_i+\beta_i+1}{2}-i\right)\\
		&=&-c\sum_{i\geq1}\left(\sum_{j=1}^i(\alpha_j-\beta_j)\right)
			\left(\frac{\alpha_i-\alpha_{i+1}+\beta_i-\beta_{i+1}}{2}+1\right)\\
		&<&0,
\end{eqnarray*}
if $\alpha>\beta$. Here we have used that $\sum_{j=1}^i(\alpha_j-\beta_j)=0$ for $i$ sufficiently large. This proves Claim 1. From this we will deduce

\textbf{Claim 2}: the lowest order term in $Ch[Ind_b(L(\tau_\lambda)\otimes M(\tau_\mu))]$ is $[\tau_{\lambda+\mu}]t^{h(\tau_{\lambda+\mu})}$.

We prove Claim 2 by descending induction on $h(\lambda)$; that is, suppose it holds for all pairs $(\nu,\mu)$ with $h(\nu)>h(\lambda)$. Of course this assumption is vacuous when $h(\lambda)$ 
is maximal, so we need not prove the base case. We have
\[
	[M(\tau_\lambda)]=[L(\tau_\lambda)]+\sum_{h(\nu)>h(\lambda)}d_\lambda^\nu[L(\tau_\nu)]
\]
for some nonnegative integers $d_\lambda^\nu$. Tensoring by $M(\tau_\mu)$ and applying $Ind_b$ and $Ch$, we obtain
\begin{eqnarray*}
	Ch[Ind_b(M(\tau_\lambda)\otimes M(\tau_\mu))]\hspace{-30mm}\\
		&=&Ch[Ind_b(L(\tau_\lambda)\otimes M(\tau_\mu))]+\sum_{h(\nu)>h(\lambda)}d_\lambda^\nu Ch[Ind_b(L(\tau_\nu)\otimes M(\tau_\mu))].
\end{eqnarray*}
By Claim 1, the lowest order term in this expression is $[\tau_{\lambda+\mu}]t^{h(\tau_{\lambda+\mu})}$. However, each summand on the right has nonnegative integer coefficients. Therefore this 
must be the lowest order term of one of the summands. For $h(\nu)>h(\lambda)$, the lowest order term of $Ch[Ind_b(L(\tau_\nu)\otimes M(\tau_\mu))]$ is $[\tau_{\nu+\mu}]t^{h(\tau_{\nu+\mu})}$ by 
induction. 
Clearly if $\nu\neq\lambda$ then $\nu+\mu\neq\lambda+\mu$, so the term $[\tau_{\lambda+\mu}]t^{h(\tau_{\lambda+\mu})}$ can only come from $Ch[Ind_b(L(\tau_\lambda)\otimes M(\tau_\mu))]$, as required. We 
now conclude

\textbf{Claim 3}: the lowest order term in $Ch[Ind_b(L(\tau_\lambda)\otimes L(\tau_\mu))]$ is $[\tau_{\lambda+\mu}]t^{h(\tau_{\lambda+\mu})}$.

Indeed this follows from Claim 2 using the same argument by which Claim 2 followed from Claim 1. This proves that the $h(\tau_{\lambda+\mu})$-eigenspace of $\eu$ in $Ind_b(L(\tau_\lambda)\otimes L(\tau_\mu))$ 
is isomorphic to $\tau_{\lambda+\mu}$ as an $S_{n+m}$-module, and is killed by $\h\subseteq H_c(S_{n+m},\h)$. We therefore have a nonzero $\C[S_{n+m}]\ltimes\C[\h^*]$-module homomorphism
\[
	\tau_{\lambda+\mu}\rightarrow Ind_b(L(\tau_\lambda)\otimes L(\tau_\mu)),
\]
where $\h$ is defined to act as $0$ on $\tau_{\lambda+\mu}$. By definition of the Verma module, we obtain a nonzero map
\[
	M(\tau_{\lambda+\mu})\rightarrow Ind_b(L(\tau_\lambda)\otimes L(\tau_\mu)),
\]
as required.
\end{proof}
Now suppose $c=\frac rm$ where $r$ is coprime with $m$. We say $\lambda$ is \emph{$m$-regular} if no part of $\lambda$ appears $m$ or more times. Any partition $\lambda\vdash n$ can be written uniquely as $\lambda=m\mu+\nu$ 
such that $\nu'$ is $m$-regular. Let $q_m(\lambda)=|\mu|$. More explicitly
\[
	q_m(\lambda)=\sum_{i\geq1}i\left\lfloor\frac{\lambda_i-\lambda_{i+1}}{m}\right\rfloor.
\]
We will eventually show, in Theorem \ref{suppSnpart}, that the support of $L(\tau_\lambda)$ is $X_{q_m(\lambda)}$ for $c>0$ and $X_{q_m(\lambda')}$ for $c<0$. For the moment we prove one direction:
\begin{thm}\label{suppSnpartle}
With $c=\frac rm$ and $\h=\C^n$ as above, the support of the $H_c$-module $L(\tau_\lambda)$ is contained in $X_{q_m(\lambda)}$ if $c>0$, and contained in $X_{q_m(\lambda')}$ if $c<0$.
\end{thm}
\begin{proof}
We denote $X_q$ by $X_q^n$ throughout this proof, as we will be considering support sets for other values of $n$. Suppose $c>0$.

Let $q=q_m(\lambda)$, and write $\lambda=m\mu+\nu$ for some partitions $\mu\vdash q$ and $\nu\vdash n-qm$, such that $q_m(\nu)=0$. By Proposition 9.7 and Theorem 9.8 of \cite{CEE}, 
the support of $L(\tau_{m\mu})$ is contained in $X_q^{qm}$. Thus the support of $L(\tau_{m\mu})\otimes L(\tau_\nu)$ is contained in $X_q^{qm}\times\C^{n-qm}\subseteq X_q^n$. By Theorem \ref{resind}(3), 
the same is true of $Ind_b(L(\tau_{m\mu})\otimes L(\tau_\nu))$, where $b\in\C^n$ is a point whose stabiliser in $S_n$ is $S_{qm}\times S_{n-qm}$. By Lemma \ref{indSn}, we have a nonzero map
\[
	\phi:M(\tau_\lambda)\rightarrow Ind_b(L(\tau_{m\mu})\otimes L(\tau_\nu)).
\]
Now $L(\tau_\lambda)$ is the only irreducible quotient of $M(\tau_\lambda)$, so the image of $\phi$ admits $L(\tau_\lambda)$ as a quotient. Thus $L(\tau_\lambda)$ is a subquotient of 
$Ind_b(L(\tau_{m\mu})\otimes L(\tau_\nu))$, so its support must be contained in $X_q^n$. This proves the $c>0$ case.

There is an automorphism of $\C[W]$ sending $s\in S$ to $-s$. Twisting by this automorphism sends $\tau_\lambda$ to $\tau_{\lambda'}$. Moreover it extends to an isomorphism $H_c(W,\h)\rightarrow H_{-c}(W,\h)$, which is the 
identity on $\h$ and $\h^*$. Therefore the statement for $c<0$ follows from that for $c>0$.
\end{proof}
Finally we will require the following classification of two-sided ideals in $H_c(W,\h)$, due to Losev.
\begin{thm}[\cite{Losev} Theorem 4.3.1 and \cite{Stoica} Theorem 5.10]\label{twosided}
There are $\lfloor n/m\rfloor+1$ proper two-sided ideals of $H_c(W,\h)$,
\[
	0=J_0\subsetneq J_1\subsetneq\ldots\subsetneq J_{\lfloor n/m\rfloor}\subsetneq H_c(W,\h).
\]
Moreover if $c>0$ then the polynomial representation admits a filtration
\[
	0=I_0\subsetneq I_1\subsetneq\ldots\subsetneq I_{\lfloor n/m\rfloor}\subsetneq\C[\h]
\]
such that $Z(I_q)=X_q$ and $\Ann_{H_c(W,\h)}(\C[\h]/I_q)=J_q$. Each $I_q$ is radical if and only if $c=\frac1m$.
\end{thm}
\section{Minimal Support for Type A}\label{Sec:minsupp}
In this section we consider the algebra $H_c=H_\frac{1}{m}(S_n,\C^n)$, and study modules in $H_c\rs$ whose support is the smallest possible set, namely $X_q$ where 
$q=\lfloor n/m\rfloor$. In particular, we will show that the full subcategory of such modules is semisimple. We begin with a general lemma concerning the localisation functor for linear actions.
\begin{lemm}\label{loc}
Consider a finite group $W$ acting linearly on a finite dimensional vector space $\h$. Suppose $\alpha\in\C[\h]^W$ is a symmetric polynomial, and let $U\subseteq\h$ denote the open set on which $\alpha$ is nonzero. Moreover 
suppose no nonzero module in $H_c(W,\h)\rs$ is supported on the zero set of $\alpha$. Then the localisation functor
\[
	L:H_c(W,\h)\coh\rightarrow H_c(W,U)\coh
\]
identifies $H_c(W,\h)\rs$ with a full subcategory of $H_c(W,U)\coh$ closed under subquotients.
\end{lemm}
\begin{proof}
Let $\mathcal A$ denote the full subcategory of $H_c(W,U)\coh$ consisting of modules $M$ such that every $m\in M$ is killed by $\h^n\alpha^n$ for some $n$. This is clearly closed under subquotients.

Certainly $L$ is exact, since $\C[\h][\alpha^{-1}]$ is flat over $\C[\h]$, and its image lies in $\mathcal A$. Conversely suppose $M\in\mathcal A$, and let $V\subseteq M$ be a finite dimensional subspace generating $M$ over 
$\C[\h][\alpha^{-1}]$. Multiplying $V$ by some power of $\alpha$, we may suppose that $\h$ acts nilpotently on $V$. Since $H_c(W,\h)=\C[\h]\otimes_\C\C[W]\otimes_\C\C[\h^*]$, the $H_c(W,\h)$-submodule $N$ of $M$ generated by $V$ 
is finitely generated over $\C[\h]$, and $\h$ acts locally nilpotently on $N$. Thus $N\in H_c(W,\h)\rs$. Moreover $M$ is generated by $N$ over $\C[\h][\alpha^{-1}]$, so $M\cong L(N)$ is in the image of the localisation functor. 
Therefore $\mathcal A$ is exactly the image of $L$.

Define the functor $E:H_c(W,U){\rm-mod}\rightarrow H_c(W,\h){\rm-mod}$ sending $M$ to the subspace $E(M)\subseteq M$ on which $\h$ acts nilpotently. Note that $E(M)$ is stable under the action of $H_c(W,\h)$. 
We will show that if $M\in H_c(W,\h)\rs$, then $EL(M)$ is naturally isomorphic to $M$. The kernel of the natural map $M\rightarrow L(M)$ is exactly the submodule $\Gamma_{Z(\alpha)}(M)$ defined in Theorem \ref{coh}(1). This is 
zero since we have assumed no modules are supported on $Z(\alpha)$. We may therefore identify $M$ with a submodule of $EL(M)$. Consider any $v\in EL(M)$. As above, the $H_c(W,\h)$-submodule $N$ of $L(M)$ generated 
by $v$ is in $H_c(W,\h)\rs$. Thus $M'=M+N\subseteq L(M)$ is also in $H_c(W,\h)\rs$, and $L(M')\cong L(M)$. Using the exactness of localisation, we conclude that $L(M'/M)=0$, so that $M'/M$ is supported on $Z(\alpha)$. Again 
this implies that $M'/M=0$, so $v\in M$. Thus $M$ is exactly $EL(M)$. We have shown above that the localisation functor $L:H_c(W,\h)\rs\rightarrow\mathcal A$ is essentially surjective, so $E$ and $L$ are mutually inverse functors.
\end{proof}
We also require the following well-known result.
\begin{lemm}\label{HeH}
Suppose a unital associative algebra $H$ contains an idempotent $e$ such that $HeH=H$. Then the functors
\[
	H{\rm-mod}\rightarrow eHe{\rm-mod},\hspace{20mm}M\mapsto eM
\]
and
\[
	eHe{\rm-mod}\rightarrow H{\rm-mod},\hspace{20mm}N\mapsto He\otimes_{eHe}N
\]
are mutually inverse equivalences.
\end{lemm}
Now consider the algebra $H_c=H_\frac{1}{m}(S_n,\C^n)$. Let
\[
	\h=\C^n=\Spec\C[x_1,\ldots,x_n],
\]
where $S_n$ permutes the $x_i$. Let $q=\lfloor n/m\rfloor$ and write $n=qm+p$, so that $0\leq p<m$. Let $\h'=\C^{q+p}$, with coordinate ring 
$\C[\h']=\C[z_1,\ldots,z_q,\,t_1,\ldots,t_p]$, on which $\Sqp$ acts naturally. Define
\[
	\pi:\C[\h]\twoheadrightarrow\C[\h']
\]
by
\[
	\pi(x_i)=\begin{cases}
		z_{\lceil i/m\rceil}&\text{if }i\leq qm\\
		t_{i-qm}&\text{if }i>qm.
	\end{cases}
\]
This identifies $\h'$ with one of the components of $X_q\subseteq\h$. The map $\pi$ restricts to a map $\pi:\C[\h]^{S_n}\rightarrow\C[\h']^{\Sqp}$. Unfortunately this is not surjective if $p>0$. We therefore localise 
as follows. Let $[n]=\{1,\,2,\ldots,n\}$, and let
\[
	\alpha_d=\sum_{J\subseteq[n]\atop|J|=p}\left(\sum_{j\in J}x_j^d\right)\prod_{j\in J,\,i\notin J}(x_i-x_j).
\]
Clearly $\alpha_d$ is symmetric. Moreover using that $p<m$, it can be shown that
\[
	\pi(\alpha_d)=\left(\sum_{j=1}^rt_j^d\right)\prod_{1\leq i\leq q\atop1\leq j\leq p}(z_i-t_j)^m.
\]
Let $U\subseteq\h$ and $U'\subseteq\h'$ denote the affine open subsets on which $\alpha_0$ and $\pi(\alpha_0)$ are nonzero, respectively. Let $A=\cO(U)^{S_n}$ and $B=\cO(U')^{\Sqp}$, 
and let $\phi:A\rightarrow B$ be the map induced by $\pi$. Now $B$ is generated as a $\C$-algebra by $\pi(\alpha_0)^{-1}$, $\sum_{j=1}^pt_j^d$ and $\sum_{i=1}^qz_i^d$. We have
\[
	\sum_{j=1}^pt_j^d=\phi\left(p\frac{\alpha_d}{\alpha_0}\right)
	\hspace{10mm}\text{and}\hspace{10mm}
	\sum_{i=1}^qz_i^d=\phi\left(\frac{1}{m}\sum_{i=1}^nx_i^d-\frac{p}{m}\frac{\alpha_d}{\alpha_0}\right),
\]
so $\phi$ is surjective. Moreover since $X_q$ is the union of the $S_n$ translates of $\h'$, the kernel of $\phi$ is $(\cO(U)I_q)^{S_n}$, where $I_q$ is the ideal vanishing on $X_q$.

Now consider the idempotents
\begin{eqnarray*}
	e&=&\frac{1}{n!}\sum_{w\in S_n}w\in\C[S_n]\subseteq H_c(S_n,\h)\subseteq H_c(S_n,U),\\
	e'&=&\frac{1}{q!p!}\sum_{w\in\Sqp}w\in\C[\Sqp]\subseteq H_{(m,c)}(\Sqp,\h')\subseteq H_{(m,c)}(\Sqp,U),
\end{eqnarray*}
where $(m,c)$ indicates the function which takes the value $m$ on transpositions in $S_q$ and $c$ on those in $S_p$. For convenience, we omit the subscripts $c$ and $(c,m)$ for the rest of this section. 
It is known (see Corollary 4.2 of \cite{BE}) that $H(S_n,\h)eH(S_n,\h)=H(S_n,\h)$, so that $H(S_n,U)eH(S_n,U)=H(S_n,U)$, and $M\mapsto eM$ is an equivalence of categories from $H(S_n,U)$-mod to 
$eH(S_n,U)e$-mod by Lemma \ref{HeH}. Similar results hold for $e'$. In particular the faithful action of $H(\Sqp,U')$ on $\cO(U')$ gives a faithful action of $e'H(\Sqp,U')e'$ on $B$. Also 
$\cO(U)I_q$ is an $H(S_n,U)$-submodule of $\cO(U)$ by Theorem 5.10 of \cite{Stoica}, so $eH(S_n,U)e$ acts on $A/\ker\phi\cong B$.
\begin{prop}\label{imsigma}
The image of the homomorphism $\sigma:eH(S_n,U)e\rightarrow\End_\C(B)$ describing the above action is exactly $e'H(\Sqp,U')e'\subseteq\End_\C(B)$, and the kernel is generated by $eI_qe$.
\end{prop}
\begin{proof}
Let $\{\ch x_a\}$ and $\{\ch z_i,\ch t_j\}$ be the bases of $\h$ and $\h'$ dual to $\{x_a\}$ and $\{z_i,t_j\}$. For $d\geq0$, let
\[
	p_x(d)=\sum_{a=1}^nx_a^d,
\]
and similarly for $p_t,\,p_z,\,p_{\ch x},\,p_{\ch z},\,p_{\ch t}$. It is known that $eH(S_n,\h)e$ is generated as an algebra by $e\C[\h]e$ and $p_{\ch x}(2)$ (see the proof of Proposition 4.9 of \cite{EG}, 
and Corollary 4.9 of \cite{BEG1}). It follows that $eH(S_n,U)e$ is generated by $Ae$ and $p_{\ch x}(2)e$, and that
\[
	e'H(\Sqp,U')e'
\]
is generated by $Be'$, $p_{\ch z}(2)e'$ and $p_{\ch t}(2)e'$. For $f,\,g\in A$, we have
\[
	\sigma(fe)\phi(g)=\phi(fg)=\phi(f)\phi(g),
\]
so the image under $\sigma$ of $Ae\subseteq eH(S_n,U)e$ is $Be'\subseteq e'H(\Sqp,U')e'$. Next, recalling that $\phi:A\rightarrow B$ is the restriction of $\pi:\cO(U)\rightarrow\cO(U')$, we show that
\begin{equation}\label{partial}
	\pi(\partial_{\ch x_a}f)=\begin{cases}
		\frac{1}{m}\partial_{\ch z_i}\phi(f)&\text{if }\pi(x_a)=z_i,\\
		\partial_{\ch t_j}\phi(f)&\text{if }\pi(x_a)=t_j
	\end{cases}
	\hspace{10mm}\text{for }f\in A.
\end{equation}
Indeed, this holds when $f=p_x(d)$, since
\[
	\frac{1}{m}\partial_{\ch z_i}\phi(p_x(d))
		=\frac{1}{m}\partial_{\ch z_i}(mp_z(d)+p_t(d))=dz_i^{d-1}
\]
and
\[
	\partial_{\ch t_j}\phi(p_x(d))=\partial_{\ch t_j}(mp_z(d)+p_t(d))=dt_j^{d-1}.
\]
Now as functions of $f$, both sides of (\ref{partial}) are $\C$-derivations from $A$ to $\cO(U')$. Therefore (\ref{partial}) holds on the subring generated by the $p_x(d)$, namely $\C[\h]^{S_n}$, 
whence it holds on $\C[\h]^{S_n}[\alpha_0^{-1}]=A$. Now for $f\in A$ we have
\begin{eqnarray*}
	p_{\ch x}(2)f
		&=&\sum_{a=1}^n\ch x_a\partial_{\ch x_a}f\\
		&=&\sum_{a=1}^n\left(\partial_{\ch x_a}^2f-c\sum_{b\neq a}\frac{\partial_{\ch x_a}f-s_{ab}\partial_{\ch x_a}f}{x_a-x_b}\right)\\
		&=&\triangle_xf-c\sum_{a\neq b}\frac{\partial_{\ch x_a}f-\partial_{\ch x_b}f}{x_a-x_b},
\end{eqnarray*}
where $\triangle_x$ denotes the Laplacian in the variables $x_a$. Similar formulae hold for the actions of $p_{\ch z}(2)$ and $p_{\ch t}(2)$ 
on $B$. Define $P\in e'H(\Sqp,U')e'$ by
\[
	P=\frac{1}{m}p_{\ch z}(2)e'+p_{\ch t}(2)e'-2\sum_{1\leq i\leq q\atop1\leq j\leq p}\frac{1}{z_i-t_j}\left(\frac{1}{m}\ch z_i-\ch t_j\right)e'.
\]
Note that
\[
	\sum_{1\leq i\leq q\atop1\leq j\leq p}\frac{1}{z_i-t_j}\left(\frac{1}{m}\ch z_i-\ch t_j\right)
\]
is symmetric under $\Sqp$, so final sum is an element of $e'H(\Sqp,U')e'$, though the individual summands are not. We claim that
\[
	\phi(p_{\ch x}(2)f)-P\phi(f)=0
\]
for $f\in A$. We first show that the left hand side is a derivation. Indeed
\begin{eqnarray*}
	\phi(p_{\ch x}(2)(fg))
		&=&\phi\left(fp_{\ch x}(2)g+gp_{\ch x}(2)f+\sum_a(\partial_{\ch x_a}f)(\partial_{\ch x_a}g)\right)\\
		&=&\phi(f)\phi(p_{\ch x}(2)g)+\phi(g)\phi(p_{\ch x}(2)f)\\
		&&{}+\sum_i\frac{1}{m}(\partial_{\ch z_i}\phi(f))(\partial_{\ch z_i}\phi(g))
			+\sum_j(\partial_{\ch t_j}\phi(f))(\partial_{\ch t_j}\phi(g)),
\end{eqnarray*}
using (\ref{partial}). Also
\begin{equation}\label{[P,f]}
	P(fg)=fP(g)+gP(f)+\sum_i\frac{1}{m}(\partial_{\ch z_i}f)(\partial_{\ch z_i}g)
			+\sum_j(\partial_{\ch t_j}f)(\partial_{\ch t_j}g).
\end{equation}
Subtracting, we see that $\phi(p_{\ch x}(2)f)-P\phi(f)$ is indeed a derivation in $f$. By the same argument as above, it suffices to 
check the equation when $f=p_x(d)$. We calculate
\[
	\phi(p_{\ch x}(2)p_x(d))
		=\phi\left(d(d-1)p_x(d-2)-cd\sum_{a\neq b}\frac{x_a^{d-1}-x_b^{d-1}}{x_a-x_b}\right).
\]
If $\pi(x_a)=\pi(x_b)=z_i$, then
\[
	\pi\left(\frac{x_a^{d-1}-x_b^{d-1}}{x_a-x_b}\right)
		=\pi\left(\sum_{i=0}^{d-2}x_a^ix_b^{d-2-i}\right)
		=(d-1)z_i^{d-2}.
\]
Therefore recalling that $c=\frac1m$,
\begin{eqnarray*}
	\phi(p_{\ch x}(2)p_x(d))\hspace{-25mm}\\
		&=&d(d-1)(mp_z(d-2)+p_t(d-2))-cd(d-1)m(m-1)p_z(d-2)\\
		&&{}-cdm^2\sum_{i\neq i'}\frac{z_i^{d-1}-z_{i'}^{d-1}}{z_i-z_{i'}}
			-2cdm\sum_{i,j}\frac{z_i^{d-1}-t_j^{d-1}}{z_i-t_j}
			-cd\sum_{j\neq j'}\frac{t_j^{d-1}-t_{j'}^{d-1}}{t_j-t_{j'}}\\
		&=&\left[d(d-1)p_z(d-2)-md\sum_{i\neq i'}\frac{z_i^{d-1}-z_{i'}^{d-1}}{z_i-z_{i'}}\right]\\
		&&{}+\left[d(d-1)p_t(d-2)-cd\sum_{j\neq j'}\frac{t_j^{d-1}-t_{j'}^{d-1}}{t_j-t_{j'}}\right]
			-2d\sum_{i,j}\frac{z_i^{d-1}-t_j^{d-1}}{z_i-t_j}\\
		&=&p_{\ch z}(2)p_z(d)+p_{\ch t}(2)p_t(d)-2\sum_{i,j}\frac{\partial_{\ch z_i}p_z(d)-\partial_{\ch t_j}p_t(d)}{z_i-t_j}\\
		&=&P(mp_z(d)+p_t(d))\\
		&=&P\phi(p_x(d)),
\end{eqnarray*}
as required. Therefore $\im\sigma$ is the subalgebra of $\End_\C(B)$ generated by $Be'$ and $P$. Now (\ref{[P,f]}) shows that for $fe'\in Be'\subseteq e'H(\Sqp,U')e'$, we have
\[
	[P,fe']=P(f)e'+\sum_i\frac{1}{m}(\partial_{\ch z_i}f)\ch z_ie'
			+\sum_j(\partial_{\ch t_j}f)\ch t_je'.
\]
In particular,
\[
	[P,\pi(\alpha_0)e']=P(\pi(\alpha_0))e'+\sum_{1\leq i\leq q\atop1\leq j\leq p}\frac{m\pi(\alpha_0)}{z_i-t_j}\left(\frac{1}{m}\ch z_i-\ch t_j\right)e'.
\]
Therefore $\im\sigma$ contains
\[
	P+\frac{2}{m\pi(\alpha_0)}\left([P,\pi(\alpha_0)e']-P(\pi(\alpha_0))e'\right)=\frac{1}{m}p_{\ch z}(2)e'+p_{\ch t}(2)e'.
\]
We also have
\[
	[P,p_z(2)e']=P(p_z(2))e'+\frac{2}{m}\sum_iz_i\ch z_ie',
\]
so $\im\sigma$ contains
\[
	\left[\frac{1}{m}p_{\ch z}(2)e'+p_{\ch t}(2)e',\sum_iz_i\ch z_ie'\right]=\frac{2}{m}p_{\ch z}(2)e'.
\]
Thus $\im\sigma$ contains $Be'$, $p_{\ch z}(2)e'$ and $p_{\ch t}(2)e'$, so it is exactly $e'H(\Sqp,U')e'$.

Finally since $H(S_n,U)eH(S_n,U)=H(S_n,U)$, the two-sided ideals of the algebra $eH(S_n,U)e$ are in one to one correspondence with those in $H(S_n,U)$. The latter are determined by their inverse image 
in $H(S_n,\h)$. Clearly the kernel of $\sigma$ is proper and contains $eI_qe$, so by Theorem \ref{twosided} it is generated by $eI_qe$.
\end{proof}
We have identified $\h'$ with a subspace of $\h$, and the stabiliser of a generic point of $\h'$ in $S_n$ is a natural copy of $S_m^q\subseteq S_n$. Let $\hr'$ denote the elements of $\h'$ 
with exactly this stabiliser. Note that $\hr'$ is contained in $U$, since the zero set of
\[
	\pi(\alpha_0)=p\prod_{1\leq i\leq q\atop1\leq j\leq p}(z_i-t_j)^m
\]
consists of elements whose stabiliser in $S_n$ is at least as large as $S_{m+1}\times S_m^{q-1}$. Also note that the normaliser subgroup $N_{S_n}(S_m^q)\subseteq S_n$ acts on $\hr'$. Each coset of $S_m^q$ in 
$N_{S_n}(S_m^q)$ has a unique representative of minimal length, and these representatives form a subgroup isomorphic to $\Sqp$. The induced action of $\Sqp$ 
on $\h'=\Spec\C[z_1,\ldots,z_q,t_1,\ldots,t_p]$ is the natural one. We now use the homomorphism $\sigma$ to analyse the subcategory of $H(S_n,\h)\rs^q$ of minimally supported modules. Some of this 
information is generalised by Theorems \ref{rsqrsq+1} and \ref{suppSnpart}.
\begin{thm}\label{minsupp}
Let $H(S_n,\h)\rs^q$ be the full subcategory of $H(S_n,\h)\rs$, consisting of modules supported on $X_q$. There is an equivalence of categories
\[
	F:H(S_n,\h)\rs^q\rightarrow H(\Sqp,\h')\rs,
\]
such that $\cO(\hr')\otimes_{\C[\h]}M\cong\cO(\hr')\otimes_{\C[\h']}F(M)$ as $\C[\Sqp]\ltimes D(\hr')$-modules. In particular, $H(S_n,\h)\rs^q$ is semisimple, and its irreducibles are exactly 
the $L(\tau_\lambda)$ for which $q_m(\lambda)=q$.
\end{thm}
\begin{proof}
Let $eH(S_n,U)e\coh$ denote the full subcategory of $eH(S_n,U)e{\rm-mod}$ consisting of modules finitely generated over $e\C[\h]e$, and similarly for $H(\Sqp,U')$. 
Recall the equivalences $H(S_n,U)\coh\isom eH(S_n,U)e\coh$ and $H(\Sqp,U')\coh\isom e'H(\Sqp,U')e'\coh$; we denote the functors by $e$ and $e'$. By Lemma \ref{loc}, we also 
have localisation functors $L:H(S_n,\h)\rs\rightarrow H(S_n,U)\coh$ and $L':H(\Sqp,\h')\rs\rightarrow H(\Sqp,U')\coh$ which identify their domains with full subcategories 
of their codomains closed under subquotients. Finally the previous proposition gives a functor $\sigma^*$ identifying $e'H(\Sqp,U')e'\coh$ with a full subcategory of $eH(S_n,U)e\coh$, 
again closed under subquotients.
\[\xymatrix{
	H(S_n,\h)\rs\ar@{^{(}->}[r]^L&H(S_n,U)\coh\ar[r]^-e_-{\sim}&e'H(\Sqp,U')e'\coh\\
	H(\Sqp,\h')\rs\ar@{^{(}->}[r]^{L'}&H(\Sqp,U')\coh\ar[r]^-{e'}_-{\sim}&e'H(\Sqp,U')e'\coh\ar@{^{(}->}[u]^{\sigma^*}.
}\hspace{-20mm}\]
Note that these functors are all exact. We make the following claims about them:

\textbf{Claim 1:} If $M\in H(S_n,\h)\rs^q$, then the $\cO$-coherent $\D$-module $\cO(\hr')\otimes_{\C[\h]}M$ on $\hr'$ has regular singularities and trivial monodromy around each irreducible component 
of $Z(\alpha_0)\cap\h'$.

\textbf{Claim 2:} If $N\in H(\Sqp,U')\coh$ is such that the $\cO$-coherent $\D$-module $\cO(\hr')\otimes_{\cO(U')}N$ on $\hr'$ has regular singularities and trivial monodromy around each component of 
$Z(\alpha_0)\cap\h'$, then $N\cong L'(N')$ for some $N'\in H(\Sqp,\h')\rs$.

\textbf{Claim 3:} If $M\in H(S_n,U)\coh$ and $N\in H(\Sqp,U')\coh$ are such that $eM\cong\sigma^*e'N$, then the $\Sqp$-equivariant $\D$-modules $\cO(\hr')\otimes_{\cO(U)}M$ and $\cO(\hr')\otimes_{\cO(U')}N$ 
on $\hr'$ are isomorphic.

Let us first see how these results imply the statements of the theorem. The composites $eL$ and $\sigma^*e'L'$ identify $H(S_n,\h)\rs^q$ and $H(\Sqp,\h')\rs$ with full subcategories of $eH(S_n,U)e\coh$ 
closed under subquotients, which we call the images of $eL$ and $\sigma^*e'L'$.  Now consider any $M\in H(S_n,\h)\rs^q$. Theorem \ref{twosided} shows that $I_q$ kills $M$, so $eL(M)$ is killed by $eI_qe$, and therefore by the 
kernel of $\sigma$. Thus $eL(M)\cong\sigma^*e'N$ for some $N\in H(\Sqp,U')\coh$. By claim 3,
\[
	\cO(\hr')\otimes_{\cO(U')}N\cong\cO(\hr')\otimes_{\cO(U)}L(M)\cong\cO(\hr')\otimes_{\C[\h]}M
\]
as $D(\hr')$-modules. By claim 1, this $\D$-module has regular singularities and trivial monodromy around each component of $Z(\alpha_0)\cap\h'$. Therefore by claim 2, $N\cong L'(N')$ for some $N'\in H(\Sqp,\h')\rs$, 
so that $eL(M)\cong\sigma^*e'L'(N')$. This proves that the image of $eL$ is contained in that of $\sigma^*e'L'$, so there is an exact fully faithful embedding $F:H(S_n,\h)\rs^q\rightarrow H(\Sqp,\h')\rs$ such that 
$eL=\sigma^*e'L'F$, and whose image is closed under subquotients. Now $H(\Sqp,\h')\rs$ is semisimple and has $\bp_q\bp_p$ (isomorphism classes of) irreducibles, where $\bp_n$ is the number of partitions of $n$. On the other hand, 
by Theorem \ref{suppSnpartle}, the $\bp_q\bp_p$ distinct irreducibles
\[
	\{L(\tau_{m\lambda+\nu})\mid\lambda\vdash q,\,\nu\vdash p\}
\]
lie in $H(S_n,\h)\rs^q$. It follows that $F$ is an equivalence of categories, and that the above are all of the irreducibles in $H(S_n,\h)\rs^q$. Finally for $M\in H(S_n,\h)\rs^q$, we have $eL(M)\cong\sigma^*e'L'F(M)$, 
so claim 3 implies
\[
	\cO(\hr')\otimes_{\C[\h]}M=\cO(\hr')\otimes_{\cO(U)}L(M)\cong\cO(\hr')\otimes_{\cO(U')}L'F(M)=\cO(\hr')\otimes_{\C[\h']}F(M)
\]
as $\Sqp$-equivariant $\D$-modules on $\hr'$. It remains to prove the three claims.

\textbf{Proof of claim 1:} Fix $M\in H(S_n,\h)\rs^q$. By assumption, $\cO(\hr')\otimes_{\C[\h]}M$ has regular singularities. Let $b\in\h'$ be a ``generic" point of $Z(\alpha_0)\cap\h'$, that is, a point 
whose stabiliser is $W'\cong S_{m+1}\times S_m^{q-1}$. By Proposition \ref{resmon}, it suffices to show that the $H(W',\h)$-module $N=Res_bM$ is such that $\cO(\hr')\otimes_{\C[\h]}N$ has trivial monodromy 
around $b$. This depends only on the action of $H(S_{m+1},\C^m)\subseteq H(W',\h)$, so we may suppose $n=m+1$. But $N$ has minimal support by Theorem \ref{resind}(2). The proof of Corollary 4.7 of 
\cite{BE} shows that the only irreducible in $H_c(S_{m+1},\C^m)\rs$ with minimal support is $L(\C)$, where $\C$ denotes the trivial representation of $S_{m+1}$. Moreover Lemma 2.9 of \cite{GGOR} shows that 
${\rm Ext}^1(L(\C),L(\C))=0$, so $N$ is a direct sum of copies of $L(\C)$. Thus $\cO(\hr')\otimes_{\C[\h]}N$ is a free module with trivial connection, and in particular has trivial monodromy.

\textbf{Proof of claim 2:} Suppose $N\in H(\Sqp,U')\coh$ has the given properties. Let $U''\subseteq\h'$ denote the subset of points with trivial stabiliser in $\Sqp$, so that $\hr'=U'\cap U''$. Then $Sh(N)|_{\hr'}$ 
is an $\cO$-coherent $\D$-module on $\hr'$ with regular singularities and trivial monodromy around $Z(\pi(\alpha_0))$, so it extends to an $\cO$-coherent $\D$-module $\cN''$ on $U''$ with regular singularities. We 
may glue $Sh(N)$ and $\cN''$ to obtain a coherent sheaf $\cN_1$ on $U'\cup U''$. Note that the complement of $U'\cup U''$ in $\h'$ has codimension 2, so $\Gamma(U'\cup U'',\cO_{\h'})=\C[\h']$. Also $\cN''$ is torsion free since 
it is an $\cO$-coherent $\D$-module. Therefore $\cN_1$ is torsion free, so Lemma \ref{cohnonaff} shows that
\[
	N'=\Gamma(U'\cup U'',\cN_1)
\]
is finitely generated over $\C[\h']$. Also $H(\Sqp,U'')\cong D(U'')$ acts on $\cN''$ consistently with the action of $H(\Sqp,U')$ on $N$, so we obtain an action of $H(\Sqp,\h')$ on $N'$. Finally since 
$\cN'|_{U''}=\cN''$ has regular singularities, we have $N'\in H(\Sqp,\h')\rs$ and $L'(N')=\cO(U')\otimes_{\C[\h']}N'=N$.

\textbf{Proof of claim 3:} We have a natural $\C[\Sqp]\ltimes D(\hr')$ action on $\cO(\hr')\otimes_{\cO(U')}H(\Sqp,U')$ given by
\[
	w(f\otimes a)=\act[w]f\otimes wa
\]
and
\[
	\nabla_y(f\otimes a)=y(f)\otimes a+f\otimes ya-\sum_{s\in S'}c_s\langle y,\alpha_s\rangle\frac{f}{\alpha_s}\otimes(s-1)a,
\]
where $S'\subseteq\Sqp$ is the set of reflections. For $N\in H(\Sqp,U')\coh$, the natural isomorphism
\[
	\cO(\hr')\otimes_{\cO(U')}H(\Sqp,U')\otimes_{H(\Sqp,U')}N\isom\cO(\hr')\otimes_{\cO(U')}N
\]
preserves the equivariant $D(\hr')$-module structures by Proposition \ref{flatconn}. By Lemma \ref{HeH}, we may write the above as
\[
	\cO(\hr')\otimes_{\cO(U')}N\cong L'\otimes_{e'H(\Sqp,U')e'}e'N,
\]
where $L'$ is the $(\C[\Sqp]\ltimes D(\hr'),e'H(\Sqp,U')e')$-bimodule
\[
	L'=\cO(\hr')\otimes_{\cO(U')}H(\Sqp,U')e'.
\]
Since $H(\Sqp,U')=\cO(U')\C[(\h')^*]\C[\Sqp]$, we have
\begin{eqnarray*}
	L'&=&\cO(\hr')\otimes_{\cO(U')}H(\Sqp,U')e'H(\Sqp,U')e'\\
		&=&\cO(\hr')\otimes_{\cO(U')}\C[(\h')^*]e'H(\Sqp,U')e'\\
		&=&\bigcup_{d\geq0}\cO(\hr')\otimes_{\cO(U')}\C[(\h')^*]^{\leq d}e'H(\Sqp,U')e'.
\end{eqnarray*}
But
\[
	f\otimes ya=\nabla_y(f\otimes a)-y(f)\otimes a+\sum_{s\in S'}c_s\langle y,\alpha_s\rangle\frac{f}{\alpha_s}\otimes(s-1)a,
\]
so
\begin{eqnarray*}
	\cO(\hr')\otimes_{\cO(U')}\C[(\h')^*]^{\leq d+1}e'H(\Sqp,U')e'\hspace{-40mm}\\
		&\subseteq&D(\hr')\cO(\hr')\otimes_{\cO(U')}\C[(\h')^*]^{\leq d}e'H(\Sqp,U')e'.
\end{eqnarray*}
Since $\C[(\h')^*]^{\leq0}=\C$, it follows that $L'$ is generated as a bimodule by $1\otimes e'$. Now $H(\Sqp,U')$ acts faithfully on $\cO(U')$, so we have an injection
\[
	H(\Sqp,U')e'\hookrightarrow\Hom_\C(e'\cO(U'),\cO(U')).
\]
Note that $e'\cO(U')=\cO(U')^{\Sqp}=B$. Also $\cO(\hr')$ is flat over $\cO(U')$, so we obtain an inclusion
\[
	i:L'\hookrightarrow\Hom_\C(B,\cO(\hr')).
\]
Now $\C[\Sqp]\ltimes D(\hr')$ and $e'H(\Sqp,U')e'$ act on $\Hom_\C(B,\cO(\hr'))$ from the left and right, respectively, via their inclusions in $\End_\C(\cO(\hr'))$ and $\End_\C(B)$, and 
$i$ is a homomorphism of bimodules. Finally $i(1\otimes e')$ is the natural map $\phi:B\rightarrow\cO(\hr')$, so $L'$ is the sub-bimodule of $\Hom_\C(B,\cO(\hr'))$ generated by $\phi$.

Similarly $\cO(\hr')\otimes_{\cO(U)}H(S_n,U)$ admits an action of $\C[\Sqp]\ltimes D(\hr')$ given by
\[
	w(f\otimes a)=\act[w]f\otimes wa
\]
and
\[
	\nabla_y(f\otimes a)=y(f)\otimes a+f\otimes ya-\frac{1}{m}\sum_{s\in S\setminus S_m^q}\langle y,\alpha_s\rangle\frac{f}{\alpha_s}\otimes(s-1)a,
\]
where $S\subseteq S_n$ is the set of reflections. We are now identifying $\Sqp$ with a subgroup of $S_n$ as discussed before the theorem. Again for $M\in H(S_n,U)\coh$, we have a 
$\C[\Sqp]\ltimes D(\hr')$-module isomorphism
\[
	\cO(\hr')\otimes_{\cO(U)}M\cong\cO(\hr')\otimes_{\cO(U)}H(S_n,U)e\otimes_{eH(S_n,U)e}eM.
\]
Therefore if $eM\cong\sigma^*e'N$, then
\[
	\cO(\hr')\otimes_{\cO(U)}M\cong L\otimes_{e'H(\Sqp,U')e'}e'N
\]
where $L$ is the $(\C[\Sqp]\ltimes D(\hr'),e'H(\Sqp,U')e')$-bimodule
\[
	L=\cO(\hr')\otimes_{\cO(U)}H(S_n,U)e\otimes_{eH(S_n,U)e}e'H(\Sqp,U')e'.
\]
The exact sequence
\[
	eH(S_n,U)I_qH(S_n,U)e\hookrightarrow eH(S_n,U)e\twoheadrightarrow e'H(\Sqp,U')e'
\]
gives rise to a right exact sequence of right $eH(S_n,U)e$-modules
\[
	\cO(\hr')\otimes_{\cO(U)}H(S_n,U)I_qH(S_n,U)e
		\rightarrow\cO(\hr')\otimes_{\cO(U)}H(S_n,U)e\stackrel{\rho}{\twoheadrightarrow}L,
\]
since $H(S_n,U)eH(S_n,U)=H(S_n,U)$. As above, we have
\[
	\cO(\hr')\otimes_{\cO(U)}H(S_n,U)e=\bigcup_{d\geq0}\cO(\hr')\otimes_{\cO(U)}\C[\h^*]^{\leq d}eH(S_n,U)e.
\]
Now $\h$ is spanned by $\h'$ and $\ch x_i-\ch x_j$, for all $i\neq j$ such that $x_i-x_j$ vanishes on $\h'$. Fix such $i$ and $j$, and let $g\in\C[\h]$ vanish on all components of $X_q$ except $\h'$, such that $g$ is nonzero on $\hr'$. 
Then $\act[s_{ij}]f-f$ is divisible by $x_i-x_j$ for any $f\in\C[\h]$, so $g(s_{ij}-1)\in H(S_n,\h)$ sends $\C[\h]$ into $I_q$. But the annihilator of $\C[\h]/I_q$ in $H(S_n,\h)$ is $H(S_n,\h)I_qH(S_n,\h)$, so
\[
	H(S_n,U)I_qH(S_n,U)\ni[\ch x_i,g(s_{ij}-1)]=[\ch x_i,g](s_{ij}-1)+g(\ch x_i-\ch x_j)s_{ij}.
\]
Note that $[\ch x_i,g](s_{ij}-1)\in\C[\h]\C[S_n]$. But $g$ is invertible in $\cO(\hr')$, so
\[
	\ch x_i-\ch x_j\in\cO(\hr')\otimes_{\C[\h]}H(S_n,U)I_qH(S_n,U)+\cO(\hr')\C[S_n],
\]
whence
\begin{eqnarray*}
	\cO(\hr')\otimes_{\cO(U)}\C[\h^*]^{\leq d+1}eH(S_n,U)e
		&\subseteq&\cO(\hr')\otimes_{\cO(U)}\h'\C[\h^*]^{\leq d}eH(S_n,U)e\\
		&&{}+\cO(\hr')\otimes_{\cO(U)}\C[\h^*]^{\leq d}eH(S_n,U)e\\
		&&{}+\cO(\hr')\otimes_{\cO(U)}H(S_n,U)I_qH(S_n,U)e.
\end{eqnarray*}
Applying $\rho$, and using the same argument as above, we obtain
\[
	\rho\left(\cO(\hr')\otimes_{\cO(U)}\C[\h^*]^{\leq d+1}eH(S_n,U)e\right)
		\subseteq D(\hr')\rho\left(\cO(\hr')\otimes_{\cO(U)}\C[\h^*]^{\leq d}eH(S_n,U)e\right).
\]
Thus $L$ is generated as a bimodule by $\rho(1\otimes e)$. Finally $H(S_n,U)$ acts on $\cO(U)/I_q\cO(U)$ with annihilator $H(S_n,U)I_qH(S_n,U)$, so we have an inclusion
\[
	H(S_n,U)e/H(S_n,U)I_qH(S_n,U)e\hookrightarrow\Hom_\C(e(\cO(U)/I_q\cO(U)),\cO(U)/I_q\cO(U))
\]
of $(\cO(U)/I_q\cO(U),eH(S_n,U)e)$-bimodules. Note that $e(\cO(U)/I_q\cO(U))=B$. Since $\cO(\hr')$ is flat over $\cO(U)/I_q\cO(U)$, we obtain an inclusion
\[
	j:L\hookrightarrow\Hom_\C(B,\cO(\hr'))
\]
of $(\cO(\hr'),e'H(\Sqp,U')e')$-bimodules. Again $j$ preserves the left action of
\[
	\C[\Sqp]\ltimes D(\hr'),
\]
and sends $\rho(1\otimes e)$ to $\phi$, so $L$ is the sub-bimodule of $\Hom_\C(B,\cO(\hr'))$ generated by $\phi$. In particular, $L\cong L'$ as bimodules, so
\[
	\cO(\hr')\otimes_{\cO(U)}M\cong L\otimes_{e'H(\Sqp,U')e'}e'N\cong L'\otimes_{e'H(\Sqp,U')e'}e'N\cong\cO(\hr')\otimes_{\cO(U')}N
\]
as $\C[\Sqp]\ltimes D(\hr')$-modules, as required.
\end{proof}
\section{The Monodromy Functors for Type A}\label{Sec:mon}
We continue to study the case $W=S_n$, $\h=\C^n$ and $c=\frac{1}{m}$. Previously we studied the modules in $H_c\rs$ supported on $X_q$, where $q=\lfloor n/m\rfloor$. Now let $q$ be any integer satisfying 
$0\leq qm\leq n$, and let $H_c\rs^q$ denote the Serre subcategory of $H_c\rs$ consisting of all modules supported on $X_q$. Our goal in this section is to construct an equivalence of categories from $H_c\rs^q/H_c\rs^{q+1}$ to 
the category of finite dimensional representations of a certain Hecke algebra (where $H_c\rs^{\lfloor n/m\rfloor+1}$ is the subcategory containing only the zero module). From this we will deduce Theorem \ref{suppSnpart}.

Let us fix an integer $q$ as above, and put $p=n-qm$. Consider the natural copy $S_m^q\subseteq S_n$, and put $\h'=\h^{S_m^q}$ and $\hr'=\h^{S_m^q}_{\rm reg}$. As in the previous section, we may identify $\h'$ with 
$\Spec\C[z_1,\ldots,z_q,\,t_1,\ldots,t_p]$ via the map
\[
	\pi:\C[\h]\twoheadrightarrow\C[z_1,\ldots,z_q,\,t_1,\ldots,t_p],\hspace{10mm}
	\pi(x_i)=\begin{cases}
		z_{\lceil i/m\rceil}&\text{if }i\leq qm\\
		t_{i-qm}&\text{if }i>qm.
	\end{cases}
\]
Note that $\h'$ is one of the components of $X_q\subseteq\h$. By Proposition \ref{flatconn}, we have a functor
\begin{eqnarray*}
	Loc^q:H_c\rs^q&\rightarrow&\C[\Sqp]\ltimes D(\hr')\rs,\\
	M&\mapsto&\cO(\hr')\otimes_{\C[\h]}M,
\end{eqnarray*}
where $\C[\Sqp]\ltimes D(\hr')\rs$ denotes the category of coherent $\Sqp$-equivariant $D(\hr')$-modules with regular singularities. 
\begin{lemm}\label{locqexact}
The functor $Loc^q$ is exact, and induces a faithful functor
\[
	H_c\rs^q/H_c\rs^{q+1}\rightarrow\C[\Sqp]\ltimes D(\hr')\rs,
\]
which we also denote by $Loc^q$.
\end{lemm}
\begin{proof}
Any $M\in H_c\rs^q$ is annihilated by some power of $I_q$. By Theorem \ref{twosided}, $I_qM=0$. Let $b$ be any point in $\hr'$, and let $U\subseteq\h$ denote the affine Zariski open set
\[
	U=\{x\in\h\mid x_i\neq x_j\text{ whenever }b_i\neq b_j\}.
\]
Then $U\cap\h'=U\cap X_q=\hr'$, so $I_q\cO(U)\rightarrow\cO(U)\rightarrow\cO(\hr')$ is right exact. Thus 
$\cO(U)\otimes_{\C[\h]}M\rightarrow\cO(\hr')\otimes_{\C[\h]}M$ is an isomorphism. Since $M\mapsto\cO(U)\otimes_{\C[\h]}M$ is exact, so is $Loc^q$. Clearly the objects killed by $Loc^q$ are exactly those in $H_c\rs^{q+1}$. The 
result now follows from general categorical considerations.
\end{proof}
We now want to determine the ``image" of $Loc^q$. The Riemann-Hilbert correspondence \cite{Deligne} gives an equivalence of categories
\begin{eqnarray*}
	\C[\Sqp]\ltimes D(\hr')\rs
		&\isom&D(\hr'/\Sqp)\rs\\
		&\isom&\pi_1(\hr'/\Sqp){\rm-mod}_{\rm fd},
\end{eqnarray*}
where the latter denotes the category of finite dimensional representations of the group $\pi_1(\hr'/\Sqp)$ over $\C$. Let $\phi$ denote the composite functor. Under the identification 
$\h'=\Spec\C[z_1,\ldots,z_q,\,t_1,\ldots,t_p]\cong\C^{q+p}$, the open subset $\hr'$ consists of points with all coordinates distinct. By the correspondence between covering spaces of $\hr'/S_{q+p}$ and 
subgroups of $\pi_1(\hr'/S_{q+p})$, we have a homomorphism
\[
	\mu:\pi_1(\hr'/S_{q+p})\twoheadrightarrow S_{q+p},
\]
and $\pi_1(\hr'/S_q\times S_p)=\mu^{-1}(S_q\times S_p)$. The fundamental group $\pi_1(\hr'/S_{q+p})$ is known as the \emph{braid group} $B_{q+p}$. We need to describe some explicit elements 
of $B_{q+p}$. Suppose $x_1,\,x_2,\ldots,x_{q+p}\in\C$ are the vertices of a convex $(q+p)$-gon in the complex plane, listed counterclockwise. We take the point
\[
	\bx=(x_1,x_2,\ldots,x_{q+p})\in\C^{q+p}\cong\h'
\]
as the basepoint of $\pi_1(\hr')$. Suppose $1\leq i,j\leq q+p$ and $i\neq j$. Pick $\epsilon>0$ such that $2\epsilon|x_i-x_j|<|x_i+x_j-2x_k|$ for all $k$. Let $\gamma_{ij}:[0,1]\rightarrow\hr'$ be the path
\[
	\gamma_{ij}(t)=\begin{cases}
		ts_{ij}\bx+(1-t)\bx&\text{if }|t-\frac{1}{2}|>\epsilon\\
		\frac{1}{2}(s_{ij}\bx+\bx)+i\epsilon e^{\frac{1}{2\epsilon}i\pi(t-\frac{1}{2})}(\bx-s_{ij}\bx)&\text{otherwise}.
	\end{cases}
\]
from $\bx$ to $s_{ij}\bx$. Geometrically, as $\gamma_{ij}$ is traversed, $x_i$ and $x_j$ switch positions linearly, traversing small semicircles counterclockwise around $\frac{1}{2}(x_i+x_j)$ to avoid 
intersecting. Pushing this path down to $\hr'/S_{q+p}$ gives an element $S_{ij}\in B_{q+p}$ such that $\mu(S_{ij})=s_{ij}$. Note that there is a unique component $Z$ of $\h'\setminus\hr'$ passing through 
$\frac{1}{2}(s_{ij}\bx+\bx)$, and for $M\in\C[\Sqp]\ltimes D(\hr')\rs$, the action of $S_{ij}$ on $\phi(M)$ is conjugate to the monodromy around $Z$ of the induced local system on $\hr'/\Sqp$.
\begin{lemm}
We have a surjection $\pi_1(\hr'/\Sqp)\twoheadrightarrow B_q\times B_p$ whose kernel is generated by $\{S_{ij}^2\mid i\leq q,\,j>q\}$.
\end{lemm}
\begin{proof}
Let $U_q\subseteq\C^q$ and $U_p\subseteq\C^p$ denote the open subsets on which all coordinates are distinct. We have a natural continuous map $\hr'\rightarrow U_q\times U_p$, and the natural action of $\Sqp$ on 
the latter space makes this map equivariant. We may therefore identify $\hr'/\Sqp$ with an open subset of $(U_q/S_q)\times(U_p/S_p)$, and the complement $Z$ is the image of
\[
	\{\bx\in U_q\times U_p\subseteq\C^{q+p}\mid x_i=x_j\text{ for some }i\leq q,\,j>q\}.
\]
Note that $(U_q/S_q)\times(U_p/S_p)$ is a smooth complex variety, and $Z$ is an irreducible divisor. Moreover $Z$ is invariant under the action of $\R^*$, and is therefore contractible. Van Kampen's Theorem now gives a surjection
\[
	\pi_1(\hr'/\Sqp)\twoheadrightarrow\pi_1((U_q/S_q)\times(U_p/S_p))=B_q\times B_p,
\]
whose kernel is generated by a loop around $Z$. That is, the kernel is generated by $S_{ij}^2$ for any $i\leq q$ and $j>q$.
\end{proof}
The group $B_q$ has a standard presentation; namely it is generated by $\{S_{i,i+1}\mid1\leq i<q\}$ subject to the braid relations
\begin{eqnarray*}
	S_{i,i+1}S_{i+1,i+2}S_{i,i+1}&=&S_{i+1,i+2}S_{i,i+1}S_{i+1,i+2},\\
	S_{i,i+1}S_{j,j+1}&=&S_{j,j+1}S_{i,i+1}\text{if }|i-j|>1.
\end{eqnarray*}
The kernel of $B_q\rightarrow S_q$ is generated by $\{S_{ij}^2\}$. For any $\bq\in\C^*$, the \emph{Hecke algebra of type $A_{p-1}$} 
is the algebra
\[
	H_\bq(S_p)=\C[B_p]/\langle(S_{ij}-1)(S_{ij}+\bq)\rangle.
\]
We will take $\bq=e^{2\pi i/m}$. Combining the above, we have a surjection
\[
	\psi:\C[\pi_1(\hr'/\Sqp)]\twoheadrightarrow\C[S_q]\otimes_\C H_\bq(S_p)
\]
whose kernel is generated by
\begin{equation}\label{Arels}
	\{S_{ij}^2-1\mid i\leq q\text{ or }j\leq q\}\cup\{(S_{ij}-1)(S_{ij}+\bq)\mid i,j>q\}.
\end{equation}
This gives rise to a fully faithful embedding
\[
	\psi^*:\C[S_q]\otimes_\C H_\bq(S_p){\rm-mod}_{fd}\rightarrow\pi_1(\hr'/\Sqp){\rm-mod}_{fd},
\]
whose image is the full subcategory $\mathcal A$ of $\pi_1(\hr'/\Sqp){\rm-mod}_{fd}$ consisting of modules on which the relations (\ref{Arels}) vanish.
\begin{lemm}\label{locqmon}
The composite functor $\phi\,Loc^q$ sends each module into $\mathcal A$.
\end{lemm}
\begin{proof}
Consider any $M\in H_c\rs^q$. We must show that each of the relations (\ref{Arels}) vanish on $\phi\,Loc^q(M)$. That is, for each component $Z$ of $\h'\setminus\hr'$, we must show that the monodromy 
of the induced local system on $\hr'/\Sqp$ around $Z$ satisfies the appropriate equation. Let $b\in Z$ be a ``generic" point, that is, chosen from $Z$ so that its stabiliser $W'\subseteq S_n$ is 
minimal. By Proposition \ref{resmon}, it suffices to prove the appropriate equation for the monodromy of the local system corresponding to $N=Res_bM$. There are three possibilities, depending on which 
two coordinates we have set equal.

\textbf{Case 1}: $W'\cong S_m^{q-2}\times S_{2m}$. We are required to show that the monodromy of the local system on $\hr'$ (ignoring the equivariance structure) around $Z$ is trivial. The monodromy 
depends only on the action of $H_c(S_{2m},\C^{2m})\subseteq H_c(W',\h)$, so we may suppose $n=2m$. But $N$ has minimal support by Theorem \ref{resind}(2). By Theorem \ref{minsupp}, it suffices 
to consider the module $F(N)\in H_m(S_2,\C^2)\rs$. However, the latter category is semisimple with irreducibles $L(\tau_{(2)})$ and $L(\tau_{(1,1)})$, so it suffices to check that these modules 
give rise to local systems with trivial monodromy around the diagonal. This is clear from the description of the connection given in Proposition \ref{flatconn}.

\textbf{Case 2}: $W'\cong S_m^{q-1}\times S_{m+1}$. Again we must show that the monodromy of the local system on $\hr'$ around $Z$ is trivial. Now the monodromy depends only on the action of 
$H_c(S_{m+1},\C^{m+1})\subseteq H_c(W',\h)$, so we may suppose $n=m+1$. The result follows as in the proof of Claim 1 in Theorem \ref{minsupp}.

\textbf{Case 3}: $W'\cong S_m^q\times S_2$. This case is well known (see \cite{GGOR} Theorem 5.13) but we give the calculation here for convenience. Now the monodromy depends on the action of 
$H_c(S_2,\C^2)\subseteq H_c(W',\h)$, so we may suppose $n=2$.  Since $\h$ acts locally nilpotently on $N$, we can find a surjection
\[
	(H_c(S_2,\C^2)/H_c(S_2,\C^2)\h^k)^{\oplus l}\twoheadrightarrow N
\]
for some $k$ and $l$. We may replace $N$ by the former module. Then
\[
	V=(\C[S_2]\otimes_\C\C[\h^*]/\h^k)^{\oplus l}\subseteq N
\]
freely generates $N$ over $\C[\h]$ and is invariant under the actions of $\h$ and $S_2$. By Proposition \ref{flatconn}, the corresponding $D$-module on $\hr$ is freely generated over $\cO(\hr)$ by $V$, and has connection
\[
	\nabla_yv=yv-\frac{\langle y,\alpha_s\rangle}{\alpha_s}c(s-1)v
\]
for $v\in V$, where $S_2=\{1,s\}$. In particular, the residue of this connection acts as $0$ on the $1$-eigenspace of $s$ in $V$, and as $2c$ on the $-1$-eigenspace. Therefore after pushing down to $\h/S_2$, the monodromy 
acts as $1$ on the $1$-eigenspace and $-e^{2\pi ic}=-\bq$ on the $-1$-eigenspace. This proves the required relation.
\end{proof}
\begin{prop}\label{LoqqA}
The functor $Loc^q$ induces an equivalence
\[
	H_c\rs^q/H_c\rs^{q+1}\cong\phi^{-1}\mathcal A.
\]
\end{prop}
\begin{proof}
By the above lemmas, $Loc^q$ induces a faithful functor
\[
	H_c\rs^q/H_c\rs^{q+1}\rightarrow\phi^{-1}\mathcal A,
\]
and it remains to show that this functor is full and essentially surjective. Both statements will follow from the existence of a functor $G:\phi^{-1}\mathcal A\rightarrow H_c\rs^q$ such that $Loc^qG$ is naturally 
equivalent to the identity functor. This will take some work, so we proceed in a sequence of lemmas.

\begin{lemm}
Let $N\subseteq S_n$ denote the normaliser of the subgroup $S_m^q\subseteq S_n$, and choose a set $C$ of left coset representatives for $N$ in $S_n$. Let $\h^\perp$ denote the orthogonal complement to $\h'$ with respect 
to the natural $S_n$-invariant inner product on $\h=\C^n$. There is a functor
\[
	\bar G:\C[\Sqp]\ltimes D(\hr')\rs\rightarrow H_c(S_n,\h){\rm-mod}
\]
such that
\[
	\bar GM\cong\C[S_n]\otimes_{\C[N]}M=\bigoplus_{w\in C}wM
\]
as an $S_n$-module, and the actions of $x\in\h^*$ and $y\in\h$ are given by
\begin{eqnarray*}
	x(w\otimes v)&=&w\otimes\pi(x^w)v,\\
	y(w\otimes v)&=&w\otimes\nabla_{\rho(y^w)}v-\frac1m\hspace{-4mm}\sum_{i,j\atop\pi(x_i)\neq\pi(x_j)}\hspace{-4mm}\langle y^w,x_i-x_j\rangle(w+ws_{ij})\otimes\frac{1}{\pi(x_i-x_j)}v,
\end{eqnarray*}
where $\pi:\C[\h]\twoheadrightarrow\C[\h']$ is as above, and $\rho:\h\twoheadrightarrow\h'$ is the projection with kernel $\h^\perp$. Note that if $\pi(x_i)\neq\pi(x_j)$, then $\pi(x_i-x_j)$ is invertible in $\cO(\hr')$.
\end{lemm}
\begin{proof}
Under the natural identification $\h'\times\h^\perp\cong\h$, the set $\hr'\times\h^\perp$ is identified with an open subset $U\subseteq\h$. Let
\[
	\bar U=S_n\times_N U=\coprod_{w\in C}wU.
\]
We have a natural \'etale morphism $\bar U\rightarrow\h$, so by Propositions \ref{local}, \ref{formal} and \ref{linearH} we have algebra homomophisms
\[
	H_c(S_n,\h)\rightarrow H_c(S_n,\bar U)\rightarrow\Mat_{[S_n:N]}(\C[N]\ltimes_{\C[S_m^q]}H_c(S_m^q,\Spf\hO_{U,\hr'})),
\]
But
\[
	H_c(S_m^q,\Spf\hO_{U,\hr'})=\hO_{U,\hr'}\otimes_{\cO(U)}H_c(S_m^q,U)
		=\left(\lim_{\leftarrow}\cO(U)/I^k\right)\otimes_{\cO(U)}H_c(S_m^q,U),
\]
where $I\subseteq\cO(U)$ is the kernel of $\cO(U)\rightarrow\cO(\hr')$. Since $U\cong\hr'\times\h^\perp$, we have
\[
	H_c(S_m^q,U)=H_c(S_m^q,\h^\perp)\otimes_\C D(\hr').
\]
Moreover $\h^\perp\cong(\C^m/\C)^q$, so by Theorem \ref{finite} and Lemma \ref{nilp} we have an algebra homormorphism $H_c(S_m^q,\h^\perp)\rightarrow\C$ whose kernel contains $(\h^\perp)^*$. This induces a homomorphism
\[
	H_c(S_m^q,\h^\perp)\otimes_\C D(\hr')\rightarrow D(\hr')
\]
which kills $IH_c(S_m^q,\h^\perp)\otimes_\C D(\hr')$. This therefore factors through the completion, and we obtain
\[
	\C[N]\ltimes_{\C[S_m^q]}H_c(S_m^q,\hO_{U,\hr'})\rightarrow\C[N]\ltimes_{\C[S_m^q]}D(\hr')=\C[N/S_m^q]\ltimes D(\hr').
\]
Since $N/S_m^q\cong\Sqp$, we obtain functors
\begin{eqnarray*}
	\C[\Sqp]\ltimes D(\hr')\rs
		&\rightarrow&\C[N]\ltimes_{\C[S_m^q]}H_c(S_m^q,\hO_{U,\hr'}){\rm-mod}\\
		&\cong&\Mat_{[S_n:N]}(\C[N]\ltimes_{\C[S_m^q]}H_c(S_m^q,\hO_{U,\hr'})){\rm-mod}\\
		&\rightarrow&H_c(S_n,\h){\rm-mod}.
\end{eqnarray*}
Let $\bar G$ be the composite functor. Following through the construction in Proposition \ref{linearH}, we see that $\bar GM$ is exactly as descibed in the statement.
\end{proof}
Now let
\[
	\ch z_i=m\rho(\ch x_{mi})=\sum_{j=0}^{m-1}\ch x_{mi-j},\hspace{10mm}
	\ch t_a=\rho(\ch x_{mq+a})=\ch x_{mq+a}.
\]
These form the basis of $\h'$ dual to the basis $\{z_i,t_a\}$ of $(\h')^*$. As in the previous section, denote
\[
	p_{\ch x}(2)=\sum_{i=1}^n(\ch x_i)^2\in H_c(S_n,\h),
\]
and similarly for $p_{\ch z}(2)$ and $p_{\ch t}(2)$.
\begin{lemm}
The element $\eu\in H_c(S_n,\h)$ acts locally finitely on $\bar GM$. Also $p_{\ch x}(2)$ acts on symmetric elements of $\bar GM$ via the formula
\[
	p_{\ch x}(2)\sum_{w\in S_n}w\otimes v=\sum_{w\in S_n}w\otimes\zeta v,
\]
where $\zeta\in D(\hr')$ is the differential operator
\[
	\zeta=\frac1mp_{\ch z}(2)+p_{\ch t}(2)-2\sum_{i\neq j}\frac{\ch z_i-\ch z_j}{z_i-z_j}
			-\frac2m\sum_{i,a}\frac{\ch z_i-m\ch t_a}{z_i-t_a}-\frac2m\sum_{a\neq b}\frac{\ch t_a-\ch t_b}{t_a-t_b}.
\]
\end{lemm}
\begin{proof}
Consider the Euler vector field
\[
	\xi=\sum_{i=1}^qz_i\ch z_i+\sum_{i=1}^rt_i\ch t_i\in D(\h').
\]
This acts locally finitely on $M$ by Lemma \ref{monodromic}. Calculating the action of the Euler element $\eu\in H_c$ on $v\in M\subseteq\bar GM$ gives
\[
	\eu\,v=\nabla_\xi v-\frac{n(n-1)}{2m}v.
\]
Since $\eu$ centralises $S_n$, we conclude that $\eu$ acts locally finitely on $\bar GM$. The second statement also follows by a straightforward calculation.
\end{proof}
Clearly the support of $\bar GM$ lies in $X_q$, and $\cO(\hr')\otimes_{\C[\h]}\bar GM\cong M$. If $\bar GM$ were finitely generated over $C[\h]$, this would be the required module in $H_c\rs^q$. Unfortunately 
it is too large. We will construct the required module $GM$ as a submodule of $\bar GM$. For each irreducible component $Z$ of $\h'\setminus\hr'$, let $x_Z\in(\h')^*$ be an element with kernel $Z$. Let
\[
	\alpha=\prod_Zx_Z^2\in\C[\h']^{\Sqp}.
\]
Note that the zero set of $\alpha$ is exactly $\h'\setminus\hr'$.
\begin{lemm}
For each component $Z$ of $\h'\setminus\hr'$, there is a subspace $D_ZM\subseteq M$ satisfying:
\begin{enumerate}
\item $D_ZM$ is functorial in $M$,
\item $D_ZM$ is preserved by the actions of $\zeta$ and $\pi(\C[\h]^{S_n})\subseteq\C[\h']$
\item
	If $U\subseteq\h'$ is a Zariski open subset such that $V\subseteq\cO(U\cap\hr')\otimes_{\cO(\hr')}M$ freely generates $\cO(U\cap\hr')\otimes_{\cO(\hr')}M$ over $\cO(U\cap\hr')$, then 
	setting
	\[
		U'=U\cap{\rm int}(\hr'\cup Z),
	\]
	we have $D_ZM\subseteq x_Z^{-K}\cO(U')V$ for some integer $K$.
\item
	If $M\in\phi^{-1}\mathcal A$ then $M^{\Sqp}=\C[\alpha^{-1}](D_ZM)^{\Sqp}$.
\end{enumerate}
\end{lemm}
\begin{proof}
Choose a generic point $b\in Z$ with stabiliser $W'\subseteq S_n$. Let $B\subseteq\h'$ 
be an open ball around $b$ which doesn't intersect the other components of $\h'\setminus\hr'$. Since $Z\subseteq\h'$ is a codimension 1 complex subspace, $\pi_1(B\cap\hr')\cong\Z$. Let $\bar B\twoheadrightarrow B\cap\hr'$ 
denote the universal cover, and let $\cO_B^{an}$ and $\cO_{\bar B}^{an}$ denote the rings of analytic functions on $B$ and $\bar B$ respectively. Now $\cO_{\bar B}^{an}\otimes_{\cO(\hr')}M$ is naturally an analytic $\cO$-coherent 
$\D$-module on $\bar B$. Since $\bar B$ is simply connected, we have
\[
	\cO_{\bar B}^{an}\otimes_\C M_{flat}\isom\cO_{\bar B}^{an}\otimes_{\cO(\hr')}M,
\]
where $M_{flat}\subseteq\cO_{\bar B}^{an}\otimes_{\cO(\hr')}M$ is the space of flat sections. Let $i:M\rightarrow\cO_{\bar B}^{an}\otimes_{\cO(\hr')}M$ denote the inclusion. Again we consider three cases, depending on which 
two coordinates are equal on $Z$.

\textbf{Case 1:} $W'\cong S_m^q\times S_2$, so $Z$ is the kernel of $x_Z=t_a-t_b\in(\h')^*$. Let $s\in\Sqp$ be the transposition switching $t_a$ with $t_b$, and let $\cO_B^{an,s}$ denote the subspace of $\cO_B^{an}$ fixed by $s$. 
We may think of $\cO_B^{an,s}$ as consisting of functions involving only even powers of $x_Z$. Let $\lambda=\frac2m+1$, and let
\[
	D_ZM=i^{-1}\left(\cO_B^{an,s}M_{flat}+x_Z^\lambda\cO_B^{an,s}M_{flat}\right),
\]
noting that $x_Z^\lambda$ is a well-defined function in $\cO_{\bar B}^{an}$. We first check that this is independent of the choice of $b$ and $B$. Suppose $b'$ and $B'$ are chosen to satisfy the same conditions. Since 
$\h^{W'}_{\rm reg}$ is connected, we may find a path joining $b$ to $b'$, and some tubular neighbourhood $T$ of this path in $\h'$ will contain $B$ and $B'$, such that the inclusions $B\cap\hr'\hookrightarrow T\cap\hr'$ 
and $B'\cap\hr'\hookrightarrow T\cap\hr'$ are homotopy equivalences. This allows us to identify $\bar B$ and $\bar B'$ with open subsets of the universal cover $\bar T$ of $T\cap\hr'$. Thus $i$ may be expressed as a composite
\[
	M\stackrel{j}{\rightarrow}\cO_{\bar T}^{an}\otimes_{\cO(\hr')}M\stackrel{k}{\rightarrow}\cO_{\bar B}^{an}\otimes_{\cO(\hr')}M.
\]
Now $\cO_T^{an,s}+x_Z^\lambda\cO_T^{an,s}$ is the inverse image of $\cO_B^{an,s}+x_Z^\lambda\cO_B^{an,s}$ under $\cO_{\bar T}^{an}\rightarrow\cO_{\bar B}^{an}$, so
\[
	\cO_T^{an,s}M_{flat}+x_Z^\lambda\cO_T^{an,s}M_{flat}=k^{-1}(\cO_B^{an,s}M_{flat}+x_Z^\lambda\cO_B^{an,s}M_{flat}).
\]
It follows that $D_ZM=j^{-1}(\cO_T^{an,s}M_{flat}+x_Z^\lambda\cO_T^{an,s}M_{flat})$, so using $b'$ and $B'$ instead of $b$ and $B$ produces the same subspace. Property (1) is clear.

Now $\pi:\C[\h]\rightarrow\C[\h']$ is equivariant with respect to the action of $\Sqp$, so
\[
	\pi(\C[\h]^{S_n})\subseteq\C[\h']^{\Sqp}\subseteq\cO_B^{an,s}.
\]
Therefore $D_ZM$ is preserved by the action of $\pi(\C[\h]^{S_n})$. To show it is preserved by $\zeta$, note that
\begin{eqnarray*}
	\zeta x_Z^\lambda
		&=&2\lambda(\lambda-1)x_Z^{\lambda-2}-\frac4m\lambda x_Z^{\lambda-2}
			-2\sum_i\left(\frac{1}{t_a-z_i}-\frac{1}{t_b-z_i}\right)\lambda x_Z^{\lambda-1}\\
		&&{}-\frac2m\sum_{c\neq a,b}\left(\frac{1}{t_a-t_c}-\frac{1}{t_b-t_c}\right)\lambda x_Z^{\lambda-1}\\
		&=&2\sum_i\frac{1}{(t_a-z_i)(t_b-z_i)}\lambda x_Z^\lambda+\frac2m\sum_{c\neq a,b}\frac{1}{(t_a-t_c)(t_b-t_c)}\lambda x_Z^\lambda\\
		&\in&x_Z^\lambda\cO_B^{an,s}.
\end{eqnarray*}
Also setting $\ch x_Z=\ch t_a-\ch t_b$, we may write
\[
	\zeta\in\frac12(\ch x_Z)^2-\frac2m\frac{\ch x_Z}{x_Z}+\sum_{y\in Z}\C y^2+\cO_B^{an}y,
\]
so for $f\in\cO_B^{an,s}$, we have
\[
	[\zeta,f]\in\ch x_Z(f)\ch x_Z-\frac2m\frac{\ch x_Z(f)}{x_Z}+\sum_{y\in Z}\C y(f)y+\cO_B^{an}y(f)+\cO_B^{an,s}.
\]
But $\ch x_Z(f)/x_Z\in\cO_B^{an,s}$ and $s$ fixes $\zeta$, so we conclude that
\[
	[\zeta,\cO_B^{an,s}]\subseteq\cO_B^{an,s}x_Z\ch x_Z+\cO_B^{an,s}Z+\cO_B^{an,s}.
\]
Hence
\begin{eqnarray*}
	\zeta(\cO_B^{an,s}M_{flat}+x_Z^\lambda\cO_B^{an,s}M_{flat})\hspace{-30mm}\\
		&\subseteq&[\zeta,\cO_B^{an,s}]M_{flat}+[\zeta,\cO_B^{an,s}]x_Z^\lambda M_{flat}+\cO_B^{an,s}\zeta(x_Z^\lambda)M_{flat}\\
		&\subseteq&\cO_B^{an,s}M_{flat}+x_Z^\lambda\cO_B^{an,s}M_{flat},
\end{eqnarray*}
proving property (2).

Now suppose that $U$ and $V$ are as in property (3). If $U\cap Z=\emptyset$ then $\cO(U\cap\hr')=\cO(U')$, and the property is clear. Otherwise we may suppose that $B\subseteq U$. Then
\[
	\cO_{\bar B}^{an}\otimes_\C M_{flat}=\cO_{\bar B}^{an}\otimes_{\cO(\hr')}M=\cO_{\bar B}^{an}\otimes_\C V,
\]
so $M_{flat}=XV$, where $X\in GL_{\cO_{\bar B}^{an}}(\cO_{\bar B}^{an}\otimes_\C V)$ satisfies
\begin{eqnarray*}
	X&\in&\sum_{\mu\in\Lambda,\atop l\in L}\End_\C(V)\otimes_\C\cO_B^{an}x_Z^\mu(\log x_Z)^l,\\
	X^{-1}&\in&\sum_{\mu\in\Lambda,\atop l\in L}\End_\C(V)\otimes_\C\cO_B^{an}x_Z^{-\mu}(\log x_Z)^l
\end{eqnarray*}
for some finite subsets $\Lambda\subseteq\C$ and $L\subseteq\Z_{\geq0}$. Thus
\begin{eqnarray*}
	M_{flat}&\subseteq&\sum_{\mu\in\Lambda,\atop l\in L}\cO_B^{an}x_Z^\mu(\log x_Z)^lV,\\
	V&\subseteq&\sum_{\mu\in\Lambda,\atop l\in L}\cO_B^{an}x_Z^{-\mu}(\log x_Z)^lM_{flat}.
\end{eqnarray*}
Thus
\begin{eqnarray*}
	D_ZM&\subseteq&(\cO(U\cap\hr')\otimes_\C V)\cap\left(\sum_{\mu\in\Lambda,\atop l\in L}(\cO_B^{an}x_Z^\mu(\log x_Z)^l+\cO_B^{an}x_Z^{\mu+\lambda}(\log x_Z)^l)V\right)\\
		&\subseteq&x_Z^{-K}\cO(U')V
\end{eqnarray*}
for some integer $K$, proving (3).

Proceeding with the above notation, suppose now that $M\in\phi^{-1}\mathcal A$. Let $S$ be the endomorphism of $\cO_{\bar B}^{an}\otimes_{\cO(\hr')}M$ given by
\[
	S(f\otimes v)=(\bar s^*f)\otimes sv,
\]
where $\bar s^*\in\End_\C(\cO_{\bar B}^{an})$ is induced by the automorphism $\bar s:\bar B\rightarrow\bar B$ obtained by lifting $s:B\rightarrow B$. That is, $\bar s^*$ sends $x_Z$ to 
$e^{-\pi i}x_Z$, and fixes any function killed by $\ch x_Z$. The restriction of $S$ to $M_{flat}$ is exactly the monodromy of the corresponding local system on $\hr'/\Sqp$ around $Z$, 
which is assumed to satisfy
\[
	(S|_{M_{flat}}-1)(S|_{M_{flat}}+\bq)=0.
\]
Thus $M_{flat}$ decomposes into eigenspaces $M_{flat}^1$ and $M_{flat}^{-\bq}$ for $S$. The above shows that
\begin{eqnarray*}
	M&\subseteq&\sum_{\mu\in\Lambda,\atop l\in L}\cO_B^{an}x_Z^{-\mu}(\log x_Z)^lM_{flat}\\
		&=&\sum_{\mu\in\Lambda,\atop l\in L}\cO_B^{an}x_Z^{-\mu}(\log x_Z)^lM_{flat}^1+\cO_B^{an}x_Z^{-\mu}(\log x_Z)^lM_{flat}^{-\bq}.
\end{eqnarray*}
But $-\bq=e^{\lambda\pi i}$, so the elements fixed by $s$ must be contained in
\[
	M^s\subseteq\cO_B^{an,s}[x_Z^{-2}]M_{flat}^1+\cO_B^{an,s}[x_Z^{-2}]x_Z^\lambda M_{flat}^{-\bq}.
\]
If $v\in M^{\Sqp}\subseteq M^s$, then for some $K>0$ we have
\[
	\alpha^Kv\in M\cap\left(\cO_B^{an,s}M_{flat}^1+\cO_B^{an,s}x_Z^\lambda M_{flat}^{-\bq}\right)
		\subseteq D_ZM.
\]
Since $\alpha$ is also $\Sqp$-invariant, we conclude that $M^{\Sqp}\subseteq\C[\alpha^{-1}](D_ZM)^{\Sqp}$, proving (4).

\textbf{Case 2:} $W'\cong S_m^{q-2}\times S_{2m}$, so $Z$ is the kernel of $x_Z=z_i-z_j\in(\h')^*$. Let $s\in\Sqp$ be the transposition switching $z_i$ with $z_j$, and again let $\cO_B^{an,s}$ denote the 
subspace of $\cO_B^{an}$ fixed by $s$. We now set $\lambda=2m+1$, and again define
\[
	D_ZM=i^{-1}\left(\cO_B^{an,s}M_{flat}+x_Z^\lambda\cO_B^{an,s}M_{flat}\right).
\]
The arguments proceed as in the previous case, the only difference being the following two calculations. We have
\begin{eqnarray*}
	\zeta x_Z^\lambda
		&=&\frac{2\lambda(\lambda-1)}{m}x_Z^{\lambda-2}-4\lambda x_Z^{\lambda-2}
			-2\sum_{k\neq i,j}\left(\frac{1}{z_i-z_k}-\frac{1}{z_j-z_k}\right)\lambda x_Z^{\lambda-1}\\
		&&{}-\frac2m\sum_a\left(\frac{1}{z_i-t_a}-\frac{1}{z_j-t_a}\right)\lambda x_Z^{\lambda-1}\\
		&=&2\sum_{k\neq i,j}\frac{1}{(z_i-z_k)(z_j-z_k)}\lambda x_Z^\lambda
			+\frac2m\sum_a\frac{1}{(z_i-t_a)(z_j-t_a)}\lambda x_Z^\lambda\\
		&\in&x_Z^\lambda\cO_B^{an,s}.
\end{eqnarray*}
Also, setting $\ch x_Z=\ch z_i-\ch z_j$, we have
\[
	\zeta\in\frac1{2m}(\ch x_Z)^2-2\frac{\ch x_Z}{x_Z}+\sum_{y\in Z}\C y^2+\cO_B^{an}y,
\]
so for $f\in\cO_B^{an,s}$, we have
\[
	[\zeta,f]\in\frac1m\ch x_Z(f)\ch x_Z-2\frac{\ch x_Z(f)}{x_Z}+\sum_{y\in Z}\C y(f)y+\cO_B^{an}y(f)+\cO_B^{an,s}.
\]
The properties follow as above.

\textbf{Case 3:} $W'\cong S_m^{q-1}\times S_{m+1}$, so $Z$ is the kernel of $x_Z=z_i-t_a$. Let $\ch x_Z=\ch z_i-m\ch t_a$, and define
\begin{eqnarray*}
	\cO_B^{an+}&=&\{f\in\cO_B^{an}\mid\ch x_Z(f)\in x_Z\cO_B^{an}\},\\
	D_ZM&=&i^{-1}(\cO_B^{an+}M_{flat}).
\end{eqnarray*}
Properties (1), (3) and (4) follow as in the previous two cases. Now suppose $f\in\C[\h]^{S_n}$, and pick $j$ and $k$ such that $\pi(x_j)=z_i$ and $\pi(x_k)=t_a$. Then $(\ch x_j-\ch x_k)f$ is 
antisymmetric under $s_{jk}$, so $(\ch x_j-\ch x_k)f\in(x_j-x_k)\C[\h]$. Using equation (\ref{partial}) from the proof of Proposition \ref{imsigma}, applying $\pi$ gives
\[
	\frac1m\ch x_Z\pi(f)\in x_Z\C[\h'].
\]
Therefore $\pi(\C[\h]^{S_n})\subseteq\cO_B^{an+}$, so $D_ZM$ is preserved by the action of $\pi(\C[\h]^{S_n})$. Finally note that, for $j\neq i$ we have
\begin{eqnarray*}
	\frac1{z_j-z_i}-\frac1{z_j-t_a}&=&\frac{x_Z}{(z_j-z_i)(z_j-t_a)}\in x_Z\cO_B^{an}\text{ and}\\
	\ch x_Z\left(\frac m{z_j-z_i}+\frac1{z_j-t_a}\right)
		&=&\frac m{(z_j-z_i)^2}-\frac m{(z_j-t_a)^2}\\
		&=&\frac{mx_Z(2z_j-z_i-t_a)}{(z_j-z_i)^2(z_j-t_a)^2}\in x_Z\cO_B^{an},\text{ whence}\\
	\frac m{z_j-z_i}+\frac1{z_j-t_a}&\in&\cO_B^{an+}.
\end{eqnarray*}
Thus
\begin{eqnarray*}
	-2\frac{\ch z_i-\ch z_j}{z_i-z_j}-\frac2m\frac{\ch z_j-m\ch t_a}{z_j-t_a}\hspace{-20mm}\\
		&=&\frac2{m+1}\left(\frac1{z_j-z_i}-\frac1{z_j-t_a}\right)(\ch z_i-m\ch t_a)\\
		&&{}+\frac2{m+1}\left(\frac m{z_j-z_i}+\frac1{z_j-t_a}\right)\left(\ch z_i+\ch t_a-(1+1/m)\ch z_j\right)\\
		&\in&x_Z\cO_B^{an}\ch x_Z+\cO_B^{an+}Z.
\end{eqnarray*}
Similarly for $b\neq a$, we have
\begin{eqnarray*}
	-\frac2m\frac{\ch z_i-m\ch t_b}{z_i-t_b}-\frac2m\frac{\ch t_a-\ch t_b}{t_a-t_b}\hspace{-20mm}\\
		&=&\frac2{m(m+1)}\left(\frac1{t_b-z_i}+\frac1{t_a-t_b}\right)(\ch z_i-m\ch t_a)\\
		&&{}+\frac2{m(m+1)}\left(\frac m{t_b-z_i}-\frac1{t_a-t_b}\right)\left(\ch z_i+\ch t_a-(m+1)\ch t_b\right)\\
		&\in&x_Z\cO_B^{an}\ch x_Z+\cO_B^{an+}Z.
\end{eqnarray*}
Finally
\[
	\frac1m(\ch z_i)^2+(\ch t_a)^2
		=\frac1{m(m+1)}(\ch x_Z)^2+\frac1{m+1}(\ch z_i+\ch t_a)^2,
\]
so
\[
	\zeta\in\frac1{m(m+1)}(\ch x_Z)^2-\frac2m\frac{\ch x_Z}{x_Z}+x_Z\cO_B^{an}\ch x_Z+\sum_{y\in Z}\C y^2+\cO_B^{an+}y.
\]
Now suppose $f\in\cO_B^{an+}$, so $\ch x_Z(f)/x_Z\in\cO_B^{an}$. We have $[\ch x_Z,x_Z]=(m+1)$, so
\begin{eqnarray*}
	\ch x_Z\zeta(f)
		&\in&\frac1{m(m+1)}\left((\ch x_Z)^2x_Z-2(m+1)\ch x_Z\right)\frac{\ch x_Z(f)}{x_Z}\\
		&&{}+\ch x_Zx_Z\cO_B^{an}\ch x_Z(f)+\sum_{y\in Z}\C\ch x_Zy^2(f)+\ch x_Z\cO_B^{an+}y(f)\\
		&\subseteq&\frac1{m(m+1)}x_Z(\ch x_Z)^2\left(\frac{\ch x_Z(f)}{x_Z}\right)\\
		&&{}+\ch x_Z\left(x_Z^2\cO_B^{an}\right)+\sum_{y\in Z}\C y^2\ch x_Z(f)+x_Z\cO_B^{an}y(f)+\cO_B^{an+}y\ch x_Z(f)\\
		&\subseteq&x_Z\cO_B^{an}+\sum_{y\in Z}\C y^2\left(x_Z\cO_B^{an}\right)+\cO_B^{an+}y\left(x_Z\cO_B^{an}\right)\\
		&\subseteq&x_Z\cO_B^{an}.
\end{eqnarray*}
Thus $\zeta(f)\in\cO_B^{an+}$, proving property (2).
\end{proof}
\begin{lemm}
Consider the intersection
\[
	DM=\bigcap_ZD_ZM
\]
over all components of $\h'\setminus\hr'$. This subspace has the following properties:
\begin{enumerate}
\item $DM$ is functorial in $M$.
\item $DM$ is preserved by the actions of $\zeta$ and $\pi(\C[\h]^{S_n})\subseteq\C[\h']$.
\item $DM$ is finitely generated over $\pi(\C[\h]^{S_n})$.
\item If $M\in\phi^{-1}\mathcal A$ then $M^{\Sqp}=\C[\alpha^{-1}](DM)^{\Sqp}$.
\end{enumerate}
\end{lemm}
\begin{proof}
Certainly (1), (2) and (4) follow immediately from the corresponding properties of $D_ZM$. To prove (3), let $\{U_i\}$ denote a Zariski open cover of $\hr'$ such that
\[
	\cO(U_i)\otimes_{\cO(\hr')}M\cong\cO(U_i)\otimes_\C V_i.
\]
We may suppose $U_i$ consists of points in $\hr'$ where $g_i\neq0$, for some $g_i\in\C[\h']$. Moreover we may suppose $g_i$ does not vanish on any component of $\h'\setminus\hr'$. Now 
$\C[\h']$ is a UFD, so we have the notion of the order of pole of any element of $\cO(U_i)=\C[\h'][\alpha^{-1},g_i^{-1}]$ along some component $Z$. Property (3) of $D_ZM$ states that the coefficients of any 
$v\in D_ZM$ relative to $V_i$ have poles along $Z$ of order at most $K$, for some integer $K$. Thus
\[
	DM\subseteq\cO(U_i')\alpha^{-K}V_i,
\]
where $U_i'=\Spec\C[\h'][g_i^{-1}]\subseteq\h'$. In particular, $\cO(U_i')DM$ is finitely generated over $\cO(U_i')$, so $\cO(U')DM$ is a coherent sheaf on $U'$, where $U'$ is the union of the $U_i'$. We have 
$\cO(U')DM\subseteq M$, and the latter is locally free, so $\cO(U')DM$ is torsion free. Also $\h'\setminus U'$ is contained in $\h'\setminus\hr'$, but doesn't contain any component of the latter, so it has 
codimension at least $2$. Therefore $\cO(U')=\C[\h']$, so Lemma \ref{cohnonaff} implies that $\C[\h']DM$ is finitely generated over $\C[\h']$. Finally $\C[\h']$ is finite over $\pi(\C[\h]^{S_n})$, and the 
latter ring is Noetherian, so (3) follows.
\end{proof}
We may now complete the proof of Proposition \ref{LoqqA}. Consider the subspaces
\begin{eqnarray*}
	EM&=&\left\{\left.\sum_{w\in S_n}w\otimes v\right|v\in(DM)^{\Sqp}\right\}\subseteq e\bar GM,\\
	GM&=&H_cEM,
\end{eqnarray*}
where, as usual, $e=\frac1{n!}\sum_{w\in S_n}w$. Property (2) of $DM$ implies that $EM$ is preserved by the actions of $p_{\ch x}(2)e\in H_c$ and $\C[\h]^{S_n}e\subseteq H_c$. These generate the subalgebra $eH_ce\subseteq H_c$, so 
$EM$ is an $eH_ce$-submodule of $e\bar GM$. Also $EM\cong(DM)^{\Sqp}$ as $\C[\h]^{S_n}$-modules, so $EM$ is finitely generated over $\C[\h]^{S_n}$. Since $\eu$ acts locally finitely on $\bar GM$, we conclude that $EM$ decomposes 
into finite dimensional generalised eigenspaces for $\eu$, with eigenvalues in $\Lambda+\Z_{\geq0}$ for some finite subset $\Lambda\subset\C$. The same is true of $GM$, since $H_c$ is finite over $eH_ce$ and $\ad\,\eu$ is locally finite 
on $H_c$. This ensures that $GM\in H_c\rs^q$. Moreover the composite
\[
	\eta:Loc^qGM=\cO(\hr')\otimes_{\C[\h]}GM\hookrightarrow\cO(\hr')\otimes_{\C[\h]}\bar GM\isom M
\]
is a homomorphism of $\C[\Sqp]\ltimes D(\hr')$-modules, by Proposition \ref{flatconn}. But $GM\supseteq EM$, so if $M\in\phi^{-1}\mathcal A$, then property (4) of $DM$ ensures that $\im(\eta)$ contains $M^{\Sqp}$. Theorem 
2.3 of \cite{Montgomery} shows that
\[
	e'=\frac1{q!p!}\sum_{w\in\Sqp}w
\]
generates $\C[\Sqp]\ltimes D(\hr')$ as a two-sided ideal, so $M^{\Sqp}=e'M$ generates $M$ over $\C[\Sqp]\ltimes D(\hr')$. Therefore $\eta$ is an isomorphism, proving the result.
\end{proof}
Although we have considered $c=\frac1m$ so far in this section, the following result will allow us to generalise to $c=\frac rm$. The first statement follows from Corollary 4.3 of \cite{BE} in the case $m=2$, and from 
Theorem 5.12 and Proposition 5.14 of \cite{Rouquier} when $m>2$. The second is Theorem 5.10 of \cite{Rouquier}.
\begin{thm}
Suppose $r>0$ is coprime with $m$. There is an equivalence of categories $H_{\frac rm}\rs\cong H_{\frac1m}\rs$ which identifies $H_{\frac rm}\rs^q$ with $H_{\frac1m}\rs^q$ and sends $L(\tau_\lambda)\in H_{\frac rm}\rs$ to 
$L(\tau_\lambda)\in H_{\frac1m}\rs$. There is a $\C$-algebra isomorphism $H_{e^{2\pi ir/m}}(S_p)\cong H_{e^{2\pi i/m}}(S_p)$ which sends $D_\nu\in H_{e^{2\pi ir/m}}(S_p){\rm-mod}_{fd}$ to 
$D_\nu\in H_{e^{2\pi i/m}}(S_p){\rm-mod}_{fd}$.
\end{thm}
We may now combine the above to prove our remaining main results.
\begin{proof}[Proof of Theorems \ref{rsqrsq+1} and \ref{suppSnpart}]
As in the proof of Theorem \ref{suppSnpartle}, it suffices to assume $c>0$. By the previous theorem, we may suppose $c=\frac1m$. We have seen that $\phi\,Loc^q$ and $\psi^*$ identify $H_c\rs^q/H_c\rs^{q+1}$ and 
$\C[S_q]\otimes_\C H_\bq(S_p){\rm-mod}_{\rm fd}$, respectively, with the full subcategory $\mathcal A\subseteq\pi_1(\hr'/\Sqp){\rm-mod}_{fd}$. This proves the first statement of Theorem \ref{rsqrsq+1}. Let
\[
	F_{n,q}:H_{\frac1m}(S_n,\C^n)\rs^q/H_{\frac1m}(S_n,\C^n)\rs^{q+1}\isom\C[S_q]\otimes_\C H_\bq(S_p){\rm-mod}_{\rm fd}
\]
denote the equivalence.

Next we prove by induction on $q$ that if $q_m(\lambda)=q$, then $\supp(L(\tau_\lambda))=X_q$. Suppose it holds for $q'<q$. If the support of $L(\tau_\lambda)$ is $X_q$, then $q_m(\lambda)\leq q$ by Theorem 
\ref{suppSnpartle}. However we cannot have $q_m(\lambda)<q$ by the inductive hypothesis. Therefore the irreducibles in $H_c\rs^q/H_c\rs^{q+1}$ are a subset of
\[
	\Omega=\{L(\tau_\lambda)\mid q_m(\lambda)=q\}.
\]
We have a bijection
\[
	\{\mu\vdash q\}\times\{\nu\vdash p\mid q_m(\nu)=0\}\rightarrow\{\lambda\vdash n\mid q_m(\lambda)=q\}
\]
given by $(\mu,\nu)\mapsto m\mu+\nu$. The irreducibles in $\C[S_q]{\rm-mod}_{\rm fd}$ are indexed by $\{\mu\vdash q\}$, and those in $H_\bq(S_p){\rm-mod}_{\rm fd}$ by $\{\nu\vdash p\mid q_m(\nu)=0\}$. Therefore 
$|\Omega|$ is exactly the number of irreducibles in $\C[S_q]\otimes_\C H_\bq(S_p){\rm-mod}_{\rm fd}$. Since the latter category is equivalent to $H_c\rs^q/H_c\rs^{q+1}$, it follows that each module in $\Omega$ must 
be an irreducible in $H_c\rs^q/H_c\rs^{q+1}$; that is, they must all be supported on $X_q$. This completes the induction.

Finally we must show that $F_{n,q}(L(\tau_{m\mu+\nu}))\cong\tau_\mu\otimes D_{\nu'}$, where $\mu\vdash q$ and $\nu\vdash p$ with $q_m(\nu)=0$. This is known when $q=0$ (see Corollary 4.7 of \cite{BE}).

Next suppose that $p=0$, so $n=qm$. We will prove the statement in this case by induction on $q$. Now $M(\tau_{(n)})$ is the polynomial representation, and $L(\tau_{(n)})=\cO(X_q)$. This induces the trivial local 
system on $\hr^{S_m^q}$, so $F_{n,q}(L(\tau_{(n)}))=\tau_{(q)}$. This proves the statement for $q\leq2$. Suppose $q>2$ and that the statement holds for $q-1$. Let $b\in\h$ be a point whose stabiliser in $S_n$ is 
$W'=S_{n-m}\times S_m$. Note that any minimally supported module $M$ in $H_c(W',\C^n)\rs$ is of the form $M\cong M'\otimes L(\tau_{(m)})$ for some $M'\in H_c(S_{n-m},\C^{n-m})\rs^{q-1}$; by abuse of notation, we'll write 
$F_{n-m,q-1}(M)$ to mean $F_{n-m,q-1}(M')$. By (4) and (5) of Theorem \ref{resind}, there is a filtration of $Res_bM(\tau_{m\mu})$ whose successive quotients are the Verma modules corresponding to the composition factors of 
$Res^{S_n}_{W'}\tau_{m\mu}$. We have a surjection
\[
	Res_bM(\tau_{m\mu})\twoheadrightarrow Res_bL(\tau_{m\mu}).
\]
By (2) of Theorem \ref{resind}, the latter module has minimal support, so it is semisimple by Theorem \ref{minsupp}. For $\alpha\vdash n-m$ and $\beta\vdash m$, the only irreducible quotient of $M(\tau_\alpha\otimes\tau_\beta)$ 
is $L(\tau_\alpha\otimes\tau_\beta)$, and this has minimal support only if $\beta=(m)$ and $\alpha=m\gamma$ for some $\gamma\vdash q-1$. Moreover by the Littlewood-Richardson rule, $\tau_{m\gamma}\otimes\tau_{(m)}$ is a 
composition factor of $Res^{S_n}_{W'}\tau_{m\mu}$ only if $\tau_\gamma$ is a composition factor of $Res^{S_q}_{S_{q-1}}\tau_\mu$, and in this case it has multiplicity one. Therefore $Res_bL(\tau_{m\mu})$ is a submodule of
\def\clap#1{\hbox to 0pt{\hss#1\hss}}
\def\mathclap{\mathpalette\mathclapinternal}
\def\mathclapinternal#1#2{\clap{$\mathsurround=0pt#1{#2}$}}
\[
	\bigoplus_{\mathclap{\gamma\vdash q-1\atop\Hom_{S_{q-1}}(\tau_\gamma,\tau_\mu)\neq0}}L(\tau_{m\gamma}\otimes\tau_{(m)}).
\]
By the proofs of Lemmas \ref{locqexact} and \ref{locqmon}, $L(\tau_{m\mu})$ is scheme-theoretically supported on $X_q$, and the induced local system on $\hr^{W''}$ has trivial monodromy 
around $Z(\alpha_s)$ for each $s\in S\setminus W'$ and each $W''\subseteq W'$ with $W''\cong S_m^q$. Therefore Proposition \ref{resmonglobal} implies
\begin{eqnarray*}
	Res^{S_q}_{S_{q-1}}F_{n,q}(L(\tau_{m\mu}))
		&\cong&F_{n-m,q-1}(Res_bL(\tau_{m\mu}))\\
		&\subseteq&\bigoplus_{\mathclap{\gamma\vdash q-1\atop\Hom_{S_{q-1}}(\tau_\gamma,\tau_\mu)\neq0}}F_{n-m,q-1}(L(\tau_{m\gamma}))\\
		&\cong&\bigoplus_{\mathclap{\gamma\vdash q-1\atop\Hom_{S_{q-1}}(\tau_\gamma,\tau_\mu)\neq0}}\tau_\gamma\hspace{10mm}\text{by the inductive hypothesis}\\
		&\cong&Res^{S_q}_{S_{q-1}}\tau_\mu.
\end{eqnarray*}
If $q>2$, then the irreducibles over $S_q$ have distinct restrictions to $S_{q-1}$. Therefore $F_{n,q}(L(\tau_{m\mu}))\cong\tau_\mu$ for each $\mu\vdash q$, as required.

Finally we deduce the general case from these two special cases by induction on $h(\tau_{m\mu+\nu})$. Suppose the statement holds for all $\lambda$ with $h(\tau_\lambda)>h(\tau_{m\mu+\nu})$. Let $b\in\h$ be a point whose 
stabiliser in $S_n$ is $W'=S_{qm}\times S_p$. By Lemma \ref{indSn}, we have a nonzero map $M(\tau_{m\mu+\nu})\rightarrow Ind_b(L(\tau_{m\mu})\otimes L(\tau_{\nu}))$. Let $N$ be the image of this map. Its composition factors 
are of the form $L(\tau_\lambda)$ where either $\lambda=m\mu+\nu$ or $h(\tau_\lambda)>h(\tau_{m\mu+\nu})$ and $q_m(\lambda)\geq q$. By (1) of Theorem \ref{resind}, there is a nonzero map
\[
	Res_bN\rightarrow L(\tau_{m\mu})\otimes L(\tau_{\nu}).
\]
Since the latter module is irreducible, it is a quotient of some $Res_bL(\tau_\lambda)$ where either $\lambda=m\mu+\nu$ or $h(\tau_\lambda)>h(\tau_{m\mu+\nu})$ and $q_m(\lambda)=q$. Again by Proposition \ref{resmonglobal}, we have
\[
	F_{n,q}(L(\tau_\lambda))\twoheadrightarrow F_{mq,q}(L(\tau_{m\mu}))\otimes F_{p,0}(L(\tau_{\nu}))
		\cong\tau_\mu\otimes D_{\nu'}
\]
by the $q=0$ and $p=0$ cases shown above. In fact the left hand side is irreducible since $F_{n,q}$ is an equivalence, so
\[
	F_{n,q}(L(\tau_\lambda))\cong\tau_\mu\otimes D_{\nu'}
\]
By the inductive hypothesis, we cannot have $h(\tau_\lambda)>h(\tau_{m\mu+\nu})$. Therefore $\lambda=m\mu+\nu$ and we are done.
\end{proof}
\begin{proof}[Proof of Corollary \ref{BOconj}]
Let us again denote by $\bp_n$ the number of partitions of $n$, and let $\bp_{n,m}=|\{\lambda\vdash n\mid q_m(\lambda)=0\}|$ denote the number of $m$-regular partitions of $n$. We have shown that the number of irreducibles in 
$H_{\frac rm}\rs(S_n,\C^n)$ whose support is $X_q$ is $\bp_q\bp_{n-qm,m}$. This is the coefficient of $s^nt^{qm}$ in the formal power series
\[
	N(s,t)=\sum_{p,q\geq0}\bp_q\bp_{p,m}s^{qm+p}t^{qm}.
\]
It is well known that
\[
	\sum_{n\geq0}\bp_nt^n=\prod_{n>0}\frac1{1-t^n}.
\]
Every partition $\lambda$ of $n$ can be written uniquely as $\lambda=m\mu+\nu$ where $\nu$ is $m$-regular. Therefore
\[
	\prod_{n>0}\frac1{1-t^n}=\sum_{n\geq0}\bp_nt^n
		=\sum_{p,q\geq0}\bp_q\bp_{p,m}t^{mq+p}
		=\left(\prod_{q>0}\frac1{1-t^{mq}}\right)\left(\sum_{p\geq0}\bp_{p,m}t^p\right),
\]
giving
\[
	\sum_{p\geq0}\bp_{p,m}t^p=\prod_{n>0\atop m\nmid n}\frac1{1-t^n}.
\]
Thus
\[
	N(s,t)=\left(\sum_{p\geq0}\bp_{p,m}s^p\right)\left(\sum_{q\geq0}\bp_q(st)^{qm}\right)
		=\left(\prod_{p>0\atop m\nmid p}\frac1{1-s^p}\right)\left(\prod_{q>0}\frac1{1-(st)^{qm}}\right).
\]
Now consider the operator
\[
	A_m=\sum_{i>0}\alpha_{-im}\alpha_{im}
\]
acting on Fock space $F$. The elements
\[
	\prod_{i>0}\alpha_{-i}^{\nu_i}+{\rm span}\{A\alpha_i\mid i>0\}\in F
\]
form a basis for $F$, where the $\nu_i$ are nonnegative integers with only finitely many nonzero. This element is an eigenvector for $A_m$ with eigenvalue
\[
	\sum_{i>0\atop m|i}i\nu_i.
\]
Therefore
\begin{eqnarray*}
	\tr_F(s^{A_1}t^{A_m})
		&=&\sum_\nu\left(\prod_{i>0}s^{i\nu_i}\right)\left(\prod_{i>0\atop m|i}t^{i\nu_i}\right)\\
		&=&\left(\prod_{i>0\atop m\nmid i}\sum_{\nu_i\geq0}s^{i\nu_i}\right)\left(\prod_{i>0\atop m|i}\sum_{\nu_i\geq0}(st)^{i\nu_i}\right)\\
		&=&\left(\prod_{i>0\atop m\nmid i}\frac1{1-s^i}\right)\left(\prod_{i>0}\frac1{1-(st)^{mi}}\right)\\
		&=&N(s,t),
\end{eqnarray*}
as required.
\end{proof}
\bibliographystyle{amsplain}
\providecommand{\bysame}{\leavevmode\hbox to3em{\hrulefill}\thinspace}
\providecommand{\MR}{\relax\ifhmode\unskip\space\fi MR }
\providecommand{\MRhref}[2]{%
  \href{http://www.ams.org/mathscinet-getitem?mr=#1}{#2}
}
\providecommand{\href}[2]{#2}

\end{document}